\documentclass{amsart}
\usepackage{amsmath}
\usepackage{amssymb}
\usepackage{amsthm}
\usepackage{amscd,amsbsy}
\newtheorem{theorem}{Theorem}[section]
\newtheorem{lemma}[theorem]{Lemma}
\newtheorem{proposition}[theorem]{Proposition}
\newtheorem{corollary}[theorem]{Corollary}

\theoremstyle{definition}
\newtheorem{definition}[theorem]{Definition}
\newtheorem{definitions}[theorem]{Definitions}
\newtheorem{example}[theorem]{Example}

\newtheorem{definitions and remarks}[theorem]{Definitions and Remarks}

\newtheorem{conjecture}[theorem]{Conjecture}

\theoremstyle{remark}
\newtheorem{remark}[theorem]{Remark}
\newtheorem{remarks}[theorem]{Remarks}

\numberwithin{equation}{section}

\newcommand{\exc}{\mathrm{exc}}

\newcommand{\rk}{\mathrm{rk}\,}
\newcommand{\Sing}{\mathrm{Sing}\,}

\newcommand{\supp}{\mathrm{supp}\,}
\newcommand{\cosupp}{\mathrm{cosupp}\,}
\newcommand{\Jac}{\mathrm{Jac}\,}

\newcommand{\ord}{\mathrm{ord}}

\newcommand{\al}{{\alpha}}
\newcommand{\be}{{\beta}}
\newcommand{\de}{{\delta}}
\newcommand{\ep}{{\epsilon}}

\newcommand{\g}{{\gamma}}
\newcommand{\Ga}{{\Gamma}}

\newcommand{\la}{{\lambda}}
\newcommand{\La}{{\Lambda}}
\newcommand{\Om}{{\Omega}}
\newcommand{\p}{{\partial}}
\newcommand{\s}{{\sigma}}
\newcommand{\Sig}{{\Sigma}}

\newcommand{\io}{{\iota}}
\newcommand{\om}{{\omega}}

\newcommand{\IN}{{\mathbb N}}
\newcommand{\IQ}{{\mathbb Q}}

\newcommand{\IC}{{\mathbb C}}

\newcommand{\cF}{{\mathcal F}}
\newcommand{\cG}{{\mathcal G}}
\newcommand{\cH}{{\mathcal H}}
\newcommand{\cI}{{\mathcal I}}
\newcommand{\cJ}{{\mathcal J}}
\newcommand{\cK}{{\mathcal K}}
\newcommand{\cM}{{\mathcal M}}

\newcommand{\cO}{{\mathcal O}}

\newcommand{\cR}{{\mathcal R}}

\newcommand{\um}{\underline{m}}
\newcommand{\fm}{{\mathfrak m}}

\newcommand{\br}{\pmb{r}}

\newcommand{\bx}{\pmb{x}}
\newcommand{\by}{\pmb{y}}
\newcommand{\bu}{\pmb{u}}
\newcommand{\obu}{\overline{\pmb{u}}}
\newcommand{\tbu}{\tilde{\pmb{u}}}
\newcommand{\bv}{\pmb{v}}
\newcommand{\bz}{\pmb{z}}
\newcommand{\obv}{\overline{\pmb{v}}}
\newcommand{\bw}{\pmb{w}}
\newcommand{\bal}{\pmb{\al}}
\newcommand{\bbe}{\pmb{\be}}
\newcommand{\bla}{\pmb{\la}}

\newcommand{\bde}{\pmb{\de}}
\newcommand{\bep}{\pmb{\ep}}
\newcommand{\bg}{\pmb{\g}}
\newcommand{\bxi}{\pmb{\xi}}
\newcommand{\bta}{\pmb{\eta}}
\newcommand{\bzero}{\pmb{0}}

\newcommand{\ou}{\overline{u}}
\newcommand{\ov}{\overline{v}}
\newcommand{\oz}{\overline{z}}

\newcommand{\ta}{{\tilde a}}
\newcommand{\tb}{{\tilde b}}

\newcommand{\tg}{{\tilde g}}
\newcommand{\tu}{{\tilde u}}
\newcommand{\tv}{{\tilde v}}
\newcommand{\tw}{{\tilde w}}
\newcommand{\tX}{{\widetilde X}}
\newcommand{\tD}{{\widetilde D}}
\newcommand{\tE}{{\widetilde E}}

\newcommand{\tQ}{{\widetilde Q}}
\newcommand{\tR}{{\widetilde R}}
\newcommand{\tT}{{\widetilde T}}
\newcommand{\tU}{{\widetilde U}}

\newcommand{\tal}{{\tilde \al}}

\newcommand{\tbal}{{\tilde \bal}}

\newcommand{\tbde}{{\tilde \bde}}
\newcommand{\tga}{{\tilde \g}}
\newcommand{\tde}{{\tilde \de}}
\newcommand{\ttau}{{\tilde \tau}}

\begin{document}
\title[Resolution of singularities of the cotangent sheaf]
{Resolution of singularities of the cotangent sheaf of a singular variety}

\author[A.~Belotto]{Andr\'e Belotto da Silva}
\author[E.~Bierstone]{Edward Bierstone}
\author[V.~Grandjean]{Vincent Grandjean}
\author[P.D.~Milman]{Pierre D. Milman}
\address{AB,\,EB,\,PM: University of Toronto, Department of Mathematics, 40 St. George Street,
Toronto, ON, Canada M5S 2E4}
\email[A.~Belotto]{andrebelotto@gmail.com}
\email[E.~Bierstone]{bierston@math.toronto.edu}
\email[P.~Milman]{milman@math.toronto.edu}
\address{VG: Universidade Federal do Cear\'a, Departamento de Matem\'atica, Campus do Pici,
Bloco 914, Cep.\,60455-760 Fortaleza-Ce, Brasil}
\email[V.~Grandjean]{vgrandjean@math.ufc.br}
\thanks{Research supported in part by NSERC grants OGP0009070 and OGP0008949}

\subjclass{Primary 14E05, 14E15, 32S45; Secondary 14C30, 14F43, 32C15, 32S20, 32S35}

\keywords{Resolution of singularities, monomialization, logarithmic differential forms, Fitting ideals, Hsiang-Pati coordinates}

\begin{abstract}
The main problem studied here is resolution of singularities of the cotangent sheaf of a complex- or 
real-analytic variety $X_0$ (or of an algebraic variety $X_0$ over a field of characteristic zero). Given
$X_0$, we ask whether there is a global resolution of
singularities $\s: X \to X_0$ such that the pulled-back cotangent sheaf of $X_0$
is generated by differential monomials in suitable coordinates at every point of $X$ (``Hsiang-Pati
coordinates''). Desingularization of the cotangent sheaf is equivalent to monomialization of 
Fitting ideals generated by minors of a given order of the logarithmic Jacobian matrix of $\s$.
We prove resolution of singularities of the cotangent sheaf in dimension up to three. It was previously known
for surfaces with isolated singularities (Hsiang-Pati 1985, Pardon-Stern 2001). Consequences
include monomialization of the induced Fubini-Study metric on the smooth part of a complex
projective variety $X_0$; there have been important applications of the latter to $L_2$-cohomology.
\end{abstract}

\maketitle
\setcounter{tocdepth}{1}
\tableofcontents

\section{Introduction}\label{sec:intro}
The subject of this article is resolution of singularities or monomialization of differential forms 
on an algebraic or analytic variety.
Let $X_0$ denote either an algebraic variety over a field of characteristic zero, or a complex- or real-analytic
variety. We assume that $X_0$ is reduced; i.e., that its structure sheaf has no nilpotents. Let $\Sing X_0$
denote the singular subset of $X_0$. Our main goal is to prove the following conjecture.

\begin{conjecture}[\emph{Resolution of singularities of the cotangent sheaf}]\label{conj:main}
There is a \emph{resolution of singularities} of $X_0$ (i.e., a proper birational or bimeromorphic morphism 
$\s: X \to X_0$ such that $X$ is smooth, $\s$ is an isomorphism over $X_0\backslash \Sing X_0$, and $\s^{-1}(\Sing X_0)$ is 
the support of a simple normal crossings divisor $E$ on $X$), such that the pulled-back cotangent sheaf of $X_0$
is locally generated by \emph{differential monomials}
\begin{equation}\label{eq:main}
d(\bu^{\bal_i}), \, i = 1,\ldots,s, \,\, \text{and } \, d(\bu^{\bbe_j}v_j), \, j=1,\ldots,n-s,
\end{equation}
where $n=\dim X_0$, $(\bu,\bv) = (u_1,\ldots,u_s,v_1,\ldots,v_{n-s})$ are local (analytic or \'etale)
coordinates on $X$, and
\begin{enumerate}
\item $\supp E = (u_1\cdots u_s = 0)$,
\item the multiindices $\bal_1,\ldots,\bal_s \in \IN^s$ are linearly independent over $\IQ$,
\item $\{\bal_i,\bbe_j\}$ is totally ordered (with respect to the componentwise partial ordering of $\IN^s$).
\end{enumerate}
\end{conjecture}

An important consequence, for example in the case that $X_0$ is a complex projective variety, is that
the pull-back to $X$ of 
the induced \emph{Fubini-Study metric} on $X_0\backslash \Sing X_0$ is locally quasi-isometric to
$$
\sum_{i=1}^s d(\bu^{\bal_i}) \otimes \overline{d(\bu^{\bal_i})}\, +\, 
\sum_{j=1}^{n-s} d(\bu^{\bbe_j}v_j) \otimes \overline{d(\bu^{\bbe_j}v_j)}.
$$

We will show that the problem of desingularization of the cotangent sheaf
can be reformulated
in terms of principalization of logarithmic Fitting ideal sheaves (an approach suggested already by
 \cite{PS}); see Section \ref{sec:fitHP} below. Given a resolution of singularities $\s: X \to X_0$, the \emph{logarithmic Fitting ideal} $\cF_k(\s)$ denotes
the sheaf of ideals of $\cO_X$ generated locally by the minors of order $n-k$ of the Jacobian matrix of $\s$
with respect to a logarithmic basis of 1-forms on $X$; see \S\ref{subsec:logdiff}.
\begin{theorem}\label{thm:fitHP}
Let $\s: (X,E) \to (X_0, \Sing X_0)$ be a resolution of singularities of $X_0$. Then the following
conditions are equivalent.
\begin{enumerate}
\item The logarithmic Fitting ideals $\cF_k(\s)$, $k=0,\ldots,n-1$, are all principal monomial ideals
(generated locally by monomials in components of the exceptional divisor).
\item The morphism $\s$ is a resolution of singularities of the cotangent sheaf of $X_0$, as in
Conjecture \ref{conj:main}.
\end{enumerate}
\end{theorem}

Conjecture \ref{conj:main} can be strengthened by asking that $\s$ be a composite of 
blowings-up with smooth admissible centres (\emph{admissible} means that each centre of
blowing-up has only normal crossings with the exceptional divisor; also see \S\ref{subsec:res}).
The following is our main result.

\begin{theorem}\label{thm:dim3}
Conjecture \ref{conj:main} (in the preceding stronger form) holds for varieties of dimension $\leq 3$.
\end{theorem}

The result was previously proved (at least locally) in the case of surfaces (2-dimensional varieties) with
isolated singularities by W.-C.~Hsiang and V.~Pati \cite[1985]{HP}, and a more conceptual proof
in this case  was given by W.~Pardon and M.~Stern \cite[2001]{PS}.
Our formulation of Conjecture \ref{conj:main} is due to B.~Youssin \cite[1998]{Y}. A system of coordinates
as in Conjecture \ref{conj:main} will be called \emph{Hsiang-Pati coordinates}.

One of the main interests of the Hsiang-Pati problem has been for applications to the $L_2$-cohomology
 of the smooth part of a singular variety, going back to the original ideas of Cheeger \cite{C1,C2}. Hsiang and
 Pati used their result to prove that the intersection cohomology (with the middle perversity) of a complex
 surface $X_0$ equals the $L_2$-cohomology of $X_0 \backslash \Sing X_0$ (Cheeger-Goresky-Macpherson
 conjecture \cite{CGM}). Several controversial articles on both the Hsiang-Pati problem and the 
 $L_2$-cohomology of singular varieties have perhaps discouraged 
 work on these questions; we hope that our results will lead to a renewal of interest.
 
 Our main conjecture \ref{conj:main} is closely related to the problem
 of \emph{monomialization} or \emph{toroidalization} of a morphism (see \S\ref{subsec:mon}),
 and our proof
 of Theorem \ref{thm:dim3} is strongly influenced by Cutkosky \cite{Cut2}; in particular, the invariant $\rho$ of
 Section \ref{sec:inv} below is introduced in the latter (but our article does not depend on \cite{Cut2}). 
 Problems involving techniques similar to those developed
 in this article are treated in \cite{Be1,Be2}.
 
 We are grateful to Franklin Vera Pacheco for several very helpful comments.
 
\subsection{Outline of the paper} 
The logarithmic Fitting ideals $\cF_k(\s)$
cannot be principalized by a simple application of resolution of singularities
because $\cF_k(\s)$ does not, in general, commute with pull-back (even up to multiplication by
a principal monomial ideal). We show, nevertheless, that standard desingularization techniques can be
used to principalize the Fitting ideal $\cF_0(\s)$ of highest order minors, as well as the Fitting ideals
of lower order minors. More precisely, we can reduce to the case that, if the logarithmic Jacobian matrix
has rank $r$ at $a \in X$, then $\cF_0(\s)$ as well as $\cF_{n-1}(\s)_a,\ldots, \cF_{n-r-1}(\s)_a$ are principal,
and the first $r+1$ components of $\s$ at $a$ (with respect to suitable local coordinates) are given by 
Hsiang-Pati monomials
$$
\s_1 = v_1,\quad \ldots,\quad \s_r = v_r,\quad \s_{r+1} = \bu^{\bal_1},
$$
(where $(\bu,\bv)$ are coordinates at $a$ in which $\supp E = (u_1\cdots u_s = 0)$; see \S\ref{sec:start}).

An immediate consequence is that Conjecture \ref{conj:main} holds in the case $\dim X_0 \leq 2$.
Moreover, to prove Theorem \ref{thm:dim3} (when $\dim X_0 = 3$),
it remains only to principalize $\cF_1(\s)$ at points of log rank $0$;
the image of such points in $X_0$ forms a discrete subset.

Principalization of $\cF_1(\s)$ is technically the most difficult part of the paper. We argue by induction on 
the maximal value of an
upper-semicontinuous local invariant $\rho$ of $\cF_1(\s)$. The invariant $\rho$ has possible values $0,1,\ldots,\infty$,
and $\rho(a) = 0$ if and only if $\cF_1(\s)_a$ is a principal monomial ideal (Section \ref{sec:inv}).
Our proof of principalization of $\cF_1(\s)$ has three main steps (cf. Section \ref{sec:outline}):

\smallskip\noindent
Step 1. Reduction to the case that $\rho(a) < \infty$, for all $a$. In this case, at a point $a$ of log rank $0$,
we can write the components of 
$\s$ (with respect to suitable local coordinates as in Conjecture \ref{conj:main}) as $\s_1 = \bu^{\bal}$ and
$\s_i = g_i(\bu) + \bu^{\bde}T_i$, $i > 1$, where $\bu^{\bal}$ divides each $\s_i$, every
$dg_i$ is in the submodule generated by $d(\bu^{\bal})$, and (in the case that $a$ is a 1- or 2-point) 
$T_2$ can be written essentially
in Weierstrass polynomial form with respect to a distinguished variable $v$ (Lemma \ref{lem:rho}. We say
that $a$ is an $s$-\emph{point} when $\supp E = (u_1\cdots u_s = 0)$ at $a$.)

\smallskip\noindent
Step 2. Reduction to prepared normal form (Lemma \ref{lem:prepnormal} and Section \ref{sec:prepnorm}). 
By further blowings-up, we reduce to the case that the coefficients
of the $T_i$ as expansions in $v$ are monomials (times units) in local coordinates as above,
and the zeroth coefficients (i.e., the coefficients
of $v^0$) are essentially in Hsiang-Pati form (as components of a morphism in dimension two).

\smallskip\noindent
Step 3. Decrease of $\rho$, by further blowings-up (Lemma \ref{lem:decrho} and Section \ref{sec:decrho}).

\subsection{Examples}\label{subsec:ex}
The following examples illustrate some of the challenges in the Hsiang-Pati problem.

\begin{example}\label{ex:1}
Let $\s = (\s_1,\s_2,\s_3)$ be the morphism given by
\begin{equation*}
\s_1 = u^2,\quad \s_2 = u^3(v^2 + uw),\quad \s_3 = u^4v,
\end{equation*}
where $E = (u=0)$. Then the log Fitting ideals $\cF_0(\s),\, \cF_2(\s)$ are principal, and
$\cF_1(\s) = u^5\cdot(u,v)$ at the 1-point $a=0$. This is essentially the simplest example of a morphism
written in prepared normal form (Lemma \ref{lem:prepnormal}), that does not yet satisfy the conditions of
Conjecture \ref{conj:main}. In this example, $\s$ can be reduced to Hsiang-Pati form (in fact, $\s$
can be monomialized; cf. Section \ref{subsec:mon} below) by two blowings-up.

Let $\tau$ denote the blowing-up with centre $C = (u=v=0)$ ($C$ is the locus of points
where the invariant $\rho =1$). Then $\tau$ can be covered by
two coordinate charts. Let $b \in \tau^{-1}(a)$. There are two possibilities:
\begin{enumerate}
\item $b$ belongs to the ``$u$-chart'', with coordinates $(x,\tv,\tw)$ in which $\tau$ is given by
$(u,v,w) = (x,x\tv,\tw)$. Then
$$
\s_1\circ\tau = x^2,\quad \s_2\circ\tau = x^4(x\tv + \tw),\quad \s_3\circ\tau = x^5\tv,
$$
and $\s\circ\tau$ is in Hsiang-Pati form in this chart; in fact, $\s_2\circ\tau = x^4w'$ after a 
coordinate change $w' = \tw + x\tv$, at any $b\in \tau^{-1}(a)$ in the chart.
\smallskip
\item $b$ is the origin of the $v$-chart, with coordinates $(x,y,\tw)$ in which $\tau$ is given by
$(u,v,w) = (xy,y,\tw)$. In this chart, $\s_3\circ\tau = y(y+x\tw)$; $b$ is a 2-point (with $\rho(b) = \infty$)
and one more blowing-up, with centre the 2-curve $(x=y=0)$, is needed to reduce to Hsiang-Pati form.
Note that this centre is globally defined in the source of $\tau$.
\end{enumerate}

We remark that Pati \cite{Pati} and Taalman \cite{Laura} claim that, in dimension three, $\cF_1(\s)$ is necessarily
principal at a 1-point, because ``since $u$ does not divide $R$ [where 
$R = T_2 = v^2 + uw$ in the example above], the 2-form $du \wedge dR$ is nowhere-vanishing [on $(u=0)$],
and thus $R$ is a coordinate independent of $u$'' \cite[p.\,258]{Laura} (cf. \cite[p.\,443]{Pati}).
The example above shows this is not true.
\end{example}

\begin{example}\label{ex:2}
Let $\s = (\s_1,\s_2,\s_3,\s_4,\s_5)$ be the morphism given by
\begin{equation*}
\s_1 = u^2,\quad \s_2 = u^3(v^3 + (y^2 + ux^2)uv + u^3w),\quad \s_3 = u^4v,\quad \s_4 = u^4y,\quad \s_5 = u^4x,
\end{equation*}where $E = (u=0)$. Then $\cF_0(\s)$ and $\cF_4(\s)$ are principal at $0$. 
It might appear reasonable to blow up $(x=y=0)$ to principalize
$\cF_3(\s)$, but this centre is not in $\supp E$.
 \end{example}
 
 \subsection{Monomialization of a morphism}\label{subsec:mon}
 Given $X_0$, we can ask whether there
 exists a resolution of singularities as in Conjecture \ref{conj:main} such that the components of $\s$ (with 
 respect to suitable local coordinates of a smooth local embedding variety of $X_0$) are themselves
 monomials (rather than only their differentials). This is not true, in general, because it would imply that
 $X_0$ locally has a toric structure.
 
 It is reasonable to ask, on the other hand, whether a morphism $\s: X \to Y$ can be monomialized by 
 blowings-up in both the source
 \emph{and the target} --- this is the problem of monomialization. In its simplest formulation,
 we can ask whether, after suitable global blowings-up in the source and target, a proper birational or
 bimeromorphic morphism $\s$ can be transformed to a morphism that can be expressed locally as
 $$
 x_i = \bu^{\bal_i},\quad y_j = \bu^{\bbe_j}(c_j + v_j),\quad z_k = w_k,
 $$
 with respect to coordinates $(\bu,\bv,\bw) = (u_1,\ldots,u_p, v_1,\ldots,v_q, w_1,\ldots,w_r)$ and
 $(\bx,\by,\bz) \allowbreak = (x_1,\ldots,x_p, y_1,\ldots,y_q, z_1,\ldots,z_r)$ in the source and target (respectively),
 where $\bu$ and $(\bx,\by)$ represent the exceptional divisors, the exponents $\bal_i$ are
 $\IQ$-linearly independent, and the $c_j$ are nonzero constants.
 
 Our approach to Conjecture \ref{conj:main} or Theorem \ref{thm:dim3} is a step in this direction. By blowings-up
 in the source, we aim to express the successive components of $\s$ as the Hsiang-Pati monomials
 $(\bu^{\bal_i}, \bu^{\bbe_j}v_j)$, in the ordering of Conjecture \ref{conj:main}(3), 
 plus additional terms whose differentials are in the
 submodule generated by \eqref{eq:main}; see \S\ref{subsec:fitHP} and Lemma \ref{lem:rho}. The passage
 from such a statement to monomialization of $\s$ by blowings-up in the source and target is an interesting
 problem that we plan to treat in a future article (cf. Cutkosky \cite[Sections\,18,\,19]{Cut1} for morphisms from 
 dimension three to two).
 
 The problem of monomialization or toroidalization of morphisms has an extensive literature (see, for example,
\cite{AKMW,Cut3} and references therein), though the only general results either are of a local nature or
 involve generically finite rather than birational or bimeromorphic modifications. Cutkosky has proved
 global monomialization for algebraic morphisms in dimension up to three \cite{Cut3,Cut2}. 
 The normal forms of \cite[Section\,3]{Cut2} cannot
 be obtained, however, as claimed in the proof of \cite[Lemma\,3.6]{Cut2} (see \cite{Cut4});
 we use different normal forms in our Lemma \ref{lem:prepnormal}.
 
 Cutkosky's arguments do not extend to analytic morphisms in an evident way
 because they involve local blowings-up that are globalized by algebraic techniques
 (e.g., Zariski closure, Bertini's theorem) that are not available in an analytic category.
 One of the main differences of our approach from that of \cite{Cut3,Cut2} is that we chose as
 centres of blowing up only subspaces that \emph{a priori} have global meaning (see Section \ref{sec:decrho}
 below).
 
 \section{Logarithmic Fitting ideals}\label{sec:logfit}
 Let $X_0$ denote either an algebraic variety over a field of characteristic zero, or a complex- or real-analytic
variety. Assume that $X_0$ is reduced. 

\subsection{Blowing up and resolution of singularities}\label{subsec:res}
A \emph{resolution of singularities} of $X_0$ is a proper birational or bimeromorphic morphism $\s: X \to X_0$ 
such that $X$ is smooth, $\s$ is an isomorphism over $X_0\backslash \Sing X_0$, and $\s^{-1}(\Sing X_0)$ is
the support of a simple normal crossings divisor $E$ on $X$ (the \emph{exceptional divisor}). See
\cite{BMinv,BMfunct}. We write 
$\s: (X,E) \to (X_0, \Sing X_0)$.

Given a smooth variety $X$ with a simple normal crossings divisor $E$, a blowing-up $\s: X' \to X$ is
called \emph{admissible} (for $E$) if the centre of $\s$ is smooth and has only normal crossings
with $E$. An admissible blowing-up is called \emph{combinatorial} if its centre is an intersection of
components of $E$. 

A singular variety $X_0$ locally admits an embedding $X_0|_U \hookrightarrow M_0$ in a smooth
variety $M_0$ ($U$ denotes an open subset of $X_0$). A divisor $E_0$ on $X_0$ has only 
\emph{normal crossings} if $E$ is the restriction
of ambient normal crossings divisors, for a suitable covering of $X_0$ by embeddings in smooth 
varieties. The notions of simple normal crossings divisor, admissible and combinatorial blowing-up
all make sense in the same way. 

If $X_0$ is an algebraic or compact analytic variety (with a simple normal crossings divisor $E_0$,
perhaps empty), then a resolution of singularities $(X,E) \to
(X_0, \Sing X_0)$ can be obtained as a composite of finitely many smooth admissible blowings-up.
In the case of a general analytic variety $X_0$, resolution of singularities can be obtained by a morphism
which can be realized as a composite of finitely many smooth admissible blowings-up over any relatively
compact open subset of $X_0$ (we will still say, somewhat loosely, that $\s$ is a ``composite of blowings-up'').

\subsection{Logarithmic differential forms}\label{subsec:logdiff}
Let $X$ denote a smooth variety with simple normal crossings divisor $E$,
and let $\Om_X^1$ denote the $\cO_X$-module of differential $1$-forms on $X$;
$\Om_X^1$ is locally free of rank $n = \dim X$.

Let $\Om_X^1(\log E)$ denote the $\cO_X$-module of logarithmic 1-forms on $X$: If $U$ is an
\'etale or analytic local coordinate chart of $X$ at a point $a$, with coordinates $(\bu,\bv) = (u_1,\ldots,u_s,v_1,\ldots,v_{n-s})$ such that $a=0$ and $(u_i=0)$, $i=1,\ldots,s$, are the components of $E$ in
$U$ (we say the coordinates $(\bu,\bv)$ are \emph{adapted} to $E$), then $\Om_X^1(\log E)$
is locally free at $a$ with basis given by
\begin{equation}\label{eq:logbasis}
\frac{du_i}{u_i}, \, i=1,\ldots,s, \,\, \text{and } \, dv_j, \, j=1,\ldots,n-s.
\end{equation}
There is a natural inclusion $\Om_X^1 \hookrightarrow \Om_X^1(\log E)$ (given by writing any
1-form in terms of a ``logarithmic basis'' \eqref{eq:logbasis}).

Given a singular variety $X_0$, we also write $\Om_{X_0}^1$ for the \emph{cotangent sheaf} of $X_0$.
($\Om_{X_0}^1$ has stalk $\um_{X_0,a}/\um_{X_0,a}^2$ at $a$, where $\um_{X_0,a}$ denotes
the maximal ideal of $\cO_{X_0,a}$.)
Suppose that $\s: (X,E) \to (X_0, \Sing X_0)$ is a resolution of singularities (in particular, $\supp E = \s^{-1}(\Sing X_0)$). 
Let $\s^*\Om_{X_0}^1$ denote the submodule of $\Om_X^1$ generated by the pull-back of 
$\Om_{X_0}^1$, and consider the quotient $\cO_X$-module 
\begin{equation}\label{eq:phi}
\Phi := \Om_X^1(\log E)/\s^*\Om_{X_0}^1.
\end{equation}
(If $X_0 \hookrightarrow Z_0$, where $Z_0$ is smooth, then $\Om_{X_0}^1$ is induced by the restriction
to $X_0$ of $\Om_{Z_0}^1$, and $\s^*\Om_{X_0}^1 = \s^*\Om_{Z_0}^1$.)

\begin{remark}\label{rem:lindep}
Let $\bx := (x_1,\ldots,x_n)$ and let $\bbe_1,\ldots, \bbe_k \in \IN^n$. If $\bg$ is linearly dependent on
$\bbe_1,\ldots, \bbe_k$ over $\IQ$, say $\bg = \sum_{i=1}^k q_i \bbe_i$, where each $q_i \in \IQ$,
then $d \bx^{\bg} = \sum_i q_i x^{\bg-\bbe_i}d \bx^{\bbe_i}$. In particular, if $\bg \geq \bbe_i$, for all $i$,
then $d \bx^{\bg}$ is in the submodule generated by the $d \bx^{\bbe_i}$.
\end{remark}

If $X_0|_V \hookrightarrow M_0$ is a local embedding over a neighbourhood $V$ of $b = \s(a)$,
then $\s^*(\Om_{X_0}^1|_V) = \s^*\Om_{M_0}^1$. If $\dim M_0 = N$ and $\s$ is expressed in components
$\s = (\s_1,\ldots,\s_N)$ with respect to local coordinates of $M_0$ at $b$, then $\Phi$ has a
presentation at $a$ given by (the transpose of) the \emph{logarithmic Jacobian matrix} of $\s$,
\smallskip
\begin{equation}\label{eq:logjac}
\log \Jac \s = 
\begin{pmatrix}
\displaystyle{u_1\frac{\p \s_1}{\p u_1}} & \cdots & \displaystyle{u_s \frac{\p \s_1}{\p u_s}} & 
     \displaystyle{\frac{\p \s_1}{\p v_1}} & \cdots & \displaystyle{\frac{\p \s_1}{\p v_{n-s}}}\\
\vdots & \ddots & \vdots & \vdots & \ddots & \vdots\\
\displaystyle{u_1\frac{\p \s_N}{\p u_1}} & \cdots & \displaystyle{u_s \frac{\p \s_N}{\p u_s}} & 
\displaystyle{\frac{\p \s_N}{\p v_1}} & \cdots & \displaystyle{\frac{\p \s_N}{\p v_{n-s}}}
\end{pmatrix}.
\end{equation}

\smallskip
Note that every minor of order $n$ of $\log \Jac \s$ equals $u_1\cdots u_s$ times the corresponding
minor of order $n$ of the standard Jacobian matrix $\Jac \s$.

The rank at $a$ of $\log \Jac \s$ will be called the \emph{logarithmic rank} $\log \rk_a \s$ of $\s$ at $a$.
It is clear from \eqref{eq:logjac} that, if $a \in \supp E$, then $\log \rk_a \s = \rk_a (\s|_{E(a)})$,  where $E(a)$ 
denotes the intersection of the components of $E$ containing $a$ (We call $E(a)$ the stratum of $E$ at $a$. 
If $a \notin \supp E$, then $\log \rk_a \s := \rk_a \s$.)

\subsection{Logarithmic Fitting ideals}\label{subsec:fit}
For each $k=0,\ldots,n-1$, the \emph{logarithmic Fitting ideal}
$\cF_k = \cF_k(\s)$ denotes the ideal (i.e., sheaf of ideals) of $\cO_X$ generated by the minors of order
$n-k$ of $\log \Jac \s$. The Fitting ideals $\cF_k$ depend only on the quotient module $\Phi$ \eqref{eq:phi};
in particular, they are independent of the choices of adapted local coordinates of $(X,E)$, the local 
embedding $X_0|_V \hookrightarrow M_0$ and the local coordinates of $M_0$ (cf. \cite[\S20.2]{E}).

The relevance of logarithmic Fitting ideals to the main conjecture was recognized by Pardon and Stern \cite{PS},
and Theorem \ref{thm:fitHP} is suggested by their work.

\subsection{Transformation of logarithmic Fitting ideals by blowing up}\label{subset:blup}
Let $\be$ denote an admissible blowing-up with centre $C \subset \supp E$, and let $E'$ denote the 
\emph{transform} of $E$ by $\be$ (by definition, the components of $E'$ are the strict transforms by $\be$ of
the components of $E$, together with the exceptional divisor of $\be$). In this case, $\supp E' = \be^{-1}(\supp E)$;
we will write $\be: (X',E') \to (X,E)$. With the respect to adapted
local coordinates $(\bu,\bv)$ at a point $a \in C$, as above, we can assume that $C$ is given by
$$
u_1 = \cdots = u_k = 0, \ \text{for some} \ k \geq 1, \ \text{and}\ \, v_1 = \cdots = v_l = 0, \ \text{for some} \ l \geq 0.
$$
The blowing-up $\be$ is combinatorial (\S\ref{subsec:res}) precisely when $l=0$.

Over the $(\bu,\bv)$ coordinate chart, $X'$ can be covered by $k+l$ coordinate charts --- a ``$u_i$-chart'', for
each $i=1,\ldots,k$, and a ``$v_j$-chart'', for each $j=1,\ldots,l$. For example, the $u_1$-chart has coordinates
$(\bu',\bv') = (u_1',\ldots,u_s',v_1',\ldots,v_{n-s}')$ given by
\begin{equation*}
u_1' = u_1,\qquad u_i' = \left\{ \begin{array}{l l} 
                                       u_i/u_1, & 2\leq i\leq k\\
                                       u_i ,      & i > k
                                       \end{array}\right.,\qquad 
                    v_j' = \left\{ \begin{array}{l l}
                                        v_j/u_1 & j\leq l\\
                                        v_j        & j > l
                                        \end{array}\right.,
\end{equation*}                                         
and the $v_1$-chart has coordinates $(\bu',\bv') = (u_1',\ldots,u_{s+1}',v_2',\ldots,v_{n-s}')$ given by
\begin{equation*}
u_i' = \left\{ \begin{array}{l l} 
                                       u_i/v_1, & 1\leq i\leq k\\
                                       u_i ,      & k < i\leq s
                                       \end{array}\right.,\qquad 
                    u_{s+1}' = v_1,\qquad
                    v_j' = \left\{ \begin{array}{l l}
                                        v_j/v_1 & 2\leq j\leq l\\
                                        v_j        & j > l
                                        \end{array}\right..
\end{equation*}                             
                                       
We compute
\begin{equation*}
\log \Jac (\s\circ\be) = ((\log \Jac \s)\circ\be)\cdot B,
\end{equation*}
where $B$ denotes the $n \times n$ matrix
\smallskip
\begin{equation*}
\begin{pmatrix}
A & 0\\
0 & I
\end{pmatrix} \cdot \Jac \be \cdot \begin{pmatrix}
                                             C & 0\\
                                             0 & D
                                             \end{pmatrix},
\end{equation*} 
with $A$ (respectively, $C$) the $s \times s$ diagonal matrix with diagonal entries $1/u_i$ 
(respectively, $u_i'$), $I$ the identity matrix of order $n-s$, and $D$ a diagonal matrix of order $n-s$
with first entry $1$ in the $u_1$-chart or $u_{s+1}'$ in the $v_1$-chart, and remaining diagonal entries $1$.                                 
We make the following simple but important observations.

\begin{remarks}\label{rem:fit} (1) $\det B = \exc^l$, where $\exc$ denotes the exceptional divisor of $\s$;
i.e., $\exc = u_1'$ in the $u_1$-chart or $u_{s+1}'$ in the $v_1$-chart.

\smallskip
(2) If $\be$ is a combinatorial blowing-up ($l=0$), then $B$ is invertible, so that every minor of $\log \Jac (\s\circ\be)$
is a linear combination of minors of the same order of $(\log \Jac \s)\circ\be$, and vice-versa.
\end{remarks}

Therefore, in the notation above, we have the following.

\begin{lemma}\label{lem:fit}\begin{enumerate}\item
$\cF_0(\s\circ\be) = exc^l \cdot \be^*\cF_0(\s)$.
\item
If $\be$ is a combinatorial blowing-up, then $\cF_k(\s\circ\be) = \be^*\cF_k(\s)$, $k=0,\ldots$\,.
\end{enumerate}
\end{lemma}

\begin{theorem}\label{thm:fit0}
Given a reduced variety $X_0$, there is a resolution of singularities $\s: (X,E) \to (X_0, \Sing X_0)$ 
such that $\s$ is a composite of admissible blowings-up and $\cF_0(\s)$ is a principal ideal generated
locally by a monomial in generators of the components of $E$ (we will say that $\s$ is a
\emph{principal $E$-monomial ideal}, or a \emph{principal monomial ideal} if $E$ is clear from the context).

Moreover, if $\s$ is a resolution of singularities such that $\cF_0(\s)$ is a principal monomial ideal,
and $\be: (X',E') \to (X,E)$ is an admissible blowing-up, then $\cF_0(\s\circ\be)$ is a principal monomial ideal.
\end{theorem}

\begin{proof}
The second assertion is immediate from Lemma \ref{lem:fit}(1). Suppose that  $\s$ is a resolution of singularities
of the variety $X_0$ by admissible blowings-up. We can then apply resolution of singularities of an ideal to 
$\cF_0(\s)$, and the first assertion again follows from Lemma \ref{lem:fit}(1).
\end{proof}

\subsection{Regularization of the Gauss mapping}\label{subset:Gauss}
Given a smooth variety $M_0$ and $n \leq \dim M_0$, let $G(n,M_0)$ denote the \emph{Grassmann
bundle} of $n$-dimensional linear subspaces of the tangent spaces to $M_0$ at every point. If
$X_0 \hookrightarrow M_0$ and $Y_0 = \Sing X_0$, then there is a natural \emph{Gauss mapping}
$G_{X_0}: X_0\backslash Y_0 \to G(n,M_0)$, where $n = \dim X_0$, given by
$a \mapsto$ tangent space of $X_0$ at $a$.

Theorem \ref{thm:fit0} also provides a \emph{regularization} of the Gauss mapping:

\begin{theorem}\label{thm:Gauss}
Let $\s: (X,E) \to (X_0,Y_0)$ denote a resolution of singularities. Then the following are 
equivalent:
\begin{enumerate}
\item the pull-back $\s^*G_{X_0}$ extends to a regular (or analytic) morphism on $X$; 
\item $\s^*\Om_{X_0}^1$ is a locally free $\cO_X$-module of rank $n$;
\item the Fitting ideal $\cF_0(\s)$ is a principal ideal (not necessarily monomial).
\end{enumerate}
\end{theorem}

\begin{proof} The equivalence of (1) and (2) is a consequence of the fact that a sheaf
of $\cO_X$-modules is locally free if and only if it defines a vector bundle (see \cite[Exercise II.5.18]{Ha}).
To see that (1)$\iff$(3), note that the Grassmannian $\text{Grass}(n,N)$ of
$n$-planes in $\IC^N$; i.e., the space of linear injections $\by = \la(\bx)$ from $\IC^n \to \IC^N$, is the 
complex projective space of dimension $\binom{N}{n} -1$ with homogeneous coordinates
given by the minors of order $n$,
$$
\frac{\p(y_{i_1},\ldots,y_{i_n})}{\p(x_1,\ldots,x_n)},\quad i_1 < \dots < i_n,
$$
of the Jacobian matrix of $\la$ with respect to coordinate systems $\bx = (x_1,\ldots,x_n)$,
$\by = (y_1,\ldots,y_N)$ of $\IC^n$, $\IC^N$ (respectively).
\end{proof}

\begin{remark}\label{rem:Gauss}
Given $a \in X$, condition (1) of Theorem \ref{thm:Gauss} can be used to choose coordinates
for $M_0$ at $\s(a)$ such that each $d\s_i$, $i>n$, is in the submodule generated by
$d\s_1,\ldots,d\s_n$ at $a$. We will not use this result, but have included Theorem \ref{thm:Gauss}
for historical reasons.
\end{remark}

\section{Logarithmic Fitting ideals and desingularization of the cotangent sheaf}\label{sec:fitHP}
We continue to use the notation of Section \ref{sec:logfit}.

\subsection{Equivalence of the main conjecture and principalization of logarithmic Fitting ideals}
\label{subsec:fitHP}

In this subsection, we will prove Theorem \ref{thm:fitHP}.


According to Theorem \ref{thm:fit0}, there is a resolution of singularities $\s$ such that $\cF_0(\s)$ is
a principal monomial ideal, and the latter condition is stable by admissible blowings-up. Although
Theorem \ref{thm:fitHP} may seem attractive because our main conjecture \ref{conj:main} is obtained ``in one shot''
from the condition (1), the Fitting ideals $\cF_k(\s)$, $k > 0$, do not enjoy the stability property of
$\cF_0(\s)$, so in practice it may be difficult to obtain condition (1) step-by-step. (Example. Consider the
morphism $\s(u,v,w) = (u^2,u^3v,u^4w)$ and the admissible blowing-up with centre $(u = v+w^2 = 0)$.)

We will show, nevertheless, that, given a resolution of singularities $\s$ by admissible blowings-up,
then, after further admissible blowings-up, $\cF_{n-1}(\s)$ is a principal monomial ideal (Theorem \ref{thm:start}).
In fact, we will show that, after further admissible blowings-up, $\cF_{n-1}(\s),\ldots,\cF_{n-(r+1)}(\s)$ are principal 
monomial ideals at every point of log rank $r$. This seems useful as a way to begin an inductive proof
of our main conjecture. In particular, Theorems \ref{thm:fit0}, \ref{thm:fitHP} and \ref{thm:start} immediately establish
Conjecture \ref{conj:main} in the case that $\dim X_0 \leq 2$ (see Corollary \ref{cor:HP2}). 
We will use Theorem \ref{thm:start} to begin
an inductive proof in the $3$-dimensional case, in Section \ref{sec:prepnorm}.

It will be useful to have the more precise local statement of the
following lemma, which immediately implies Theorem \ref{thm:fitHP}. The proof of Lemma 
\ref{lem:fitHP} will include the precise relationship between the exponents of monomials generating
the log Fitting ideals, and the exponents appearing in the differential monomials involved in
Hsiang-Pati coordinates.

\begin{lemma}\label{lem:fitHP}
Let $\s: X \to M_0$ be a morphism between smooth varieties. Say $n=\dim X$, $N=\dim M_0$.
Suppose that the \emph{critical locus} of $\s$ (i.e., $\{a \in X: \rk_a \s < n\}$) is the support of a
simple normal crossings divisor $E$ on $X$. Let $a \in X$. Then, for each $k = 1,\ldots,n$ the
following are equivalent:
\smallskip
\begin{enumerate}
\item The logarithmic Fitting ideals $\cF_{n-1}(\s), \ldots, \cF_{n-k}(\s)$ are all principal $E$-monomial ideals at $a$.
\smallskip
\item There are (analytic or \'etale) coordinates $(\bu,\bv)$ of $X$ at $a$ adapted to $E$, and coordinates $\bz =
(z_1,\ldots,z_N)$ of $M_0$ at $\s(a)$, such that, writing $\s = (\s_1,\ldots,\s_N)$ with respect to the coordinates $\bz$,
\smallskip
\begin{enumerate}
\item the submodule $\cM_k$ of $\Om_{X,a}^1$ generated by the
pull-backs $\s^*dz_m = d\s_m$, $m=1,\ldots,k$, is also generated by differential monomials 
$d(\bu^{\bal_i})$, $i = 1,\ldots,l_k$, and $d(\bu^{\bbe_j}v_j)$, $j=1,\ldots,k-l_k$, for some $l_k \leq k$,
satisfying conditions as in Conjecture \ref{conj:main};
\smallskip
\item for each $m > k$, $\s_m = g_m + S_m$, where $d g_m \in \cM_k$ and $S_m$ is divisible
by $\bu^{\max\{\bal_{l_k}, \bbe_{k-l_k}\}}$.
\end{enumerate}
\end{enumerate}
\end{lemma}

\begin{proof} (2)$\implies$(1). Given ``partial'' Hsiang-Pati 
coordinates at a point $a$ of $X$, as in (2), each log Fitting ideal $\cF_{n-m}(\s)$, $m=1,\ldots,k$,  at $a$ 
is generated by
a minor of order $m$ of the matrix with rows given by the coefficients of the differential monomials
$d(\bu^{\bal_i})$ and $d(\bu^{\bbe_j}v_j)$ with respect to the logarithmic basis \eqref{eq:logbasis}. Each row is a
vector $(\bxi,\bta)$, where the components of $\bxi$ (respectively, $\bta$) are log derivatives with respect to
the coordinates $u_i$ (respectively, derivatives with respect to the $v_j$). The row vector 
corresponding to $\bu^{\bal_i}$
(respectively, to $\bu^{\bbe_j}v_j$) is $\bu^{\bal_i}(\bal_i,0)$ (respectively, $\bu^{\bbe_j}(v_j\bbe_j,(j))$, where 
$\bta = (j)$ denotes the vector with $1$ in the $j$'th place and $0$ elsewhere).

It is easy to see that each $\cF_{n-m}$ is generated by $\bu^{\bg_m}$, where $\bg_m$ is the sum of 
the first $m$
elements of $\{\bal_i,\bbe_j\}$ (as an ordered set).
\medskip

\noindent
(1)$\implies$(2). Assume that $\cF_{n-m} = (\bu^{\bg_m})$, $m=1,\ldots,k$, where $\bg_1 \leq \cdots \leq \bg_k$.
Write $\s$ as $\s = (\s_1,\ldots,\s_N)$ with respect to local
coordinates $\bz = (z_1,\ldots,z_N)$ of $M_0$ at $\s(a) = 0$. 

We will prove (2) (in fact, a stronger statement) by induction on $k$. 
Given $k\geq 0$, assume there are local coordinates $(\bu,\bv)$ at $a=0$ such that:
\smallskip
\begin{enumerate}
\item[(1)$_k$] The $\cO_{X,a}$-module generated by $d\s_1,\ldots,d\s_k$ is generated by $d(\bu^{\bal_i})$,
$i=1,\ldots,l_k$, and $d(\bu^{\bbe_j}v_j)$, $j=1,\ldots,k-l_k$, where $\{\bal_i,\bbe_j\}$ is totally ordered and
$\bal_1,\ldots,\bal_{l_k}$ are linearly independent over $\IQ$.
\smallskip

\item[(2)$_k$] For all $m=1,\ldots,k$, 
$$
\s_m = g_m + S_m
$$
(sum of analytic functions, or regular functions in the \'etale chart), where,
\smallskip
\begin{enumerate}
\item for every monomial $\bu^{\bbe}\bv^{\bg}$ appearing (i.e., with nonzero coefficient) in the formal
expansion of $g_m$ at $a=0$, $(\bbe,\bg)$ is linearly dependent on the $(\bal_i,0)$,  $i=1,\ldots,l_{m-1}$, and 
$(\bbe_j,(j))$, $j=1,\ldots,m-1-l_{m-1}$, over $\IQ$, and $\bbe \geq \bal_i,\, \bbe_j$, for all such $i,j$;
\smallskip
\item $S_m = \bu^{\bal_{l_{m-1}+1}}$ or $S_m = \bu^{\bbe_{m-1-l_{m-1} +1}} v_{m-1-l_{m-1} +1}$.
\smallskip
\end{enumerate}
\smallskip
\noindent
Note that, if $(\bbe,\bg) = \sum_{i=1}^{l_{m-1}} q_i  (\bal_i,0) + \sum_{j=1}^{m-1-l_{m-1}} r_j (\bbe_j,(j))$,
where the $q_i, r_j \in \IQ$ and $r_j \neq 0$ for some $j$, then $\bbe \geq \bbe_j \implies (\bbe,\bg) \geq
(\bbe_j,(j))$, so (by Remark \ref{rem:lindep}), $dg_m$ is in the $\cO_{X,a}$-submodule generated by
the $d\bu^{\bal_i}$ and $d(\bu^{\bbe_j}v_j)$.
\smallskip

\item[(3)$_k$] For each $m=k+1,\ldots, N$,
$$
\s_m = g_{mk} + \bu^{\bde_{mk}}S_{mk},
$$
where $g_{mk}$ and $S_{mk}$ are analytic or regular functions such that
\smallskip
\begin{enumerate}
\item for every monomial $\bu^{\bbe}\bv^{\bg}$ appearing in the formal
expansion of $g_{mk}$ at $a=0$, $(\bbe,\bg)$ is linearly dependent on the $(\bal_i,0)$,  $i=1,\ldots,l_k$, and 
$(\bbe_j,(j))$, $j=1,\ldots,k-l_k$, over $\IQ$, and $\bbe \geq \bal_i,\, \bbe_j$, for all such $i,j$ (so $dg_{mk}$
is in the submodule generated by $d\s_1,\ldots,d\s_k$);
\smallskip
\item $\bde_{mk} \geq \max\{\bal_{l_k}, \bbe_{k-l_k}\}$ and, if $S_{mk}$ is a unit, then $\bde_{mk}$ is
linearly independent of $\bal_1,\ldots,\bal_{l_k}$;
\smallskip
\item $S_{mk}$ is divisible by no $u_i$ (unless $S_{mk} = 0$).
\end{enumerate}
\smallskip

\item[(4)$_k$] $\sum_{i=1}^{l_k} \bal_i + \sum_{j=1}^{k-l_k} \bbe_j = \bg_k$.
\end{enumerate}
\smallskip

The assumptions above (except for (3)$_k$(c)) are all empty if $k=0$.

Note that the generator $\bu^{\bg_k}$ of $\cF_{n-k}$ is given by a coefficient of
$d\s_1\wedge \cdots \wedge d\s_k$ (with respect to the log basis). It is easy to see that
the generator $\bu^{\bg_{k+1}}$ of $\cF_{n-(k+1)}$ is given by a coefficient of
$$
d\s_1 \wedge \cdots \wedge d\s_k \wedge d\s_{m_0},
$$
for some $m_0 \geq k+1$ (e.g., using elementary row and column operations on $\log \Jac \s$).

Set
$$
\Om_k := \bigwedge_{i=1}^{l_k} d(\bu^{\bal_i}) \wedge \bigwedge_{j=1}^{k-l_k} d(\bu^{\bbe_j}v_j).
$$
Then
$$
\Om_k = \bu^{\bg_k}\left\{ \sum_I c_I \bigwedge_I\hspace{-.2em}\frac{du_i}{u_i} \wedge dv_1 \wedge \cdots \wedge dv_{k-l_k} 
               + \eta_k\right\},
$$
where
\begin{enumerate}
\item[] $I$ runs over all sets $\{i_1,\ldots,i_{l_k}\}$ such that $1 \leq i_1 < \cdots < i_{l_k} \leq s$,
\smallskip
\item[] $\displaystyle{ \bigwedge_I\hspace{-.1em}\frac{du_i}{u_i} := \frac{du_{i_1}}{u_{i_1}} \wedge \cdots \wedge
                                 \frac{du_{i_k}}{u_{i_k}} }$,
\item[] $c_I$ is a unit, for some $I$,
\item[] $\eta_k \in (v_1,\dots, v_{k-l_k})\cdot\Om_{X,a}$.
\end{enumerate}

Let us compute $\Om_k \wedge d(\bu^{\bbe}\bv^{\bg})$, where $\bu^{\bbe}\bv^{\bg}$ is a monomial 
with $\bbe \geq \max\{\bal_{l_k}, \bbe_{k-l_k}\}$ (for example, a monomial in the formal expansion of
$\bu^{\bde_{mk}}S_{mk}$, where $m \geq k+1$).

\smallskip
\noindent
Case (i) $\bg = 0$. Now, $d\bu^{\bbe} = \bu^{\bbe} \sum \be_i du_i/u_i$. If $\bbe$ is $\IQ$-linearly dependent
on $\bal_1,\ldots,\bal_{l_k}$, then $\Om_k \wedge d\bu^{\bbe} = 0$. If $\bbe$ is $\IQ$-linearly independent of
the $\bal_i$, then some coefficient of $\Om_k \wedge d\bu^{\bbe}$ (with respect to the log basis) is $u^{\bg_k + \bbe}$
times a unit, and all other coefficients are divisible by $u^{\bg_k + \bbe}$.

\smallskip\noindent
Case (ii) $|\bg| \geq 1$. Then
$$
d(\bu^{\bbe}\bv^{\bg}) = \bu^{\bbe} \bv^{\bg} \sum \be_i \frac{du_i}{u_i} + \bu^{\bbe} \sum \g_j \bv^{\bg - (j)} dv_j.
$$
If $|\bg| > 1$, then all coefficients of $\Om_k \wedge d(\bu^{\bbe}\bv^{\bg})$ are divisible by some $v_j$.
Consider $|\bg| = 1$. If $\bg = (j)$, for some $j \leq k-l_k$, then $\bu^{\bbe} \bv^{(j)} = u^{\bbe - \bbe_j}u^{\bbe_j}v_j$
and $\Om_k \wedge d(\bu^{\bbe} \bv^{(j)})$ is in the submodule of logarithmic $(k+1)$-forms divisible by $v_j$. If
$\bg = (q)$, for some $q > k-l_k$, then some coefficient of $\Om_k \wedge d(\bu^{\bbe} \bv^{(q)})$ is $\bu^{\bg_k + \bbe}$
times a unit, and all other coefficients are divisible by $\bu^{\bg_k + \bbe}$.

\smallskip
We can assume that the monomial $\bu^{\bg_{k+1}}$ generating $\cF_{n-(k+1)}$ is given (up to a unit) by a coefficient of
$$
d\s_1 \wedge \cdots \wedge d\s_k \wedge d\s_{k+1}.
$$
Then
$$
\bg_{k+1} = \bg_k + \bde_{k+1,k}\quad \text{and }\quad \bde_{mk} \geq \bde_{k+1,k},\,\, m \geq k+1.
$$
Moreover, by the computation above,
either there is a monomial $P = \bu^{\bal}$ appearing in the formal expansion of $\bu^{\bde_{k+1,k}}S_{k+1,k}$
such that $\bde_{k+1,k} = \bal$, or there is a monomial $P = \bu^{\bbe}v_q$, where $q> k - l_k$, appearing
in the formal expansion of $\bu^{\bde_{k+1,k}}S_{k+1,k}$
such that $\bde_{k+1,k} = \bbe$.

We can now obtain (1)$_{k+1}$. Set $S := S_{k+1,k}$. First suppose $P = \bu^{\bal}$. Then $S$ is a unit and
$\bal_1,\ldots,\bal_{l_k},\bal$ are linearly independent over $\IQ$. Let $\bep = (\ep_1,\ldots,\ep_s) \in \IQ^s$ 
denote a shortest vector such that 
\begin{align*}
\langle \bal_i, \bep\rangle &= 0,\quad i=1,\ldots,l_k,\\
\langle \bal,\bep\rangle &= 1,
\end{align*}
and consider the coordinate change
\begin{alignat*}{3}
\ou_h &= S^{\ep_h} u_h, &&h = 1,\ldots,s,\\
\ov_j &= S^{-\langle\bbe_j,\bep\rangle}v_j,\quad &&j = 1,\ldots,k-l_k,\\
\ov_j &= v_j, &&j > k-l_k.
\end{alignat*}
Then $\obu^{\bbe} = S^{\langle\bbe,\bep\rangle}\bu^{\bbe}$, for any $\bbe$, so
\begin{alignat*}{3}
\obu^{\bal_i} &= \bu^{\bal_i}, &&i=1,\ldots,l_k,\\
\obu^{\bbe_j}\ov_j &= \bu^{\bbe_j}v_j,\quad &&j=i,\ldots, k-l_k,\\
\obu^{\bal} &= S \bu^{\bal}. &&
\end{alignat*}
Set
\begin{align*}
\bal_{l_k + 1} &:= \bal = \bg_{k+1} - \bg_k,\\
l_{k+1} &:= l_k + 1.
\end{align*}
Then the $\cO_X$-module generated by $d\s_1,\ldots,d\s_{k+1}$ is also generated by
$$
d\obu^{\bal_i},\,\, i=1,\ldots,l_{k+1},\quad \text{and }\quad d(\obu^{\bbe_j}\ov_j),\,\, j=1,\ldots,k+1-l_{k+1}.
$$

Secondly, suppose $P = \bu^{\bbe}v_q$, where $q > k - l_k$. Then $S(0)=0$ and
$(\p S/\p v_q)(0) \neq 0$. We can assume that $q = k-l_k+1$. Consider
the coordinate change
\begin{alignat*}{3}
\obu &= \bu, &&\\
\ov_j &= v_j,\quad &&j\neq k-l_k+1,\\
\ov_{k-l_k+1} &= S. &&
\end{alignat*}
Set $l_{k+1} := l_k$. Then the $\cO_X$-module generated by $d\s_1,\ldots,d\s_{k+1}$ is also generated by
$$
d\obu^{\bal_i},\,\, i=1,\ldots,l_{k+1},\quad \text{and }\quad d(\obu^{\bbe_j}\ov_j),\,\, j=1,\ldots,k+1-l_{k+1}.
$$

Properties (2)$_{k+1}$ and (4)$_{k+1}$ are clear from the construction.

It remains to verify (3)$_{k+1}$. Clearly, $g_{mk}(\obu,\obv) = g_{mk}(\bu,\bv)$, for all $m \geq k+1$;
i.e., $\obu^{\bal_i} = \bu^{\bal_i}$ and $\obu^{\bbe_j}\ov_j = \bu^{\bbe_j}v_j$, for all monomials involved.
In the analytic case, for each $m\geq k+2$, we can define $g_{m,k+1}$ by adding to $g_{mk}(\obu,\obv)$ 
all monomials $\obu^{\bbe}\obv^{\bg}$ (times nonzero constants) appearing in $\s_m - g_{mk}$ such that 
$(\bbe,\bg)$ is a $\IQ$-linear combination
of $(\bal_i,0),\, (\bbe_j,(j))$, $i=1,\ldots,l_{k+1},\, j=1,\ldots,k+1-l_{k+1}$. Note that such $(\bbe,\bg)$ is $\geq$
all the $(\bal_i,0),\, (\bbe_j,(j))$.

In the algebraic case, the preceding construction provides formal expansions $\hat{g}_{m,k+1}$ and
$\bu^{\bde_{m,k+1}}\widehat{S}_{m,k+1}$ that are not \emph{a priori} algebraic. In this case, for each
$m\geq k+2$, we can define $g_{m,k+1} := \hat{g}_{m,k+1} - \hat{h}_{m,k+1}$ and 
$S_{m,k+1} := \widehat{S}_{m,k+1} + \hat{h}_{m,k+1}/\bu^{\bde_{m,k+1}}$, where $\hat{h}_{m,k+1}$ denotes
the sum of all terms $c_{\bbe \bg}\bu^{\bbe}\bv^{\bg}$ ($c_{\bbe \bg} \neq 0$) of $\hat{g}_{m,k+1}$ with
$(\bbe,\bg) > (\bde_{m,k+1},0)$. Then we still have (3)$_{k+1}$ (as well as (1)$_{k+1}$, (2)$_{k+1}$ and (4)$_{k+1}$).
Moreover, $g_{m,k+1}$ and $S_{m,k+1}$ are algebraic; we explain this in Remark \ref{rem:fitHP}(1) following
because the remark will be needed also in the proof of Lemma \ref{lem:rho}.

This completes the proof of Lemma \ref{lem:fitHP}.
\end{proof}

\begin{remarks}\label{rem:fitHP} 
(1) Given $\s_m = \hat{g}_{m,k+1} + \bu^{\bde_{m,k+1}}\widehat{S}_{m,k+1}$ as well as $g_{m,k+1}$ and 
$S_{m,k+1}$, as above, let $Q$ and $R$ 
denote the quotient and remainder of $\s_m$ (respectively) after division by the monomial $\bu^{\bde_{m,k+1}}$. 
(This means that, formally, $R$ is the sum of all terms of $\hat{g}_{m,k+1}$ which are not divisible by
$\bu^{\bde_{m,k+1}}$.) Then $Q$ and $R$ are algebraic. Moreover, $g_{m,k+1} = R$ and $S_{m,k+1} = Q$
unless $\hat{g}_{m,k+1}$ includes a term $c\bu^{\bde_{m,k+1}}$ (with nonzero coefficient $c$). In the latter
case, $\bde_{m,k+1}$ is linearly dependent on $\bal_1,\ldots,\bal_{l_k}$, and $g_{m,k+1} = R + c\bu^{\bde_{m,k+1}}$,
$S_{m,k+1} = R - c$.

\smallskip
(2) Theorem \ref{thm:fitHP} says that, under the assumption that all log Fitting ideals $\cF_k(\s)$ are principal,
we get Hsiang-Pati coordinates at every point of $X$, as in Conjecture \ref{conj:main}; according to the
proof of Theorem \ref{thm:fitHP}, it may happen that
the first differential monomial according to the order in condition (3) of the latter is $d(\bu^{\bbe_1}v_1)$.
In practice, we will principalize $\cF_0$ using Theorem \ref{thm:fit0} and then try to principalize $\cF_{n-1}, \cF_{n-2},\ldots$
inductively. The resulting ordered list of differential monomials will always begin with either $\bu^{\bal_1}$ or $v_1$
(see Theorem \ref{thm:start} below). 

One can add to Conjecture \ref{conj:main} the condition that each $\bbe_j$ be linearly dependent on all
preceding $\bal_i$ in the ordered list. See also \cite{Y}. Given Hsiang-Pati coordinates as in Conjecture
\ref{conj:main}, one can obtain the additional condition at least locally, by further admissible
blowings-up over $E$ if necessary. We do not know of a situation where the stronger condition is needed, but
it may simplify proofs.

\smallskip
(3) Lemma \ref{lem:fitHP} includes, in particular, the statement of Theorem \ref{thm:fitHP}
locally at a point $a \in X$. Since Lemma \ref{lem:fitHP}(1) is an open condition, the lemma
implies that Hsiang-Pati coordinates at a point of $X$ induce Hsiang-Pati coordinates at
nearby points.
\end{remarks}

\subsection{Logarithmic rank and principalization of Fitting ideals of low-order minors}\label{sec:start}
In this section, we assume that $X_0$ is a complex-analytic variety or an algebraic variety over an
algebraically closed field of characteristic zero. The results also hold, however, for a real-analytic variety
or an algebraic variety over an arbitrary field of characteristic zero (see Remark \ref{rem:real}).
Let $\s: (X,E) \to (X_0, Y_0 :=\Sing X_0)$ be a resolution of singularities of $X_0$. 
In particular, $\s^*\cI_{Y_0}$ is a principal $E$-monomial ideal. Set
$$
p := \max_E \log\rk \s = \dim Y_0,
$$
and, for each $k=0,\ldots, p$, set
\begin{align*}
\Sigma_k &:= \{a \in E: \log\rk_a \s \leq p-k\},\\
Y_k &:= \s(\Sigma_k).
\end{align*}
Since $\s$ is proper, each $Y_k$ is a closed subvariety of $X_0$, and 
$$
\Sing X_0 = Y_0 \supset Y_1 \supset \cdots \supset Y_p.
$$
Define
$$
\cI_{Y_k} := \text{ideal of $Y_k$ in $\cO_{X_0}$},\quad k=0,\ldots,p.
$$

\begin{lemma}\label{lem:start}
After further admissible blowings-up over $Y_0$ if necessary, we can assume that, for each $k$, $\Sigma_k = \s^{-1}(Y_k)$,
$Y_k\backslash Y_{k+1}$ is smooth, and $\s^*\cI_{Y_k}$ is 
a principal monomial ideal.
\end{lemma}

\begin{remarks}\label{rem:start}
(1) It follows from the lemma that, for each $k$, $\log \rk \s = p-k$ on $\Sigma_k\backslash \Sigma_{k+1}$.

\smallskip
(2) The condition that $\s^*\cI_{Y_k}$ be a principal monomial ideal is not stable under admissible blowing-up.
\end{remarks}

\begin{proof}[Proof of Lemma \ref{lem:start}] $Y_0 = \Sing X_0$, so that $\Sigma_0 = \supp E = \s^{-1}(Y_0)$
and $\s^*\cI_{Y_0}$ is a principal monomial ideal with support $\Sigma_0$.

Of course, $\Sig_1 \subset \Sig_0$. Set
$$
Y'_1 := \s(\Sig_1) \cup \Sing Y_0.
$$ 
Then $\dim Y'_1 \leq p-1$. By resolution of singularities, 
after further admissible blowings-up with centres over $Y'_1$ (i.e., in the
inverse image of $Y'_1$), we can assume that $\s^*\cI_{Y'_1}$ is a principal monomial ideal. Then
$\Sig_1 = \supp \s^*\cI_{Y'_1} = \s^{-1}(Y'_1)$, $Y_1 = Y'_1$ and $Y_0\backslash Y_1$ is smooth.

Again, $\Sig_2 \subset \Sig_1$. Let 
$$
Y'_2 := \s(\Sig_2) \cup \Sing Y_1.
$$
Then $\dim Y'_2 \leq p-2$. After further admissible blowings-up with centres over $Y'_2$, we can 
assume that $\s^*\cI_{Y'_2}$ is a principal monomial ideal. Then
$\Sig_2 = \supp \s^*\cI_{Y'_2} = \s^{-1}(Y'_2)$, $Y_2 = Y'_2$ and $Y_1\backslash Y_2$ is smooth. Moreover,
it is still true that $\s^*\cI_{Y_1}$ is a principal monomial ideal,
$\Sig_1 = \supp \s^*\cI_{Y_1} = \s^{-1}(Y_1)$, and $Y_0\backslash Y_1$ is smooth. 

We can continue in the same way to prove the lemma.
\end{proof}

\begin{remark}\label{rem:real}
If $X_0$ is a real-analytic variety, or an algebraic variety over a field that is not algebraically closed,
then the $\s(\Sig_k)$ need not be closed varieties. The proof of Lemma \ref{lem:start} goes through,
however, provided that each $\s(\Sig_k)$ lies in a closed variety of dimension 
$= d_k :=\max \{\log\rk_a \s: a \in \Sig_k\}$; 
in this case,
we can simply replace each $Y'_k$ in the proof by the smallest closed subvariety of $Y_{k-1}$
containing $\s(\Sig_k) \cup \Sing Y_{k-1}$. The preceding condition holds in the algebraic case, in general
(cf. \cite{Rond}). It holds in the real-analytic case because $X_0$ has a complexification $X_0^\IC$
\cite{Tognoli}, and $\s$ is induced by a resolution of singularities $\s^\IC: (X^\IC, E^\IC) \to (X_0^\IC, \Sing X_0^\IC)$
of $X_0^\IC$ (in particular, $\s^\IC$ is proper). Lemma \ref{lem:start} applies to $\s^\IC$. In the proof of 
Lemma \ref{lem:start} in the real case, we can take $Y'_k$ to be the real part (i.e., the invariance
space with respect to the canonical autoconjugation) of a union of components of $(Y_k^\IC)'$; cf. \cite[Sect.\,2]{Hiro}.
\end{remark}

\begin{theorem}\label{thm:start}
Let $\s: (X,E) \to (X_0, Y_0 :=\Sing X_0)$ denote a resolution of singularities of $X_0$ satisfying the
conclusion of Lemma \ref{lem:start}. (We use the notation
at the beginning of the subsection.) Let $a \in E$, and let $r := \log \rk_a \s$. Then, for any
local embedding $X_0 \hookrightarrow M_0$ (at $a$) in a smooth variety $M_0$, we can choose 
coordinates $(\bu,\bv) = (u_1,\ldots,u_s,v_1,\ldots,v_{n-s})$ adapted to $E$ for $X$ at $a=0$, and
coordinates $z=(z_1,\ldots,z_N)$ for $M_0$ at $\s(a)$, with respect to which, if $\s = (\s_1,\ldots,\s_N)$,
then
$$
\s_1 = v_1,\,\, \ldots,\,\, \s_r = v_r,\,\, \s_{r+1} = \bu^{\bal_1},
$$ 
where $\bal_1 \in \IN^s$, $\cI_{Y_{p-r}}$ is generated by $z_{r+1},\ldots,z_N$ at $\s(a)$,
and $\s^*\cI_{Y_{p-r}}$ is generated by $\s_{r+1} = \s^*(z_{r+1})$ at $a$.
\end{theorem}

\begin{proof}
Let $E_i$, $i=1,\ldots,s$, denote the components of $E$ at $a$. We call $\cap E_i$ the 
\emph{stratum} $E(a)$ of $a$ (cf. \ref{subsec:logdiff}). We can assume that
$$
\supp  \s^*\cI_{Y_{p-r}} = \bigcup_{i=1}^t E_i,
$$
at $a$, where $t \leq s$. Then $\log \rk_b \s = r$, for all $b \in \cup_{i=1}^t E_i$ near $a$.

Let $z=(z_1,\ldots,z_N)$ denote coordinates for $M_0$ at $\s(a)$. It follows from the
implicit function theorem that, after permuting the $z_j$ if necessary,
we can choose coordinates $(\bu,\bv) = (u_1,\ldots,u_s,v_1,\ldots,v_{n-s})$ 
for $X$ adapted to $E$ at $a=0$, such that $(v_1,\ldots,v_r)$ forms part of a system of coordinates for 
$E(a)$ at $a$, and
\begin{enumerate}
\item $\s_1 = v_1,\, \ldots,\, \s_r = v_r$ at $a$;
\smallskip
\item for each $j > r$, $\s_j = \s_j(v_1,\ldots,v_r)$ on $E_i$ at $a$, $i=1,\ldots,t$.
\end{enumerate}

Since $Y_{p-r}$ is smooth at $\s(a)$, then $z_j - \s_j(z_1,\ldots,z_r)$, $j > r$, generate the ideal of 
$Y_{p-r}$ at $a$. After a coordinate change
$$
\oz_j := z_j - \s_j(z_1,\ldots,z_r),\quad j> r,
$$
we can therefore assume that $z_{r+1},\ldots,z_N$ generate $\cI_{Y_{p-r}}$ at $\s(a)$ (so that 
 $\s_{r+1},\ldots, \s_N$ generate $\s^*\cI_{Y_{p-r}}$ at $a$). Since the
latter is a principal monomial ideal, we can also assume that $\s_{r+1} = \bu^{\bal_1}$, as required.
\end{proof}

The following is an immediate consequence of 
Theorems \ref{thm:fit0}, \ref{thm:fitHP} and \ref{thm:start}.

\begin{corollary}\label{cor:HP2}
Conjecture \ref{conj:main} holds in the case that $\dim X_0 \leq 2$.
\end{corollary}

\section{Invariant of a logarithmic Fitting ideal}\label{sec:inv}
We use the notation of Section \ref{sec:fitHP}.
Let $\s: (X,E)\to (X_0,\Sing X_0)$ denote a resolution of singularities of $X_0$, $n = \dim X_0$.

\begin{definition}\label{def:inv} Given $k=0,\ldots,n-1$ and
$a \in X$, let $\cR_{k,a}$ denote the \emph{residual ideal} of $\cF_k(\s)_a$ in $\cO_{X,a}$; i.e.,
$\cF_k(\s)_a = \prod_q \cI_{E_q,a}^{\mu_q} \cdot \cR_{k,a}$, with the $\mu_q \in \IN$
as large as possible, where $\{E_q\}$ denotes the set of components of $E$ and
$\cI_{E_q}$ is the ideal sheaf of $E_q$. Let $\rho_k(a)$ denote the order of $\cR_{k,a}$ in the
local ring $\cO_{X,a}/\sum_{a\in E_q} \cI_{E_q,a}$ (cf. \cite[Section\,5]{BMfunct}, \cite[Section\,2]{Cut2}; 
$\rho_k(a) := \infty$ if and only if $\cR_{k,a} = 0$ in the preceding local ring).
\end{definition}

The following lemma lists several properties of the basic invariant $\rho_k(a)$ that are all either clear or
easy to prove.

\begin{lemma}\label{lem:invprops}\begin{enumerate}\item $0 \leq \rho_k(a) \leq \infty$.
\item $\rho_k(a) = 0$ if and only if $\cF_k(\s)_a$ is a principal monomial ideal.
\item $\rho_k$ is upper-semicontinuous in the Zariski topology of $X$.
\item If $a$ is an $n$-point, then $\rho_k(a) =0$ or $\infty$.
\item If $\be: (X',E') \to (X,E)$ is a combinatorial blowing-up and $a' \in \s^{-1}(a)$, then $\rho_k(a') \leq \rho_k(a)$.
\end{enumerate}
\end{lemma}

We will only need $\rho_k$ in the case that $k=n-2$ in this article; i.e., for the log Fitting ideal of $2\times 2$ minors.
Write $\rho := \rho_{n-2}$. Lemma \ref{lem:rho} below extends in a straightforward way to $\rho_k(a)$, for any $k$, with
the assumption that the Fitting ideals $\cF_{n-1}(\s)_a, \cF_{n-2}(\s)_a,\ldots,\cF_{k+1}(\s)_a$ are all principal
monomial ideals. We present the lemma only in the case needed for the remainder of the paper, in part so
that we can fix notation that will be used in the following sections.

\begin{lemma}[Weierstrass form]\label{lem:rho}
Let $a \in \supp E$ be an $s$-point ($1 \leq s \leq n$). Suppose that $\cF_{n-1}(\s)_a$ is a principal monomial
ideal and that $\log\rk_a \s = 0$. Then:
\begin{enumerate}
\item Let $M_0$ denote a local embedding variety for $X_0$ at $\s(a)$. Then there are adapted local coordinates
$(\bu,\bv) = (u_1,\ldots,u_s,v_1,\ldots v_{n-s})$ for $X$ at $a$ (where the $(u_k=0)$ are the components of $E$ at $a$),
and local coordinates $\bz=(z_1,\ldots,z_N)$ for $M_0$ at $\s(a)$ with respect to which the components $\s_i$ of $\s$
can be written
\begin{equation}\label{eq:prelim}
\begin{aligned}
\s_1 &=  \bu^{\bal}, \quad \bal \in \IN^s,\\
\s_i &= g_i(\bu) + \bu^{\bde} T_i, \quad i=2,\ldots,N,
\end{aligned}
\end{equation}
where $\bu^{\bal}$ divides all $\s_i$, each $g_i$ and $T_i$ is analytic (or regular), each $dg_i$ is in the
submodule generated by $d\bu^{\bal}$, the $T_i$ are not simultaneously divisible by any $u_k$, and
$\bde$ is linearly independent of $\bal$ if some $T_i$ is a unit.

\medskip
\item Given coordinates as above in which \eqref{eq:prelim} satisfied, let 
$$
d = d(a) := \min\{|\bg|: \bg \in \IN^{n-s} \text{ and } \p_{\bv}^{\bg} T_i \text{ is a unit, for some } i\}.
$$
Then\begin{enumerate}\item $\rho(a) < \infty$ if and only if $d(a) < \infty$;
\item $\rho(a) = 0$ if and only if $d(a) = 0$ or $1$;
\item if $0 < \rho(a) < \infty$, then $d(a) = \rho(a) +1$.
\end{enumerate}

\medskip
\item Suppose that $0 < \rho(a) < \infty$. Then there are adapted coordinates $(\bu,v,\bw) = 
(u_1,\ldots,u_s,v,w_1,\ldots w_{n-s-1})$ for $X$ at $a$ and coordinates $\bz=(z_1,\ldots,z_N)$ for $M_0$ 
at $\s(a)$ with respect to which the components $\s_i$ of $\s$ can be written as in \eqref{eq:prelim} with
\begin{equation}\label{eq:weier}
T_i(\bu,v,\bw) = \tT_i(\bu,v,\bw)v^d  + \sum^{d-1}_{j=0} a_{ij}(\bu,\bw) v^j,\quad i=2,\ldots,N,
\end{equation}
where $\tT_2$ is a unit that we will also denote $U$, $a_{2,d-1} = 0$, 
and all monomials of order $\leq d + |\bde|$ in the formal expansions of the $\bu^{\bde} T_i$ at $a$,
$i=2,\ldots,N$, are ``linearly independent'' of $\bu^{\bal}$ (i.e., have exponents with respect to $(\bu,v,\bw)$ that
are linearly independent of $(\bal,0,\bzero)$.
\end{enumerate}
\end{lemma}

\begin{proof}
It is easy to obtain (1), where the formal expansion of each $g_i$ at $a$ is a sum of monomials $\bu^{\bbe}$
where each $\bbe$ is a rational multiple $q\bal$, $q\geq 1$ (see Remarks \ref{rem:lindep} and \ref{rem:fitHP}(1)).

In (2), if $d=0$, then
$\bde$ is linearly independent of $\bal$ and $T_i$ is a unit, for some $i$; say $i=2$. Then, after a coordinate
change, we can assume $T_2 = 1$, so that $\cF_{n-2}(\s)_a$ is generated by $\bu^{\bal+\bde}$. If $d=1$,
then, after a coordinate change, we can assume $T_2 = v_1$, and again $\cF_{n-2}(\s)_a$ is generated by 
$\bu^{\bal+\bde}$. On the other hand, (a) is clear and it is easy to see directly from the log Jacobian
matrix of $\s$ that the residual ideal $\cR_a$ of $\cF_{n-2}(\s)_a$ is generated modulo $\sum_{a\in E_k} \cI_{E_k,a}$
by the partial derivatives $\p_{v_l}T_i$, $i=2,\ldots,N$, $l =1,\ldots, n-s$, together with $(\al_p\de_q - \al_q\de_p)T_i$,
$i=2,\ldots,N$, $p,q=1,\ldots,s$. It follows that, if $d \geq 1$, then $\rho(a) = d-1$.

Given \eqref{eq:prelim}, after a permutation of the coordinates $(z_1,\ldots,z_N)$ and a generic linear
coordinate change in $\bv = (v_1,\ldots, v_{n-s})$, we can write the $T_i$ in the form \eqref{eq:weier},
where $\tT_2=U$ is a unit,
if we allow the sum in $T_2$ also to go from $j=0$ to $d-1$. Then we can eliminate $a_{2,d-1}$ by
completing the $d$th power with respect to $v$. The final condition in (3) can be obtained by ``moving''
the lower order monomials of (the formal expansions of) the $\bu^{\bde} T_i$ that are rational powers
of $\bu^{\bal}$ to the $g_i$.
\end{proof}

\section{Three-dimensional case: outline of the proof}\label{sec:outline}

In this section, we outline the proof Theorem \ref{thm:dim3}, which will be completed in Sections
6, 7 following.
Assume that $\dim X_0 =3$. By Theorems \ref{thm:fit0}, \ref{thm:fitHP} and \ref{thm:start}, there
is a resolution of singularities $\s: (X,E) \to (X_0, \Sing X_0)$ such that $\cF_0(\s),\, \cF_2(\s)$ are
principal monomial ideal sheaves and, moreover, if $a \in \supp E$ and $\log \rk_a \s > 0$, then
$\cF_1(\s)_a$ is also a principal monomial ideal. We will make $\cF_1(\s)$ a principal monomial ideal
sheaf by admissible blowings-up that preserve the preceding conditions on $\s$. Recall that
combinatorial blowings-up, in particular, preserve these conditions (Lemma \ref{lem:fit}(2)).

The blowings-up that we use to principalize $\cF_1(\s)$ will have the additional property that
$\log \rk \s$  is identically zero on the image in $X$ of every centre (or, equivalently, that
every centre of blowing up lies over $\supp \cF_2(\s)$). Thus we will blow up only over a discrete
subset of $X_0$.

We will say that $a$ is \emph{resolved} if $\cF_1(\s)_a$ is a principal monomial ideal;
i.e., $\rho(a) = 0$. Our proof of Theorem \ref{thm:dim3} has three main steps:

\medskip\noindent
{\bf Step 1.} Reduction to the case that $\rho(a) < \infty$, for all $a$. 

\medskip
If $0 < \rho(a) < \infty$ and $\log \rk_a \s = 0$, then the conclusions of Lemma \ref{lem:rho}(3) hold; in this case, we
will say that $\s$ is in \emph{Weierstrass form} at $a$.

Note that the set of 2-points forms a collection of curves (``2-curves'') each given by (a connected component of)
the intersection of precisely two components of $E$. A 2-curve either is closed or has limiting 3-points. By Lemma
\ref{lem:start}, we can assume that, throughout each 2-curve, either $\log \rk \s = 0$ or $\log \rk \s = 1$. A
2-curve on which $\log \rk \s = 0$ is relatively compact. By Theorem \ref{thm:start}, $\rho = 0$ on every 2-curve
where $\log \rk \s = 1$.

\begin{lemma}\label{lem:rhofinite}
By combinatorial blowings-up (more precisely, by composing $\s$ with a morphism $\tau: (\tX,\tE) \to (X,E)$
that restricts to a finite sequence of
combinatorial blowings-up over any relatively compact open subset of $X_0$), we can reduce
to the case that 
\begin{enumerate}
\item $\rho(a) < \infty$ at every point $a$ (and, therefore, by Lemma \ref{lem:rho}, we can choose coordinates at 
every nonresolved point $a$ and its image $\s(a)$ in which $\s$ has Weierstrass form);
\item $\rho$ is generically zero on every component of the set of 1-points, and on every 2-curve.
\end{enumerate}
In particular, in this case, the set of nonresolved points comprises isolated
2-points, isolated 1-points, and closed curves that are generically 1-points. Moreover, $\tau$ can be
chosen so that no centre of blowing up includes points over a 2-curve in $X$ where $\rho = 0$ (in particular,
every centre lies over the locus $(\log \rk \s = 0)$).
\end{lemma}

The proof of Lemma \ref{lem:rhofinite} following involves repeated blowings-up of 2-curves and 3-points. The
last assertion of the lemma is important in the case of (not necessarily compact) analytic varieties
because it means that only relatively compact 2-curves will be blown up, and this will imply that
only finitely many blowings-up will be needed over a given 2-curve in $X$. In fact, Lemma 
\ref{lem:rhofinite} involves only finitely many blowings-up over each point of the discrete subset 
$\Gamma$ of $X_0$ given by the image of $(\log \rk \s = 0)$.

\begin{proof}[Proof of Lemma \ref{lem:rhofinite}]
It is clear that $\rho(a) < \infty$ at every 1-point $a$. It follows from Lemma \ref{lem:rho} that 
$\rho$ is generically zero on every component of the set of 1-points. 

Recall we can assume that $\log \rk \s$ is constant (either $0$ or $1$) on every 2-curve in $X$,
and $\rho =0$ on every 2-curve in $X$ where $\log \rk \s = 1$.

Three-points are isolated. It is also clear that,
by combinatorial blowings-up, we can reduce to the case that $\rho(a) = 0$ at every 3-point $a$. (The blowings-up involved have centres that are 3-points or closures of 2-curves, but it is unnecessary to
blow up 2-curves that are already resolved; i.e., on which $\rho =0$.)
It follows that $\rho$ is generically zero on every 2-curve with a 3-point $a$ as a limiting point.

Suppose that $a$ is a 2-point. If $\rho(a) < \infty$ at a 2-point $a$, then $\rho$ is generically zero on the
2-curve containing $a$, again by Lemma \ref{lem:rho}.
A blowing-up with centre given by the intersection of two components of $E$ is
combinatorial. If $\rho(a) = \infty$, then we can reduce to the case that $\rho < \infty$ over $a$ by finitely
many such blowings-up (since the effect of such blowings-up is to principalize the ideal generated by the coefficients
$a_{ij}(u_1,u_2)$ of formal expansions $T_i = \sum_{j=0}^\infty a_{ij}(u_1,u_2)v^j$ at $a$;
cf. Lemma \ref{lem:rho}). It therefore follows that we can reduce to $\rho < \infty$ on every 2-curve, by
a locally-finite sequence of combinatorial blowings-up.
\end{proof}

\medskip\noindent
{\bf Step 2.} Reduction to prepared normal form.

\medskip
The following lemma \ref{lem:prepnormal} will be proved in Section \ref{sec:prepnorm}. 

\begin{lemma}[Prepared normal form]\label{lem:prepnormal}
Suppose that $\rho(a) < \infty$, for all $a\in X$ (and that $\rho = 0$ on every 2-curve with non-compact 
closure; cf. Step 1 above).
Assume that $\rho$ takes a maximum value $\rho_{\max} > 0$,
and let $\Sigma \subset X$ denote the closed subset on
which $\rho$ takes the value $\rho_{\max}$ (so that $\Sigma \subset \supp E$). 
Then, by a (locally) finite sequence of admissible 
blowings-up over $(\log \rk \s = 0)$, we can reduce to the case that,
for every point $a\in \Sigma$, $\s$ has the Weierstrass form of Lemma \ref{lem:rho},
where the coefficients $a_{ij}$ satisfy the following additional conditions.
\begin{enumerate}
\item At a 2-point $a$, with adapted coordinates $(\bu,v)=(u_1,u_2,v)$, where $\supp E = (u_1u_2=0)$,
\begin{equation}\label{eq:norm2}
\begin{alignedat}{2}
a_{ij} &= \bu^{\br_{ij}}\ta_{ij}(\bu), \quad &&i = 2,\ldots,N,\ \ j=1,\ldots,d-1,\\
a_{i_0,0} &= \bu^{\bbe}, &&\text{for some } i_0,
\end{alignedat}
\end{equation}
where each $\ta_{ij}$ is either zero or a unit, $\bde + \bbe$ is linearly independent of $\bal$,
and $\bu^{\bbe}$ divides $a_{i0}$, for all $i$.

\medskip
\item At a 1-point $a$, with adapted coordinates $(u,v,w)$, where $\supp E = (u=0)$,
\begin{equation}\label{eq:norm1}
\begin{alignedat}{2}
a_{ij} &= u^{r_{ij}}w^{s_{ij}}\ta_{ij}(u,w) \quad &&i = 2,\ldots,N,\ \ j=1,\ldots,d-1,\\
a_{i_0,0} &= u^{\be}w, &&\text{for some } i_0,
\end{alignedat}
\end{equation}
where each $\ta_{ij}$ is either zero or a unit, and $u^\be$ divides $a_{i0}$, for all $i$.
\end{enumerate}
The blowings up involved do not increase the value of $\rho$ over any point.
\end{lemma}

We will say that $a$ is a \emph{prepared} 2-point 
(resp., a \emph{prepared} 1-point) if $\s$ has Weierstrass form \eqref{eq:weier} at $a$,
where the coefficients are given by \eqref{eq:norm2} (resp., \eqref{eq:norm1}), with respect to suitable
adapted coordinates at $a$. In either case, we will also say that $\s$
has \emph{prepared normal form} at $a$.
At a \emph{generic} prepared 1-point, all $s_{ij} = 0$ in \eqref{eq:norm1}. Points where not all $s_{ij} = 0$
will be called \emph{non-generic}.

\begin{remark}\label{rem:transverse}
Suppose that $\rho_{\max} > 0$. Then all points of the maximum locus $\Sigma$ of $\rho$ are unresolved and, 
if $\s$ has prepared normal form at every point of $\Sigma$, then (the closure of) 
any curve of 1-points in $\Sigma$ has only normal crossings with respect to 2-curves.
\end{remark}

\medskip\noindent
{\bf Step 3.} Further admissible blowings-up to decrease the maximal value of the invariant $\rho$.

\medskip
Lemma \ref{lem:decrho} following is the subject of Section \ref{sec:decrho}. Theorem \ref{thm:dim3} then follows by
induction on the maximal value of $\rho$. 

\begin{lemma}\label{lem:decrho} Assume that $\rho(a)<\infty$, for all $a \in X$, and that $\rho$ takes a maximum
value $\rho_{\max} > 0$ on $X$. Let $\Sigma \subset X$ denote the maximum locus of $\rho$.  
Suppose that $\s$ has prepared 
normal form at every point $a \in \Sigma$ (see Lemma \ref{lem:prepnormal}). Let $A$ denote the discrete set of all 
non-generic points of $\Sigma$ (i.e., all 2-points and non-generic 1-points of $\Sigma$). Then there is a morphism
$\tau: (\widetilde{X},\widetilde{E}) \to (X,E)$ given by a locally finite sequence of admissible blowings-up over $\Sigma$
such that $\rho < \rho_{\max}$ throughout $\tX$. Moreover, $\tau$ can be realized as a composite 
$\tau = \tau_3\circ \tau_2\circ \tau_1$, where
\begin{enumerate}
\item $\tau_1: (X_1,E_1) \to (X,E)$ is a single blowing-up with centre $A$;
\item $\tau_2:(X',E') \to (X_1,E_1)$ is the composite of a locally finite sequence of admissible blowings-up 
$(X_{i+1},E_{i+1}) \to (X_i,E_i)$, $i\geq 1$, with centres 
$\overline{\Sigma_i \setminus A_i}$, where 
$\Sigma_i$ is the maximum locus of $\rho$ and $A_i$ the preimage of $A$ in $X_i$;
\item $\tau_3: (\widetilde{X},\widetilde{E}) \to (X',E')$ is the composite of a locally finite sequence of 
blowings-up with centres over $A$.
\end{enumerate}
\end{lemma}

\begin{proof}[Proof of Theorem \ref{thm:dim3}] 
The theorem follows from Lemmas \ref{lem:rhofinite}, \ref{lem:prepnormal} and \ref{lem:decrho} 
by induction on the maximal value of $\rho$,
at least for $X$ on which $\rho$ assumes a maximum value (e.g., algebraic varieties or the restrictions of 
analytic varieties to relatively compact open sets). Theorem \ref{thm:dim3} follows for analytic varieties,
in general, because the preceding lemmas show that, if we start with a resolution of singularities $\s$ as at
the beginning of this section, then $\rho$ can be everywhere decreased to zero
by finitely many blowings-up over each point of the discrete subset $\Gamma$ of $X_0$ given
by the image of $(\log \rk \s = 0)$.
\end{proof}

\section{Prepared normal form}\label{sec:prepnorm}

In this section, we prove Lemma \ref{lem:prepnormal}. The proof is by induction on
pairs $(\rho(a),\io(a))$ (ordered lexicographically), where $\io(a)$ is a secondary invariant with values
in $\IN$, introduced in the following subsection. 

\subsection{Secondary invariant}\label{subsec:io}
Suppose that $\s$ has Weierstrass form \eqref{eq:weier} in adapted coordinates
$(\bu,v)$ at a 2-point $a$; respectively, in adapted coordinates $(u,v,w)$ at a 1-point $a$. 
For each $i=2,\ldots,N$, write
\begin{align*}
d\s_1|_{(v=0)}\wedge d\s_i|_{(v=0)} &= d(\bu^{\bal})\wedge d(\bu^{\bde}a_{i0}(\bu))\\
                                                         &= H_i(\bu) \frac{du_1}{u_1} \wedge \frac{du_2}{u_2};
\end{align*}
respectively,      
\begin{align*}
d\s_1|_{(v=0)}\wedge d\s_i|_{(v=0)} &= d(u^{\al})\wedge d(u^{\de}a_{i0}(u,w))\\    
                                                         &= H_i(u,w) \frac{du}{u} \wedge dw   .
\end{align*}                                            
In either case, let $\cH_a$ denote the ideal generated by $H_i$,  $i=2,\ldots,N$, in the local ring
of ${(v=0)}$ at $a$.

\begin{remarks}\label{rem:H}
$\cH_a$ is the log Fitting ideal of $2\times 2$ minors of the morphism $\s|_{(v=0)}$ at $a$.
Blowing up of the point $a$ in ${(v=0)}$ is admissible for $E|_{(v=0)}$. If $\cH_a$ is a principal
monomial ideal (i.e., generated by a monomial in components of $E|_{(v=0)}$), then,
by Theorem \ref{thm:fit0} and Lemma \ref{lem:fitHP}, $\s|_{(v=0)}$ 
can be written in Hsiang-Pati form $d\s_1|_{(v=0)} = d(\bu^{\bal})$ (resp., $d(u^{\al})$),
and $a_{i_0,0} = \bu^{\bbe}$ (resp., $u^{\be}w$), for some $i_0$, where $\bu^{\bbe}$
(resp., $u^{\be}$) satisfies the additional conditions given in Lemma \ref{lem:prepnormal}.                                                                                                
\end{remarks}

Let $\cG_a$ denote the ideal 
$$
\cG_a := \Bigg(\prod_{(i,j)\in J}a_{ij}\Bigg)\cdot \cH_a,
$$
where $J:= \{(i,j): a_{ij}\neq 0,\, i=2,\ldots,N,\, j=1,\ldots,d-1\}$. 

\begin{remarks}\label{rem:G}
(1) Since $\cG_a$ is an ideal of functions in two variables, it follows from resolution of singularities that,
after a finite number $\io(a; \bu)$ (respectively,
$\io(a;(u,w))$) of blowings-up of discrete sets (beginning
with $a$ and) lying over $a$, the pull-back of $\cG_a$ is a principal ideal generated by a monomial
$\tbu^{\bg}$ (resp., $\tu^p \tw^q$) with respect to adapted coordinates $\tbu$ (resp., $(\tu,\tw)$) at
any point over $a$. (At each step, the centre of blowing-up is the finite set of points over $a$ at which
the pull-back of $\cG_a$ is not already a principal ideal generated by such a monomial.) Note that the
blowing-up of $(v=0)$ with centre $a$ corresponds to the blowing-up of $X$ with centre
$(\bu = \bzero)$ (resp., $(u=w=0)$).

\medskip
(2) In particular, multiplication of $\cG_a$ by a monomial in components of the exceptional
divisor does not change the value of $\io(a; \bu)$ (resp., $\io(a;(u,w))$).
\end{remarks}

\begin{definition}\label{def:io}
Let $\io(a)$ denote the minimum of $\io(a; \bu)$ (respectively,
$\io(a;(u,w))$) over all adapted coordinate systems $(\bu,v)$ (resp., $(u,v,w)$) at $a$ in which $\s$
has Weierstrass form.
\end{definition}

Of course, $\cH_a$ and $\cG_a$ themselves depend on the coordinates in Weierstrass form.

\begin{lemma}\label{lem:io=0}
If $\io(a) = 0$, then $\s$ has prepared normal form at $a$.
\end{lemma}

\begin{proof}
Assume that $\io(a) = 0$. Then we can choose adapted local coordinates in which each coefficient
$a_{ij}$, $j>0$, in \eqref{eq:weier} is a monomial times a unit as in \eqref{eq:norm2} (resp., \eqref{eq:norm1}),
and $\cH_a$ is a principal ideal generated by a monomial $\bu^{\bg}$ (resp., $u^p w^q$) as in Remarks
\ref{rem:G}. In the 1-point case, necessarily $q=0$ since $\s|(v=0)$ has rank $2$ outside $\supp E$ (i.e.,
since the Fitting ideal of $3 \times 3$ minors of $\log \Jac \s$ is supported in $E$). By Lemma
\ref{lem:fitHP}, we can choose coordinates also in which the coefficients $a_{i0}$ satisfy the conditions
of Lemma \ref{lem:prepnormal}.
\end{proof}

\begin{lemma}\label{lem:semicontin}
If $\io(a) \neq 0$, then there is a neighbourhood of $a$ in which $(\rho(b), \io(b)) < (\rho(a), \io(a))$,
$b\neq a$.
\end{lemma}

\begin{proof}
This is clear.
\end{proof}

\subsection{Proof of Lemma \ref{lem:prepnormal}}\label{subset:proofprepnormal}
The lemma will be given by an algorithm presented in three distinct cases, beginning
with $\s$ in Weierstrass form as in Lemma \ref{lem:rho}:
\begin{itemize}
\item $a$ is a 2-point;
\item $a$ is a 1-point with $\ord_a (T_i) = d(a)$ (where $(T_i)$ denotes the ideal generated
by $T_2,\ldots,T_N$ and $\ord_a$ means the order at $a$);
\item $a$ is a 1-point where $\ord_a (T_i) < d(a)$.
\end{itemize}
The third case is the most delicate.

\subsection{Case that $a$ is a 2-point}\label{subsec:2pt}

\begin{lemma}\label{lem:2pt} Let $a\in X$ be a 2-point. Suppose that $\rho(a) > 0$, $\io(a)>0$ and 
$\s$ is in Weierstrass form at $a$ (Lemma \ref{lem:rho}). Let $C$ denote the $2$-curve through $a$.
Then $C \subset (\log \rk \s = 0)$. Let 
$\tau: (\widetilde{X},\widetilde{E}) \to (X,E)$ denote the combinatorial blowing-up with centre $C$.
Then $(\rho(b),\io(b)) < (\rho(a),\io(a))$, for all $b \in \tau^{-1}(a)$.
\end{lemma}

\begin{proof}
Consider $\s$ in the Weierstrass form of Lemma \ref{lem:rho} in adapted coordinates $(\bu,v)$, 
where $d=d(a)=\rho(a)+1$ and $\io(a) = \io(a;\bu)$. Since $\rho(a) > 0$, $\log \rk_a \s = 0$.
Then $\log \rk \s = 0$ in a neighbourhood of $a$ in $C$ (by \eqref{eq:prelim}), and therefore on $C$.
 
Let $b\in \tau^{-1}(a)$. If $b$ is a 2-point,
then, without loss of generality, there are adapted coordinates $(\bx,z)=(x_1,x_2,z)$ at $b$, 
with $\tE=(x_1x_2=0)$, in which $\tau$ is given by
$$
u_1=x_1,\quad u_2 = x_1x_2,\quad v=z.
$$
If $b$ is a 1-point, then we can assume that $\al_2 \neq 0$, and there are adapted coordinates $(x,y,z)$ such that
$$
u_1 = x (\eta+y),\quad u_2 = x (\eta+y)^{-\al_1/\al_2},\quad v = z,
$$
where $\eta\neq 0$.
Thus, if $b$ is a 2-point (resp., 1-point), we can write
\begin{equation*}
\begin{aligned}
\tau^*\sigma_1 &= \bx^{\tbal},\\
\tau^*\sigma_i &= \tg_i(\bx)  + \bx^{\tbde} \tau^*T_i,
\end{aligned}
\qquad\quad \text{resp.,}\qquad\quad  
\begin{aligned}
\tau^*\sigma_1 &= x^{\tal},\\
\tau^*\sigma_i &= \tg_i(x,y)  + x^{\tde} \widetilde{U} \tau^*T_i,
\end{aligned}
\end{equation*}
$i=2,\ldots,N$, where $\tbal = (\al_1+\al_2,\al_2)$, $\tbde=(\de_1+\de_2,\de_2)$ (resp.,
$\tal= \al_1+\al_2$, $\tde=\de_1+\de_2$, and $\widetilde{U} = (\eta+y)^{\de_1 - \de_2 \al_1 / \al_2}$ is a unit).

In either case,
$$
 \tau^*T_i = z^d \tau^*\tT_i + \sum^{d-1}_{j=0} b_{ij} z^j,\quad i=2,\ldots,N,
$$
where $b_{ij}= \tau^*a_{ij}$ and $b_{ij} = b_{ij}(\bx)$ (resp., $b_{ij}= b_{ij}(x,y)$). Clearly, in
either case, $\rho(b)\leq \rho(a)$. 

Moreover, in either case it follows from Lemma \ref{lem:fit}(1) that $\cH_b = \tau^*\cH_a$, 
so that $\cG_b = \tau^*\cG_a$; therefore, $\io(b) < \io(a)$ and then $(\rho(b), \io(b)) < (\rho(a), \io(a))$, by
the definition of $\io$ (see Remarks \ref{rem:G}).
\end{proof}

\begin{remark}\label{rem:2pt} Lemma \ref{lem:prepnormal} in the case that $a$ is a 2-point thus
follows directly from resolution of singularities of the ideal $\cG_a$. If $a$ is a 1-point, then blowing-up
$(u=w=0)$ (with respect to adapted coordinates as in Lemma \ref{lem:rho}) likewise
gives $(\rho(b), \io(b)) < (\rho(a), \io(a))$, for $b\in \tau^{-1}(a)$. This can be used to prove a
local version of Lemma \ref{lem:prepnormal}, but the centre $(u=w=0)$ need not have a global meaning
in $X$. The challenge in \S\S\ref{subsec:1ptd},\,\ref{subsec:1pt<d} is to decrease the value of the
invariant $(\rho,\io)$ by \emph{global} blowings-up only.
\end{remark}

\subsection{Case that $a$ is a 1-point with $\ord_a (T_i) = d(a)$}\label{subsec:1ptd}

\begin{lemma}\label{lem:1pt1} Let $a\in X$ be a 1-point. Suppose that $\rho(a) > 0$, $\io(a)>0$, and 
$\s$ is in Weierstrass form at $a$ (Lemma \ref{lem:rho}), in adapted coordinates $(u,v,w)$, 
where $\ord_a(T_i)=d(a)=\rho(a)+1$ and $\io(a) = \io(a;(u,w))$. Let 
$\tau: (\widetilde{X},\widetilde{E}) \to (X,E)$ denote the blowing-up with centre 
$a$. Then $(\rho(b),\io(b)) < (\rho(a),\io(a))$, for all $b \in \tau^{-1}(a)$.
\end{lemma}

\begin{proof} We again write the components of $\s$ using the notation of Lemma \ref{lem:rho}.
We consider three cases, depending on the coordinate chart of $X$ containing $b$.

\medskip\noindent
\emph{Case I. The point $b$ belongs to the $u$-chart.} This chart has adapted coordinates $(x,\tv,\tw)$ in which
$\tau$ is given by 
$$
u=x,\quad v = x\tv,\quad w=x\tw,
$$
and $b \in (x=0)$; say, $b=(0,\nu,\om)$. Since $\ord_a (T_i) = d$,
\begin{align*}
\tau^*\s_1 &= x^\al,\\
\tau^*\s_i &= g_i(x) + x^{\de+d}\cdot\frac{\tau^*T_i}{x^d},\quad i=2,\ldots,N,
\end{align*}
and each
\begin{equation}\label{eq:u}
\frac{\tau^*T_i}{x^d} = \tv^d \tau^*\tT_i + \sum^{d-1}_{j=0} b_{ij}(x,\tw) \tv^j,
\end{equation}
where each $b_{ij} = \tau^*a_{ij}/x^{d-j}$. It is clear from \eqref{eq:u} that, if $\nu\neq 0$, then 
$d(b) \leq d(a)-1 < d(a)$ (recall that $\tT_2 = U$ is a unit). Assume that $\nu =0$. 
Then $d(b)\leq d(a)$, by \eqref{eq:u}.
Moreover, it follows from Lemma \ref{lem:fit}(1) that $\cH_b = x\cdot\tau^*\cH_a$, 
so that $\cG_b = x^q\cdot\tau^*\cG_a$,
where $q = 1 - \sum_{(i,j)\in J}(d-j)$; therefore, $\io(b) < \io(a)$ and $(\rho(b), \io(b)) < (\rho(a), \io(a))$.

\medskip\noindent
\emph{Case II. The point $b$ belongs to the $w$-chart, but not to the $u$-chart.} The $w$-chart has
adapted coordinates $(\tu,x,\tv)$, where $\tE = (\tu=x=0)$, in which
$$
u = x\tu,\quad v= x\tv,\quad w=x,
$$
and $b=(0,0,\nu)$. Similarly to Case I, 
\begin{align*}
\tau^*\s_1 &= x^\al\tu^\al,\\
\tau^*\s_i &= g_i(x\tu) + \tu^\de x^{\de+d} \cdot\frac{\tau^*T_i}{x^d},\quad i=2,\ldots,N,
\end{align*}
and
$$
\frac{\tau^*T_i}{x^d} = \tv^d \tau^*\tT_i + \sum^{d-1}_{j=0} b_{ij}(\tu,x) \tv^j,
$$
where $b_{ij} = \tau^*a_{ij}/x^{d-j}$. If $\nu\neq 0$, then $d(b) \leq d(a)-1 < d(a)$. Assume $\nu=0$.
Then $d(b)\leq d(a)$, and $\cH_b = x\cdot\tau^*\cH_a$, by Lemma \ref{lem:fit}(1). Again,
$\cG_b = x^q\cdot\tau^*\cG_a$, with $q$ as in Case I, and $(\rho(b), \io(b)) < (\rho(a), \io(a))$.

\medskip\noindent
\emph{Case III. The point $b$ belongs to the $v$-chart, but not to the $u$- or $w$-charts.} The $v$-chart
has adapted coordinates $(\tu,x,\tw)$, where $\tE = (\tu=x=0)$, in which
$$
u = x\tu,\quad v= x,\quad w=x\tw,
$$
and $b=(0,0,0)$. In the same way as before,
\begin{align*}
\tau^*\s_1 &= x^\al \tu^\al\\
\tau^*\s_i &= g_i(x\tu) + \tu^\de x^{\de+d} \cdot\frac{\tau^*T_i}{x^d},\quad i=2,\ldots,N,
\end{align*}
and
$$
\frac{\tau^*T_i}{x^d} = \tau^*\tT_i + \sum^{d-1}_{j=0} b_{ij}(\tu,x,\tw),
$$
where each $b_{ij}(\tu,x,\tw) = a_{ij}(x\tu,x\tw)/x^{d-j}$; in particular, all $b_{ij}(0,0,0) = 0$.
Hence $\tau^*T_2/x^d$ is a unit at $b$ and, moreover, $(\de,\de+d)$ is linearly independent of $(\al,\al)$
(since $d=d(a)\neq 0$). Therefore, $d(b) = 0$, so that $\rho(b) = 0$.
\end{proof}

\subsection{Case that $a$ is a 1-point with $\ord_a (T_i) < d(a)$}\label{subsec:1pt<d}
Let $a\in X$ be a 1-point. Suppose that $\rho(a) > 0$, $\io(a)>0$, and 
$\s$ is in Weierstrass form at $a$ (Lemma \ref{lem:rho}), in adapted coordinates $(u,v,w)$, 
where $\ord_a(T_i) < d(a)=\rho(a)+1$ and $\io(a) = \io(a;(u,w))$. Set $\mu:= \ord_a(T_i)$.
We rewrite \eqref{eq:weier} as
\begin{equation}\label{eq:weiermu}
T_i = \sum_{k=\mu}^{d-1} u^{\al_{ik}} P_{ik}(u,v,w) + \left(\tT_i v^d  + \sum^{d-1}_{j=0} c_{ij}(u,w) v^j\right),
\quad i=2,\ldots,N,
\end{equation}
where $c_{2,d-1}=0$, $\ord_a c_{ij}\geq d-j$, for all $i,\,j$, and $u^{\al_{ik}} P_{ik}$ is a homogeneous polynomial
of degree $k$ such that either $P_{ik}=0$ or $P_{ik}(0,v,w) \neq 0$ (i.e., $P_{ik}$ is not divisible by $u$),
for all $i,\,k$.

\begin{remark}\label{rem:alpha}
For all $i,\,k$ such that $P_{ik}\neq 0$, we have $0 < \al_{ik} < k$. The left-hand inequality is clear from the definition
of $\rho(a)$. On the other hand, if  $\al_{ik} = k$, then $u^{\al_{ik}} P_{ik}(u,v,w) = u^k P_{ik}(0,0,0)$ is a monomial
of degree $< d$; since $\al_{ik}$ is trivially linearly dependent on $\al$, this would contradict the definition of 
Weierstrass form.
\end{remark}

Let $\fm_{X,a}$ denote the maximal ideal of $\cO_{X,a}$ and let $I := \{(i,k): P_{ik}\neq 0\}$. Let $\cJ$ denote
the ideal
$$
\cJ := \fm_{X,a}^{d} + \sum_{(i,k) \in I} u^{\alpha_{ik}} \fm_{X,a}^{k-\alpha_{ik}}.
$$
Clearly, $V(\cJ) = \{a\}$.

\begin{lemma}\label{lem:1pt2}

Let $a\in X$ be a 1-point. Suppose that $\rho(a) > 0$, $\io(a)>0$, and 
$\s$ is in Weierstrass form at $a$ (Lemma \ref{lem:rho}), in adapted coordinates $(u,v,w)$, 
where $\mu = \ord_a(T_i)<d(a)=\rho(a)+1$ and $\io(a) = \io(a;(u,w))$. Then there is a morphism
$\tau: (\widetilde{X},\widetilde{E}) \to (X,E)$ given by a sequence of admissible blowings-up that
principalizes $\cJ$, such that
\begin{enumerate}
\item $\tau$ is an isomorphism over $X\setminus \{a\}$;
\item $(\rho(b),\io(b)) < (\rho(a),\io(a))$, for all $b \in \tau^{-1}(a)$.
\end{enumerate}
\end{lemma}

\begin{proof} Let $\tau_1: (X_1,E_1)\to  (X,E)$ denote the blowing-up with centre the point $a$, and let
$\tau_2: (\widetilde{X},\widetilde{E}) \to (X_1,E_1)$ denote a morphism given by admissible blowings-up
that principalizes $\tau_1^*\cJ$. Set $\tau=\tau_1\circ\tau_2$. Clearly, $\tau$ is an isomorphism over 
$X\setminus \{a\}$. Let $b\in \tau^{-1}(a)$. 

We write the components of $\s$ using the notation of \eqref{eq:weiermu}, and again consider three cases,
depending on the coordinate chart containing $\tau_2(b)$.

\medskip\noindent
\emph{Case I. The point $\tau_2(b)$ belongs to the $u$-chart of $\tau_1$.}
This chart has adapted coordinates $(x,\tv,\tw)$ in which $\tau_1$ is given by 
$$
u=x,\quad v = x\tv,\quad w=x\tw,
$$
and $\tau_2(b) \in (x=0)$. It follows that $\tau_1^*\cJ$ is a principal monomial ideal in this chart, 
so that $\tau_2$ is an isomorphism over the chart. Since $\ord_a (T_i)\geq \mu$,
\begin{align*}
\tau_1^*\s_1 &= x^\al,\\
\tau_1^*\s_i &= g_i(x) + x^{\de+\mu}\cdot\frac{\tau_1^*T_i}{x^\mu},\quad i=2,\ldots,N,
\end{align*}
and each
\begin{equation}\label{eq:u2}
\frac{\tau_1^*T_i}{x^\mu} = \sum_{k=\mu}^{d-1} x^{k-\mu}Q_{ik}(1,\tv,\tw) 
+ x^{d-\mu} \left(\tv^d \tau_1^*\tT_i + \sum^{d-1}_{j=0} d_{ij}(x,\tw) \tv^j\right),
\end{equation}
where each $Q_{ik} = \tau_1^*P_{ik}/x^{k-\al_{ik}}$ and each $d_{ij} = \tau_1^*c_{ij}/x^{d-j}$. 
Since $\mu < d$, there exists $i_0$ such that $Q_{i_0,\mu}\neq 0$, and
$$
\frac{\tau_1^*T_{i_0}}{x^\mu} = Q_{i_0,\mu}(1,\tv,\tw) + x R(x,\tv,\tw),
$$
for some $R$. By Remark \ref{rem:alpha}, $Q_{i_0,\mu}$ is a non constant polynomial of
degree $< \mu$. Therefore, $d(\tau_2(b)) \leq \deg\, Q_{i_0,\mu} < d(a)$.

\medskip\noindent
\emph{Case II. The point $\tau_2(b)$ belongs to the $w$-chart, but not to the $u$-chart.} The $w$-chart has
adapted coordinates $(\tu,\tw,\tv)$, where $\tE = (\tu=\tw=0)$, in which
$$
u = \tu\tw,\quad v= \tv\tw,\quad w=\tw,
$$
and $\tau_2(b)\in V(\tu,\tw)$. Similarly to Case I, 
\begin{equation}\label{eq:w2}
\begin{aligned}
\tau_1^*\s_1 &= \tu^\al\tw^\al,\\
\tau_1^*\s_i &= g_i(\tu\tw) + \tu^\de \tw^{\de+\mu} \cdot\frac{\tau_1^*T_i}{\tw^\mu},\quad i=2,\ldots,N,
\end{aligned}
\end{equation}
and each

\begin{equation}\label{eq:w2a}
\frac{\tau_1^*T_i}{\tw^\mu} = \sum_{k=\mu}^{d-1} \tu^{\al_{ik}}\tw^{k-\mu}Q_{ik}(\tu,\tv) 
+ \tw^{d-\mu} \left(\tv^d \tau_1^*\tT_i + \sum^{d-1}_{j=0} d_{ij}(\tu,\tw) \tv^j\right),
\end{equation}
where each $Q_{ik} = \tau_1^*P_{ik}/\tw^{k-\al_{ik}}$ and each $d_{ij} = \tau_1^*c_{ij}/\tw^{d-j}$. 
Then, for each $i,k$, either $Q_{ik}=0$ or there exists $j_{ik} < k$ such that:
\begin{equation}\label{eq:w2expansion}
Q_{ik}(\tu,\tv) = \tv^{j_{ik}}Q_{i,j_{ik},k} + \sum_{j=0}^{j_{ik}-1}\tv^jQ_{ijk} +\tu R_{ik}(\tu,\tv)  
\end{equation}
where $Q_{ij_{ik}k}$ is a nonzero constant.

Moreover, $\tau_1^*\cJ$ is the ideal
$$
\tau_1^*\cJ = \tw^{\mu}\cdot\left(\tw^{d-\mu};\, \tu^{\al_{ik}}\tw^{k-\mu},\, (i,k) \in I\right),
$$
and principalization of $\tau_1^*\cJ$ is equivalent to principalization of $\cK := \tw^{-\mu}\cdot \tau_1^*\cJ$.
Since $\tau_1^*\cJ$ is generated by finitely many exceptional monomials in two variables,
$\tau_2: (\tX,\tE) \to (X_1,E_1)$ is a composite of combinatorial blowings-up, over the $w$-chart. 
We consider two subcases, depending on whether $b$ is a 2-point or a 1-point; each of these
subcases will be divided into further subcases, depending on which generator of $\cK$
pulls back to a generator of the principal ideal $\tau_2^*\cK$.

\medskip\noindent
\emph{Subcase II.1. The point $b$ is a 2-point.} There are adapted coordinates $(\bx,z) = (x_1,x_2,z)$
centred at $b$ such that $\tE = (x_1x_2=0)$ and
$$
\tu = \bx^{\bla_{1}},\quad \tw = \bx^{\bla_{2}},\quad \tv = \zeta + z,
$$
where $\bla_{1}, \bla_{2}$ are $\IQ$-linearly independent. By \eqref{eq:w2},
\begin{equation}\label{eq:w2.1}
\begin{aligned}
\tau^*\s_1 &= \bx^{\al(\bla_{1}+\bla_{2})},\\
\tau^*\s_i &= \tg_i(\bx) + \bx^{\de\bla_{1} + (\de+\mu)\bla_{2}} \cdot\frac{\tau^*T_i}{\bx^{\mu\bla_{2}}},\quad i=2,\ldots,N.
\end{aligned}
\end{equation}

\medskip\noindent
\emph{Subcase II.1.1. The ideal $\tau_2^*\cK$ is generated by $\tau_2^*(\tw^{d-\mu})
= \bx^{\bbe}$, where $\bbe = (d-\mu)\bla_{2}$.}
Then
$$
\frac{\tau^*T_i}{\bx^{d\bla_{2}}} = (\zeta + z)^d\tau^*\tT_i + \sum_{j=0}^{d-1}b_{ij}(\bx)(\zeta+z)^j, \quad i=2,\ldots,N,
$$
where $b_{ij} = \tau^*a_{ij}/\bx^{(d-j)\bla_{2}}$.

Clearly, if $\zeta\neq0$, then $d(b) \leq d-1 < d(a)$. Suppose that $\zeta = 0$. Then $d(b) \leq d(a)$. Moreover,
by Lemma \ref{lem:fit}(1),\,(2), $\cH_b = \bx^{\bla_{2}}\cdot\tau^*\cH_a$, so that 
$$
\cG_b = \bx^{\bla_{2}}\prod_{(i,j)\in J}\frac{1}{\bx^{\bbe+(\mu-j)\bla_{2}}}\cdot\tau^*\cG_a;
$$
therefore, $\io(b) < \io(a)$ and $(\rho(b),\io(b))<(\rho(a),\io(a))$.

\medskip\noindent
\emph{Subcase II.1.2. There exists $i_0$ and $k_0$ (where $\mu \leq k_0 < d$) such that
$\tau_2^*\cK$ is generated by $\tau_2^*(\tu^{\al_{i_0, k_0}}\tw^{k_0-\mu}) = \bx^{\bbe}$, where
$\bbe = \al_{i_0, k_0}\bla_{1}+(k_0-\mu)\bla_{2}$.} By \eqref{eq:w2a},
$$
\frac{\tau^*T_{i_0}}{\bx^{\bbe+\mu\bla_{2}}} = \tQ_{i_0,k_0}(\bx,z) + \tR(\bx,z),
$$
where $\tQ_{i_0,k_0} = \tau_2^*Q_{i_0,k_0}$ and $\tR(0,z) = 0$. By \eqref{eq:w2expansion},
$$
\tQ_{i_0,k_0} = z^{j_0}Q_{i_0, j_0, k_0} + \sum_{j=0}^{j_0-1} z^j Q_{i_0, j, k_0} + \bx^{\bla_{1}}R_{i_0,k_0}(\bx,z),
$$
where $j_0 = j_{i_0, k_0}$. If $j_0 > 0$, then $d(b) \leq j_0 < k_0 < d(a)$. On the other hand, if $j_0 = 0$, then,
by \eqref{eq:w2.1},
$$
\tau^*\s_{i_0} = \tg_{i_0}(\bx) + \bx^{(\de+\al_{i_0,k_0})\bla_{1} + (\de+k_0)\bla_{2}} \cdot
\frac{\tau^*T_{i_0}}{\bx^{\bbe+\mu\bla_{2}}};
$$
since $\al_{i_0,k_0} < k_0$ and $\bla_{1},\, \bla_{2}$ are linearly independent, we conclude
that $d(b) = 0 < d(a)$. This completes Subcase II.1.

\medskip\noindent
\emph{Subcase II.2. The point $b$ is a 1-point.} There are adapted coordinates $(x,y,z)$ centred at $b$ such that
$\tE=(x=0)$ and
$$
\tu = x^{\la_1}(\eta+y)^{-1},\quad \tw = x^{\la_2}(\eta+y),\quad \tv = \zeta+ z,
$$
where $\eta\neq 0$. By \eqref{eq:w2},
\begin{equation}\label{eq:w2.2}
\begin{aligned}
\tau^*\s_1 &= x^{\al(\la_1+\la_2)},\\
\tau^*\s_i &= \tg_i(x,y) + x^{\de\la_1 + (\de+\mu)\la_2} \cdot\frac{\tau^*T_i}{x^{\mu\la_2}},\quad i=2,\ldots,N.
\end{aligned}
\end{equation}

\medskip\noindent
\emph{Subcase II.2.1. The ideal $\tau_2^*\cK$ is generated by $\tau_2^*(\tw^{d-\mu}) = 
x^{(d-\mu)\la_2}(\eta+y)^{d-\mu}$.} Then
$$
\frac{\tau^*T_i}{x^{d\la_2}} = (\zeta + z)^d(\eta + y)^d\tau^*\tT_i + \sum_{j=0}^{d-1}b_{ij}(x,y)(\zeta+z)^j, \quad i=2,\ldots,N,
$$
where $b_{ij} = \tau_2^*(\tw^ja_{ij})/x^{d\la_2}$. Since $\eta\neq 0$, it is clear that, if $\zeta\neq 0$, then
$d(b) \leq d-1 < d(a)$. Suppose that $\zeta = 0$. Then $d(b) \leq d(a)$. As above,
$\cH_b = x^{\la_2}\cdot\tau^*\cH_a$, and
$$
\cG_b = x^{\la_2}\prod_{(i,j)\in J}\frac{1}{x^{\be+(\mu-j)\la_2}}\cdot\tau^*\cG_a,
$$
so that $(\rho(b),\io(b))<(\rho(a),\io(a))$.

\medskip\noindent
\emph{Subcase II.2.2. The ideal $\tau_2^*\cK$ is not generated by $\tau_2^*(\tw^{d-\mu})$.} Then there
exists $(i,k) \in I$ such that $\tau_2^*(\tu^{\al_{ik}}\tw^{k-\mu})$ generates $\tau_2^*\cK$. Set
$$
\La := \{(i,k)\in I: \tau_2^*(\tu^{\al_{ik}}\tw^{k-\mu}) \text{ generates } \tau_2^*\cK\}.
$$
Note that, if $(i,k_1),(i,k_2) \in \La$, then $\la_1\al_{i,k_1} + \la_2(k_1-\mu) = \la_1\al_{i,k_2} + \la_2(k_2-\mu) = \be$, say,
so that
\begin{equation}\label{eq:pairs}
\al_{i,k_1} = \al_{i,k_2} + (k_2 - k_1)\frac{\la_2}{\la_1}.
\end{equation}
From \eqref{eq:w2.2}, we rewrite
\begin{equation*}
\tau^*\s_i = \tg_i(x,y) + x^{\de\la_1 + (\de+\mu)\la_2 + \be} \cdot\frac{\tau^*T_i}{x^{\mu\la_2+\be}}.
\end{equation*}
For each $i$, let $\La_i := \{k: (i,k)\in  \La\}$. By \eqref{eq:w2a},
\begin{equation}\label{eq:w2avar}
\frac{\tau^*T_i}{x^{\mu\la_2+\be}} = \sum_{k\in \La_i} (\eta+y)^{k-\al_{ik}}(\tau_2^*Q_{ik})(x,z) + xR_i(x,y,z).
\end{equation}
We claim that, for each $i$, the exponents $k-\al_{ik}$ in \eqref{eq:w2avar} are all distinct. Indeed, if
$k_1-\al_{i,k_1} = k_2-\al_{i,k_2}$, where $(i,k_1),(i,k_2) \in \La$, $k_1\neq k_2$, then $\la_1=-\la_2$,
by \eqref{eq:pairs}; a contradiction since $\la_1,\la_2>0$.

By \eqref{eq:w2expansion}, for each $(i,k) \in \La$,
$$
(\tau_2^*Q_{ik})(x,z) = (\zeta+z)^{j_{ik}}Q_{i,j_{ik},k} + \sum_{j=0}^{j_{ik}-1} (\zeta + z)^j Q_{ijk} + xR_{ik}(x,y),
$$
where $Q_{i,j_{ik},k} \neq 0$. Set $j_0 := \max\{j_{ik}: (i,k) \in \La\}$ and
$$
\Ga := \{(i,k) \in \La: j_{ik} = j_0\}.
$$
We can rewrite \eqref{eq:w2avar} as
\begin{equation}\label{eq:w2avar1}
\frac{\tau^*T_i}{x^{\mu\la_2+\be}} = \sum_{k\in \Ga_i} z^{j_0} (\eta+y)^{k-\al_{ik}} Q_{i,j_0,k}
+ \sum_{j=0}^{j_0-1}z^jR_{ij}(x,y) + xR_i(x,y,z),
\end{equation}
where $\Ga_i := \{k: (i,k)\in \Ga\}$. 

Recall that the $k-\al_{ik}$, $k \in \Ga_i$ are distinct.
Choose $i_0$ such that $\Ga_{i_0} \neq \emptyset$.
We consider three subcases of Subcase II.2.2, depending on $j_0$.

\medskip\noindent
\emph{First suppose that $j_0 = 0$.} 
Choose $k_0$ such that 
$k_0-\al_{i_0,k_0} = \max\{k-\al_{ik}: k \in \Ga_{i_0}\}$. Then $0<k_0-\al_{i_0,k_0}<d$ and
$\p_y^{k_0-\al_{i_0,k_0}}\left(\tau^*T_{i_0}/x^{\mu\la_2+\be}\right)$ does not vanish at $b$;
therefore, $d(b) \leq k_0-\al_{i_0,k_0}<d(a)$.

\medskip\noindent
\emph{Secondly, suppose that $0 < j_0 \leq \mu$.} By \eqref{eq:w2avar1},
\begin{equation}\label{eq:sum}
\p_z^{j_0}\left(\frac{\tau^*T_{i_0}}{x^{\mu\la_2+\be}}\right) = j_0! \sum_{k\in \Ga_{i_0}} (\eta+y)^{k-\al_{i_0,k}} Q_{i_0,j_0,k}
+ xR_{i_0}(x,y,z).
\end{equation}
Since $\mu\leq k < d$ in \eqref{eq:sum}, there are at most $d-\mu$ terms in the sum, with distinct exponents. 
Hence there exists $l_0 < d-\mu$ such that $\p_y^{l_0}$ applied to the sum in \eqref{eq:sum} is a unit;
therefore, $\p_z^{j_0}\p_y^{l_0}\left(\tau^*T_{i_0}/x^{\mu\la_2+\be}\right)$ is a unit. It follows that
$d(b) \leq j_0 + l_0 \leq \mu + l_0 < \mu + d(a) - \mu = d(a)$.

\medskip\noindent
\emph{Finally, suppose that $\mu < j_0$.} If $k \in \Ga_{i_0}$, then $j_0 < k$. So the sum 
in \eqref{eq:sum} has at most $d-j_0$ terms, with distinct exponents. As above, there exists $l_0 < d-j_0$ 
such that $\p_y^{l_0}$ applied to this sum is a unit; therefore, 
$\p_z^{j_0}\p_y^{l_0}\left(\tau^*T_{i_0}/x^{\mu\la_2+\be}\right)$ is a unit, and
again $d(b) \leq j_0 + l_0 < j_0 + d(a) - j_0 = d(a)$.

\medskip
This completes Case II of the proof of Lemma \ref{lem:1pt2}. Case III following has a similar pattern.

\medskip\noindent
\emph{Case III. The point $\tau_2(b)$ belongs to the $v$-chart, but not to the $u$- or $w$-charts.}
The $v$-chart has
adapted coordinates $(\tu,\tv,\tw)$, where $\tE = (\tu=\tv=0)$, in which
$$
u = \tu\tv,\quad v= \tv,\quad w=\tv\tw,
$$
and $\tau_2(b)=0$. Similarly to Case II, 
\begin{equation}\label{eq:w3}
\begin{aligned}
\tau_1^*\s_1 &= \tu^\al\tv^\al,\\
\tau_1^*\s_i &= g_i(\tu\tv) + \tu^\de \tv^{\de+\mu} \cdot\frac{\tau_1^*T_i}{\tv^\mu},\quad i=2,\ldots,N,
\end{aligned}
\end{equation}
and each

\begin{equation*}
\frac{\tau_1^*T_i}{\tv^\mu} = \sum_{k=\mu}^{d-1} \tu^{\al_{ik}}\tv^{k-\mu}Q_{ik}(\tu,\tw) 
+ \tv^{d-\mu} \left(\tau_1^*\tT_i + \sum^{d-1}_{j=0} d_{ij}(\tu,\tv,\tw) \right),
\end{equation*}
where each $Q_{ik} = \tau_1^*P_{ik}/\tv^{k-\al_{ik}}$ and each $d_{ij} = \tau_1^*c_{ij}/\tv^{d-j}$. 
Then, for each $i,k$, either $Q_{ik}=0$ or there exists $j_{ik} < k$ such that:
\begin{equation*}
Q_{ik}(\tu,\tw) = \tw^{j_{ik}}Q_{i,j_{ik},k} + \sum_{j=0}^{j_{ik}-1}\tw^jQ_{ijk} +\tu R_{ik}(\tu,\tw)  
\end{equation*}
where $Q_{i,j_{ik},k}$ is a nonzero constant.

Moreover, $\tau_1^*\cJ$ is the ideal
$$
\tau_1^*\cJ = \tv^{\mu}\cdot\left(\tv^{d-\mu};\, \tu^{\al_{ik}}\tv^{k-\mu},\, (i,k) \in I\right),
$$
and principalization of $\tau_1^*\cJ$ is equivalent to principalization of $\cK := \tv^{-\mu}\cdot \tau_1^*\cJ$.
Since $\tau_1^*\cJ$ is generated by finitely many exceptional monomials in two variables,
$\tau_2: (\tX,\tE) \to (X_1,E_1)$ is a composite of combinatorial blowings-up, over the $v$-chart. 
As in Case II, we consider two subcases, depending on whether $b$ is a 2-point or a 1-point; each of these
subcases will be divided into further subcases, depending on which generator of $\cK$
pulls back to a generator of the principal ideal $\tau_2^*\cK$.

\medskip\noindent
\emph{Subcase III.1. The point $b$ is a 2-point.} There are adapted coordinates $(\bx,z) = (x_1,x_2,z)$
centred at $b$ such that $\tE = (x_1x_2=0)$ and
$$
\tu = \bx^{\bla_{1}},\quad \tv = \bx^{\bla_{2}},\quad \tw = z,
$$
where $\bla_{1}, \bla_{2}$ are $\IQ$-linearly independent. By \eqref{eq:w3},
\begin{align*}
\tau^*\s_1 &= \bx^{\al(\bla_{1}+\bla_{2})},\\
\tau^*\s_i &= \tg_i(\bx) + \bx^{\de\bla_{1} + (\de+\mu)\bla_{2}} \cdot\frac{\tau^*T_i}{\bx^{\mu\bla_{2}}},\quad i=2,\ldots,N.
\end{align*}

\medskip\noindent
\emph{Subcase III.1.1. The ideal $\tau_2^*\cK$ is generated by $\tau_2^*(\tv^{d-\mu})
= \bx^{(d-\mu)\bla_{2}}$.}
Then 
$$
\tau^*\s_2 = \tg_2(\bx) + \bx^{\de\bla_{1} + (\de+d)\bla_{2}} \cdot \left(\tau^*T_2 / \bx^{\mu\bla_{2}}\right)
$$
and
$$
\frac{\tau^*T_2}{\bx^{d\bla_{2}}} = \tau^*U + R(\bx,z),
$$
where $R(\bzero,z)=0$. Since $\bla_{1}, \bla_{2}$ are linearly independent and
$d\neq 0$, it follows that $d(b)=0<d(a)$, so $\rho(b)=0<\rho(a)$.

\medskip\noindent
\emph{Subcase III.1.2. There exists $i_0$ and $k_0$ (where $\mu \leq k_0 < d$) such that
$\tau_2^*\cK$ is generated by $\tau_2^*(\tu^{\al_{i_0, k_0}}\tv^{k_0-\mu})$.} In this subcase, we can show that
$d(b)<d(a)$ exactly as in Subcase II.1.2 above.

\medskip\noindent
\emph{Subcase III.2. The point $b$ is a 1-point.} There are adapted coordinates $(x,y,z)$ centred at $b$ such that
$\tE=(x=0)$ and
$$
\tu = x^{\la_1}(\eta+y)^{-1},\quad \tv = x^{\la_2}(\eta+y),\quad \tw = z,
$$
where $\eta\neq 0$. By \eqref{eq:w3},
\begin{align*}
\tau^*\s_1 &= x^{\al(\la_1+\la_2)},\\
\tau^*\s_i &= \tg_i(x,y) + x^{\de\la_1 + (\de+\mu)\la_2} \cdot\frac{\tau^*T_i}{x^{\mu\la_2}},\quad i=2,\ldots,N.
\end{align*}

\medskip\noindent
\emph{Subcase III.2.1. The ideal $\tau_2^*\cK$ is generated by $\tau_2^*(\tv^{d-\mu}) = 
x^{(d-\mu)\la_2}(\eta+y)^{d-\mu}$.} Then
$$
\frac{\tau^*T_2}{x^{d\la_2}} = (\eta + y)^d\tau^*U + \sum_{j=0}^{d-2}\tau^*a_{2j}x^{(j-d)\la_2}(\eta+y)^j, 
$$
and $\tau^*a_{2j} = a_{2j}(x^{\la_1+\la_2},x^{\la_2}(\eta+y)z)$. Since $\eta\neq0$, it is clear that
$d(b) \leq d(a)-1$.

\medskip\noindent
\emph{Subcase III.2.2. The ideal $\tau_2^*\cK$ is not generated by $\tau_2^*(\tv^{d-\mu})$.} Then there
exists $(i,k) \in I$ such that $\tau_2^*(\tu^{\al_{ik}}\tv^{k-\mu})$ generates $\tau_2^*\cK$. In this subcase,
we can show that $d(b)<d(a)$ precisely as in Subcase II.2.2 above.
\end{proof}

\medskip
This completes the proof of Lemma \ref{lem:prepnormal}.

\section{Decreasing the main invariant}\label{sec:decrho}

In this section, we prove Lemma \ref{lem:decrho} and thus complete the proof of Theorem  \ref{thm:dim3}. 
We continue to use the notation of Sections \ref{sec:outline} and \ref{sec:prepnorm}.
A sequence of blowings-up as in the conclusion of Lemma \ref{lem:decrho} will be called \emph{permissible}.

To prove Lemma \ref{lem:decrho}, we will show that, in general,
every prepared point $a \in \supp E$ admits a neighbourhood $U$ over which there is
a morphism $\tau: (\widetilde{U},\widetilde{E}) \to (U,E|_U)$ given by a permissible finite sequence of blowings-up
over $(\rho = \rho(a))$, such that $\rho(b) < \rho(a)$, for all $b \in \tU$ (see Lemmas \ref{lem:decrho1}, 
\ref{lem:decrho2}, \ref{lem:decrho3}, depending
on the nature of $a$). Lemma \ref{lem:decrho} clearly follows from this local statement, because the sequence of
blowings-up in Lemma \ref{lem:decrho}(2) is uniquely determined by the maximum value of $\rho$. Note that, if $a$ is a
generic prepared 1-point, then there is a neighbourhood $U$ of $a$ over which $A = \emptyset$, so that a
permissible blowing-up sequence means a finite sequence as in Lemma \ref{lem:decrho}(2).


\subsection{Declared local exceptional divisor}\label{subsec:declared} 
Suppose that $\sigma$ has prepared normal form \eqref{eq:norm2} in adapted coordinates $(\pmb{u},v)$ at a $2$-point $a$ (respectively, prepared normal form \eqref{eq:norm1} in adapted coordinates $(u,v,w)$ at a $1$-point $a$). Although
$v$ (and $w$) are not globally defined, there is a neighbourhood $U$ of $a$ in which $(v=0)$ (or $(v=0)$ and
$(w=0)$) are smooth hypersurfaces that we can add to $E$ to obtain a divisor $D$ on $U$. We consider $D$ a
``declared exceptional divisor''. The divisor $D$ and a corresponding monomial idea $\cI = \cI(\s,a)$ are defined according
to the nature of $a$, as follows.

\begin{definitions}[\emph{Declared exceptional divisor} and associated monomial ideal]
\label{def:declared} We use the notation of Lemma \ref{lem:prepnormal}.
\begin{itemize}
\item If $a$ is a prepared 2 point \eqref{eq:norm2},
\begin{align*}
D &:= E|_U +(v=0) = (u_1 u_2 v=0), \\
\mathcal{I} &:= (v^d,\,  \pmb{u}^{\pmb{r}_{ij}}v^j, \bu^{\bbe}),
\end{align*}
where $\bbe$ will also be denoted $\br_{i_0,0}$ (for reasons evident from 
\eqref{eq:norm2});
\smallskip
\item If $a$ is a generic prepared 1-point (\eqref{eq:norm1} with all $s_{ij}=0$),
\begin{align*}
D &:=  E|_U +(v=0) = (uv=0),\\
\mathcal{I} &:= (v^d,\,  u^{r_{ij}}v^j,\,u^{\beta}),
\end{align*}
where $\be$ will also be denoted $r_{i_0,0}$;
\smallskip
\item If $a$ is a non-generic prepared 1-point \eqref{eq:norm1}, 
\begin{align*}
D &:=  E|_U +(v=0) +(w=0) = (uvw=0),\\
\mathcal{I} &:= (v^d,\, u^{r_{ij}} w^{s_{ij}}v^j,\, u^{\beta}w),
\end{align*}
where again $\be$ will also be denoted $r_{i_0,0}$.
\end{itemize}
The notation for $\cI$ in each case above is understood to mean that $(i,j)$ runs over
the index set $J := \{(i,j): a_{ij}\neq 0,\, i=2,\ldots,N,\, j=1,\ldots,d-1\}$.
\end{definitions}

\begin{remarks}\label{rem:declared}
(1) $D$ and $\mathcal{I}(\s,a)$ depend on the adapted coordinates at $a$. 
Nevertheless, if $U$ is a small open neighbourhood of $a$ and $\rho(b) = \rho(a)$, where $b \in U$,
then $D$ induces a declared exceptional divisor at $b$, and $\cI(\s,a)$ induces $\cI(\s,b)$.

\medskip\noindent
(2) A blowing-up that is admissible for $D$ is also admissible for $E|_U$. Suppose that
$\tau: \tU \to U$ is given by a finite sequence of blowings-up that are admissible for $D$. 
If $\tD,\,\tE$ denote the
transforms of the divisors $D,\, E|_U$, respectively, then $\tD = \tE$ $+$ strict transform
of $(v=0)$ (or $\tD =\tE$ $+$ strict transforms of $(v=0)$ and $(w=0)$ in the non-generic prepared 1-point case).
\end{remarks}

\begin{proposition}\label{prop:declared}
Let $a\in X$ and suppose that $\sigma$ has prepared normal form at $a$. Take $U,\, D$ and $\cI$ as in Definitions
\ref{def:declared}.  Let $\tau: (\widetilde{U},\widetilde{E}) \to (U,E|_U)$ be a sequence of blowings-up with centres over
$\supp E|_U$ that are combinatorial with respect to $D$, and let $b \in \tau^{-1}(a)$. If $\tau^*(\mathcal{I})$ is 
a principal $\tD$-monomial ideal, then $\tau^*(\mathcal{I})$ is also $\tE$-monomial, and $\rho(b) < \rho(a)$.
\end{proposition}

Note that the hypotheses of Proposition \ref{prop:declared} does not exclude $(v=w=0)$ as centre of blowing up (in the 
non-generic 1-point case). Proposition \ref{prop:declared} is a purely local assertion; we will not claim to principalize 
$\cI(\s,a)$ by blowings-up that are global over $X$. The proposition plays an important part in the proof of Lemma \ref{lem:decrho} but, in the latter, we do not necessarily principalize $\cI(\s,a)$ at every point $b$ over a given 
$a \in \Sigma$ because $\rho$ may decrease before $\cI$ becomes principal.

The proof of Proposition \ref{prop:declared} is a case-by-case analysis which we leave to the end;
we first complete the proof of Lemma \ref{lem:decrho}.

\subsection{Permissible sequences of blowings-up} As indicated above, in order to prove 
Lemma \ref{lem:decrho},
it is enough to show that every prepared
point $a\in  \supp E$ admits a neighbourhood $U$ over which the invariant $\rho$ can be decreased by a permissible
finite sequence of blowings-up. This will be done separately in the case that $a$ is a generic 1-point, a 2-point or
a non-generic 1-point, in the three lemmas \ref{lem:decrho1}, \ref{lem:decrho2}, \ref{lem:decrho3} following. In each 
of these lemmas, $U$ denotes a (relatively compact) open neighbourhood of $a$ in which $\rho \leq \rho(a)$, and
$\Sigma := \{x\in U: \rho(x) = \rho(a)\}$. We assume that $U$ is small enough that the declared exceptional
divisor $D$ can be defined as in \S\ref{subsec:declared}, and we use the notation of Lemma \ref{lem:prepnormal} and
Definitions \ref{def:declared}.

\begin{lemma}\label{lem:decrho1}
Let $a$ be a generic 1-point of $\Sigma$. Then there is a finite sequence of permissible blowings-up $\tau: (\widetilde{U},\widetilde{E}) \to (U,E|_U)$ which is combinatorial with respect to $D$, such that $\rho(b)<\rho(a)$ for all 
$b \in \widetilde{U}$. Moreover, the weak transform of $\cI$ by $\tau$ is principal except perhaps on 
$\widetilde{E} \cap \widetilde{H}$, where $\widetilde{H}$ is the strict transform of $H := (v=0)$.
(The \emph{weak transform} means the residual ideal after factoring out the exceptional divisor as much
as possible.) 
\end{lemma}

\begin{proof} For brevity, we write $E$ instead of $E|_U$. Note that the unique permissible centre of
blowing up in $U$ is $V(u,v) = (u=v=0)$. Let $\mu(a) := \min\{\be, r_{ij}+j\}$. 

We first show that we can reduce to the case $\mu(a) < d$. Suppose that $\mu(a) \geq d$.
Consider the blowing-up $\tau_1: (U_1,E_1) \to (U,E)$ with centre $V(u,v)$, and let $b \in \tau_1^{-1}(a)$. 
There are two possibilities:
\begin{enumerate}
\item $b$ belongs to the $v$-chart,  with coordinates $(\tu,\tv,\tw)$ in which $\tau_1$ is given by $u = \tu\tv,\, v=\tv,\, w=\tw$.
(The strict transform of $H$ does not intersect this chart.) 
Then $\tau_1^*\cI$ is the principal ideal generated by $\tv^d$ at $b$; therefore, $\rho(b) < \rho(a)$, by
Proposition \ref{prop:declared}. 
\smallskip
\item $b$ is the origin of the $u$-chart, with coordinates $(x,\tv,\tw)$ in which $\tau_1$ is given by
$u=x,\, v=x\tv,\, w=\tw$. Then
\[
\tau_1^{\ast}T_i = x^d\left(\tilde{v}^d + \sum_{i=1}^{d-1} x^{r_{ij}+j -d }\tilde{v}^j\tau_1^{\ast}\tilde{a}_{ij} 
+ x^{\beta-d} \tau_1^{\ast}\tilde{a}_{i0} \right).
\]
Clearly, $\rho(b) \leq \rho(a)$ and $\mu(b) < \mu(a)$ if $\rho(b) = \rho(a)$.
\end{enumerate}

We therefore assume that $\mu(a) < d$. Again let $\tau_1: (U_1,E_1) \to (U,E)$ denote the blowing-up with 
centre $V(u,v)$, let $b \in \tau_1^{-1}(a)$, and consider the two coordinate charts as above.
\begin{enumerate}
\item Suppose that $b$ belongs to the $u$-chart. If $b\neq 0$, then $\tv \neq 0$ at $b$, so that
$\tau_1^*\cI$ is a principal ideal generated by a monomial in $x$, and $\rho(b) < \rho(a)$, by
Proposition \ref{prop:declared}. 
On the other hand, if $b=0$, then
\[
\tau_1^{\ast}T_i = x^{\mu(a)}\left(\tilde{v}^dx^{d-\mu(a)} + \sum_{i=1}^{d-1} x^{r_{ij}+j -\mu(a)}\tilde{v}^j\tau_1^{\ast}\tilde{a}_{ij}
+ x^{\beta-\mu(a)} \tau_1^{\ast}\tilde{a}_{i0}\right),
\]
and clearly $\rho(b) < \rho(a)$. We remark that, in this case, the weak transform of $\cI$ is supported in
$\widetilde{E} \cap \widetilde{H}$.

\smallskip
\item Otherwise, $b$ is the origin of the $v$-chart, where $E_1 = (\tilde{u}\tilde{v}=0)$ and 
\begin{equation}\label{eq:excideal}
\tau_1^{\ast}\mathcal{I}  = \tv^{\mu(a)}\left( \tilde{v}^{d-\mu(a)},\, \tilde{u}^{r_{ij}}\tilde{v}^{r_{ij}+j-\mu(a)},\,
\tilde{u}^{\beta}\tilde{v}^{\beta-\mu(a)} \right).
\end{equation}
\end{enumerate}

Now, let $\cJ$ denote the monomial ideal on $U_1$ determined by the right hand side of \eqref{eq:excideal},
and let $\tau_2$ be a sequence of combinatorial blowings-up such that $\tau_2^*\mathcal{J}$ is a principal monomial
ideal. Clearly, $\tau = \tau_2 \circ \tau_1$ is a sequence of combinatorial blowings-up with respect to $F$, and the weak transform of $\mathcal{I}$ by $\tau$ is supported on $\widetilde{E} \cap \widetilde{H}$. Furthermore, if $c \in \tau^{-1}(a)$,
then $\rho(c) < \rho(a)$, either by the preceding case (1) or by Proposition \ref{prop:declared}. 

Consider a sequence of blowings-up
$\overline{\tau}: (\overline{U},\overline{E}) \to (U_1,E_1)$ that are combinatorial 
with respect to $E$ (e.g., part of the sequence $\tau_2$). Let $b$ be the origin of the $v$-chart in the preceding
case (2), and let $c \in \overline{\tau}^{-1}(b)$ be a point where 
$\overline{\tau}^*\mathcal{J}$ is not principal. We will show that $\rho(c) = \infty$. 

This implies that we can principalize $\cJ$ by combinatorial blowings-up given at each step by the
maximal locus of $\rho$; in fact, the locus $(\rho = \infty)$: If $\cJ$ is not principal at $b$, then $(\rho = \infty)$
is given by the 2-curve $(\tu = \tv = 0)$ in \eqref{eq:excideal}. After each blowing-up with centre $(\rho = \infty)$,
the new locus $(\rho = \infty)$ is a disjoint union of analogous 2-curves defined in the various coordinate charts.
In other words, $\cJ$ can be principalized by a sequence $\tau_2$ of permissible blowings-up.

To show that $\rho(c) = \infty$ above: 
Since $\overline{\tau}^{\ast}\mathcal{J}$ is not principal, $c$ must be a 2-point of $\overline{E}$. Therefore,
there are coordinates $(\bx,y)= (x_1,x_2,y)$ at $c$ in which $\overline{\tau}$ us given by
\[
\tu =  x_1^{\kappa_1}x_2^{\kappa_2}, \quad \tv= x_1^{\la_1}x_2^{\la_2},\quad \tilde{w}=y,
\]
so that
\[
\overline{\tau}^{\ast}\tilde{v}^d = \pmb{x}^{\pmb{\bar{r}}_d}, \quad \overline{\tau}^{\ast}\tilde{u}^{r_{ij}}\tilde{v}^{r_{ij}+j}
= \pmb{x}^{\pmb{\bar{r}}_{ij}} , \quad \overline{\tau}^{\ast}\tilde{u}^{\beta} = \pmb{x}^{\pmb{\bar{\beta}}},
\]
with suitable exponents. Since $(\kappa_1,\kappa_2)$ and $(\lambda_1,\lambda_2)$ are linearly independent, these exponents are distinct for fixed $i$, and all except $\pmb{\bar{\beta}}$ are 
linearly independent of $\pmb{\bar{\alpha}}$, where
$\overline{\tau}^{\ast}\sigma_1 = \pmb{x}^{\pmb{\bar{\alpha}}}$. Let $\mathcal{K}$ denote the ideal generated by the coefficients of the formal expansions with respect to $y$ of $\overline{\tau}^{\ast}T_i$, $i=2, \ldots, N$.
Note that $\mathcal{K} \subset \overline{\tau}^{\ast}\mathcal{J}$, and that each monomial 
$\overline{\tau}^{\ast}\tilde{v}^d$, $\overline{\tau}^{\ast}\tilde{u}^{r_{ij}}\tilde{v}^{r_{ij}+j}$ in 
$\overline{\tau}^{\ast}\mathcal{J}$ appears in the expansion of the constant term (i.e., the coefficient of $y^0$)
of the formal expansion of $\overline{\tau}^{\ast}T_i$, for some $i$. Moreover, $\overline{\tau}^{\ast}\tilde{u}^{\beta}$ appears 
in the coefficient of $y$ in $\overline{\tau}^{\ast}T_{i_0}$. It follows that $\mathcal{K}$ is principal if and only if  
$\overline{\tau}^{\ast}\mathcal{J}$ is principal. Therefore, $\rho(c) = \infty$, as claimed.
\qed

\begin{lemma}\label{lem:decrho2}
Let $a$ be 2-point of $\Sigma$. Then there is a finite sequence of permissible blowings-up 
$\tau: (\widetilde{U},\widetilde{E}) \to (U,E)$ which is combinatorial with respect to $D$, such that $\rho(b)<\rho(a)$ 
for all $b \in \widetilde{U}$. Moreover, $\tau$ can be realized as a composite 
$\tau = \tau_3\circ \tau_2\circ \tau_1$, where $\tau_1$ is a single blowing-up with centre $a$, $\tau_2$
is the composite of a finite sequence of permissible blowings-up as in Lemma \ref{lem:decrho}(2), and $\tau_3$ is
the composite of a finite sequence of blowings-up over $a$.
\end{lemma}

\begin{proof}
The proof consists of four steps.

\medskip\noindent
\emph{Step 1. Decreasing $\rho$ outside the preimage of $a$.} Let $W := U \setminus \{a\}$. Then every point of
$\Sigma \cup W$ is a generic prepared 1-point. (If $a$ is an isolated point of $\Sigma$, then the proof of the
lemma reduces essentially to Step 3 below.) By Lemma $\ref{lem:decrho1}$, there is a finite sequence 
$\tilde{\tau}$ of permissible blowings-up
$\tilde{\tau} : (\widetilde{W},\widetilde{E}) \to (W, E|_W)$
such that $\rho(b) < \rho(a)$ for all $b \in \widetilde{W}$. Moreover, all blowings-up involved are combinatorial with
respect to $D|_W$ and, after a first blowing-up with centre $a$
(to separate curves that will be blown up simultaneously according to Lemma \ref{lem:decrho1}), 
we can take the closures of all the centres of blowings-up that comprise 
$\tilde{\tau}$, to get a sequence 
$$
\ttau_1: (U_1,E_1) \to (U,E)
$$ 
of permissible blowings-up over $U$,
such that $\rho(b) < \rho(a)$ for all $b$ outside $\ttau_1^{-1}(a)$. Each centre of blowing up is a union
(necessarily disjoint) of closures of curves given by Lemma \ref{lem:decrho1}). Moreover, all blowings-up 
are combinatorial with 
respect to $D$. In particular, $\ttau_1$ also makes sense as a morphism $\ttau_1 : (U_1,D_1) \to (U,D)$.

Set $\mathcal{I}_1 := \ttau_1^{\ast}\mathcal{I}$. Let $H_1$ denote the strict transform of $H := (v=0)$ by $\ttau_1$,
and $\mathcal{J}_1$ the residual ideal sheaf of $\cI_1$ after factoring out the greatest possible monomial 
in local generators of components of the exceptional divisor $E_1$. By Lemma \ref{lem:decrho1}, 
$V(\mathcal{J}_1) \subset \ttau_1^{-1}(a) \cup (E_1 \cap H_1)$. Let $\Gamma_1$ denote the closure of
$V(\mathcal{J}_1)\setminus \ttau^{-1}_1(a)$. Then $\Gamma_1$ is an analytic (or regular) curve (unless it is
empty) with at most two connected components $\g_1^{(k)}$ ($k=1$ or $k=1,2$), each of which intersects
$\ttau^{-1}_1(a)$ in a point $p_1^{(k)}$. If there are two components, then they were separated by the first 
blowing-up (with centre $a$) of the sequence $\ttau_1$, so that $p_1^{(1)},\, p_1^{(2)}$ are distinct.

\medskip\noindent
\emph{Step 2. Decreasing $\rho$ at the limit points of the 1-curve(s) $\g_1^{(k)}$.} 
Let $p_1$ denote either of the points $p_1^{(k)}$,
and $\g_1$ the corresponding curve $\g_1^{(k)}$. Assume that $\cI_1$ is not principal at $p_1$.
Then there is a coordinate system $(\bx,y) = (x_1,x_2,y)$ at $p_1$ in which $\ttau_1$ is given by:
\[
u_1 = \pmb{x}^{\pmb{\lambda}_1} ,\quad u_2 = \pmb{x}^{\pmb{\lambda}_2}  ,\quad v= \pmb{x}^{\pmb{\lambda}_3} y,
\]
where $\pmb{\lambda}_1$, $\pmb{\lambda}_2$ are $\mathbb{Q}$-linearly independent, $E_1 = (x_1x_2=0)$ 
and $D_1 = (x_1x_2y=0)$. We can assume that $(x_1=0)$ is the component of $E_1$ that does not project 
to $a$, so that $\g_1 = V(x_1,y)$. In particular, in this coordinate neighbourhood of $p_1$, the ideal $\cI_1$
has the form
\begin{equation}\label{eq:I1}
\cI_1 = ( \pmb{x}^{\pmb{\tilde{r}}_{2d}}y^d, \pmb{x}^{\pmb{\tilde{r}}_{ij}}y^j,\pmb{x}^{\pmb{\tilde{\beta}}})
\end{equation}
for suitable $\pmb{\tilde{r}}_{2d}$, $\pmb{\tilde{r}}_{ij}$ and $\pmb{\tilde{\beta}}$. (Recall that the
first monomial in $\cI_1$ in \eqref{eq:I1} comes from $T_2$, and the last comes from $T_{i_0}$.) Write
$\pmb{\tilde{r}}_{i_0,0}= \pmb{\tilde{\beta}}$ (cf. Definitions \ref{def:declared}). Each $\pmb{\tilde{r}}_{ij}$ is a pair
$\pmb{\tilde{r}}_{ij} =(r_{ij1},r_{ij2})$. 

Let $m$ denote the minimum of $r_{ij1}$ over all $(i,j)$ corresponding to
monomials in \eqref{eq:I1}, and let $\Lambda := \{(i,j): r_{ij1}= m\}$. Take $(i_1,j_1) \in \Lambda$ with minimal
$j_1$. Since $d(b) < d(a)=d$ for $b$ outside $\ttau^{-1}(a)$ (cf. notation of Lemma \ref{lem:rho}), 
we see that $j_1  < d(a)$. We can also assume that $r_{i_1,j_1,2}$ is minimal over all pairs $(i,j)$; indeed, we
can blow up $p_1 = (0,0,0)$ once, and then:
\begin{enumerate}
\item After a further sequence of blowings-up of curves of the form $(x_2=y=0)$ (which are combinatorial  
with respect to $D_1$ and project to $a$), we can assume that $r_{i_1,j_1,2} \leq r_{ij2}$, for all $(i,j) \in \Lambda$. 
(This does not change the values of $r_{ij1}$.)

\smallskip
\item After a further sequence of blowings-up of curves of the form $(x_1 = x_2 =0)$ (which are combinatorial in respect to 
$D_1$ and project to $a$), we can assume that $r_{i_1,j_1,2} \leq r_{ij2}$, for all $(i,j) \notin \Lambda$. 
(Again this does not change the values of $r_{ij1}$.)
\end{enumerate} 

The above construction applies to $p_1 = p_1^{(k)}$, $k=1$ or $k=1,2$, and provides a sequence of permissible 
blowings-up
$$
\ttau_2 : (U_2,E_2) \to (U_1,E_1)
$$
(where the first blowing-up in the sequence has centre $p_1^{(1)} \cup p_1^{(2)}$ if $k=1,2$).
Let $D_2$, $\mathcal{I}_2$, $\mathcal{J}_2$, $p_2^{(k)}$ and $\gamma_2^{(k)}$ denote the objects
defined after $\ttau_2$ that are analogous to $D_1$, $\mathcal{I}_1$, $\mathcal{J}_1$, $p_1^{(k)}$ and 
$\gamma_1^{(k)}$. Clearly, at $p_2 = p_2^{(k)}$, the ideal $\mathcal{I}_2$ has the form
\begin{equation}\label{eq:I2}
\cI_2 =  \pmb{x}^{\pmb{\tilde{r}}_{i_1,j_1}} ( y^{j_1}, \pmb{x}^{\pmb{r}_{ij}}y^j, \pmb{x}^{\pmb{\beta}} ),
\end{equation}
(where $\gamma_2 = V(x_1,y)$)
for suitable $\pmb{r}_{ij}$, $\pmb{\beta}$, and it follows that $\rho(p_2) <\rho(a)$. In particular, if $j_1 = 0$, 
then $\mathcal{I}_2$ is principal at $p_2$ and $\rho(p_2)=0$.

\medskip\noindent
\emph{Step 3. Decreasing $\rho$ over $a$, outside the preimage(s) of the limit point(s) $p_2^{(k)}$.}
Let $W$ denote the complement of the curve(s) $\gamma_2^{(k)}$ in $U_2$. Then the ideal 
$\mathcal{J}_2\cdot \mathcal{O}_{W}$ has support in the preimage of $a$. There is a sequence 
$\ttau: (\widetilde{W},\widetilde{E}) \to (W,E_2|_W)$ of blowings-up that are combinatorial
with respect to $D_2|_W$, which principalizes
$\mathcal{J}_2\cdot\mathcal{O}_{W}$. Then $\ttau$ is permissible (all centres project to $a$), and $\rho < \rho(a)$
throughout $\widetilde{W}$, by Proposition \ref{prop:declared}.
Since all blowings-up involved are combinatorial with respect to $D_2|_W$, we can take the closures of all 
centres to get a sequence of permissible blowings-up
\[
\ttau_3 : (U_3,E_3) \to (U_2,E_2),
\]
where $\rho(b) < \rho(a)$ for every $b$ not in the preimage of $p_2^{(k)}$, and all blowings-up are combinatorial 
with respect to the divisor $D_2$. 

We define $D_3$, $\mathcal{I}_3$, $\mathcal{J}_3$, $p_3^{(k)}$ and $\gamma_3^{(k)}$ in the same way as
before. The precise form of the ideal $\mathcal{I}_3$ at $p_3^{(k)}$ is important for the next step, so let us compute it. 
Again let $p_3$ denote either point $p_3^{(k)}$ and let $\g_3 = \g_3^{(k)}$. By construction (see \eqref{eq:I2}), 
the centres of all blowings-up in $\ttau_3$ containing $p_2$ 
(or the analogous limit point after blowing up) are 
of the form $V(y,x_2)$. Therefore, there is a coordinate system $(u_1,u_2,v)$ at $p_3$ in which $\ttau_3$ is given by
\[
x_1 = u_1 ,\quad x_2 = u_2  ,\quad y= u_2^{\lambda}v.
\]
It follows that, in these coordinates, $\g_3 = V(u_1,v)$ and $\mathcal{I}_3$ has the form
\[
\cI_3 = \pmb{u}^{\pmb{\delta}} (u_2^{r_{j_1}}v^{j_1}, \pmb{u}^{\pmb{r}_{ij}}v^{j} , \pmb{u}^{\pmb{\beta}}),
\]
for suitable exponents $\pmb{\delta}$, $r_{j_1}$, $\pmb{r}_{ij}$ and $\pmb{\beta}$
(where we use notation unchanged from before for simplicity). Since $\mathcal{I}_3$ is principal outside 
the preimage of $p_2$ (in the preimage of $a$), we have $\beta_2 =0$ and 
\begin{equation}\label{eq:I3}
\cI_3 = \pmb{u}^{\pmb{\delta}} (u_2^{r_{j_1}}v^{j_1}, \pmb{u}^{\pmb{r}_{ij}}v^{j}, u_1^{\beta_1}).
\end{equation}

\medskip\noindent
\emph{Step 4. Decreasing $\rho$ in the preimage(s) of the point(s) $p_2^{(k)}$.}
Let $W$ denote the complement of the curve(s) $\gamma_3^{(k)}$ in $U_3$. Then the ideal 
$\mathcal{J}_3\cdot \mathcal{O}_{W}$ has support in the preimage of $a$. There is a permissible sequence 
$\ttau: (\widetilde{W},\widetilde{E}) \to (W,E_3|_W)$ of blowings-up that are combinatorial with
respect $D_3|_W$, which principalizes $\mathcal{J}_3\cdot \mathcal{O}_{W}$. By Proposition \ref{prop:declared},
$\rho < \rho(a)$ throughout $\widetilde{W}$. Since all blowings-up are combinatorial with respect to $D_3|_W$,
we can take the closures of all centres to get a sequence of permissible blowings-up 
\[
\ttau_4 : (U_4,E_4) \to (U_3,E_3),
\]
where $\rho(b) < \rho(a)$ for every $b$ not in the preimage of $p_3^{(k)}$, and all blowings-up are combinatorial 
with respect to $D_3$. 

We define $\mathcal{I}_4$ as before. We claim that $\rho < \rho(a)$ in the preimage of $p_3 = p_3^{(k)}$, as required
to finish the proof. By \eqref{eq:I3}, the centres of all blowings-up in $\ttau_4$ containing $p_3$ (or the analogous
limit point after blowing up) are of the form 
$V(u_1,u_2)$, and they must principalize the ideal
\[
\mathcal{K} = ( u_2^{r_{j_1}}, \pmb{u}^{\pmb{r}_{ij}}, u_1^{\beta_1}).
\]
Let $b$ denote a point in the preimage of $p_3$. There are two possible cases.

\medskip\noindent
\emph{Case I. $b$ is a 2-point.} Then there are coordinates $(x_1,x_2,y)$ at $b$ in which $\ttau_4$ is given by
\[
u_1 = \pmb{x}^{\pmb{\lambda}_1} ,\quad u_2 = \pmb{x}^{\pmb{\lambda}_2}  ,\quad v=y.
\]
Since $\ttau_4^{\ast}\mathcal{K}$ is principal, there exists $j_2 \leq j_1$ such that 
\[
\cI_4 = \pmb{u}^{\pmb{\tilde{\delta}}} ( y^{j_2}, \pmb{u}^{\pmb{\tilde{r}}_{ij}}y^{j}, \pmb{u}^{\pmb{\tilde{\beta}}}),
\]
for suitable $\pmb{\tilde{\delta}}$, $\pmb{\tilde{r}}_{ij}$, $\pmb{\tilde{\beta}}$. Therefore, $\rho(b)<\rho(a)$. 
In particular, if $j_2 = 0$, then $\mathcal{I}_4$ is principal.

\medskip\noindent
\emph{Case II. $b$ is a 1-point.} Then there are coordinates $(x,y,z)$ at $b$ in which $\ttau_4$ is given by
\[
u_1 = x^{\lambda_1} ,\quad u_2 = x^{\lambda_2}(\zeta + z),\quad v=y.
\]
Since $\ttau_4^{\ast}\mathcal{K}$ is principal, there exists $j_2 \leq j_1$ such that 
\[
\cI_4 = x^{\tilde{\delta}} ( y^{j_2}, x^{\tilde{r}_{ij}}y^{j}, x^{\tilde{\beta}})
\]
for suitable $\tilde{\delta}$, $\tilde{r}_{ij}$, $\tilde{\beta}$. Therefore, $\rho(b)<\rho(a)$. In particular, if $j_2 = 0$, 
then $\mathcal{I}_4$ is principal.

\end{proof}

\begin{lemma}\label{lem:decrho3}
Let $a$ be a non-generic 1-point of $\Sigma$. Then there is a finite sequence of permissible blowings-up 
$\tau: (\widetilde{U},\widetilde{E}) \to (U,E)$ which is combinatorial with respect to $D$, such that $\rho(b)<\rho(a)$ 
for all $b$ in $\widetilde{U}$. Moreover, $\tau$ can be realized as a composite 
$\tau = \tau_3\circ \tau_2\circ \tau_1$, where $\tau_1$ is a single blowing-up with centre $a$, $\tau_2$
is the composite of a finite sequence of permissible blowings-up as in Lemma \ref{lem:decrho}(2), and $\tau_3$ is
the composite of a finite sequence of blowings-up over $a$.
\end{lemma}

\begin{proof}
The proof consists of four steps, as for Lemma \ref{lem:decrho2}.

\medskip\noindent
\emph{Step 1. Decreasing $\rho$ outside the preimage of $a$.} Let $\g$ and $\de$ denote the curves
$(v=u=0)$ and $(v=w=0)$, respectively. Then $\g$ coincides with $\Sigma$, since $\de \setminus \{a\}$
lies outside $\supp E$. We first blow up with centre $a$ to separate $\g$ and $\de$, and define the
morphism $\ttau_1$ as in the proof of Lemma \ref{lem:decrho2} (so that $\ttau_1$ consists of blowings-up
over $\g$).

Set $\mathcal{I}_1 := \ttau_1^{\ast}\mathcal{I}$. Let $H_1$ and $K_1$ denote the strict transforms of 
$H := (v=0)$ and $K := (w=0)$ by $\ttau_1$ (respectively), and let
$\mathcal{J}_1$ be the residual ideal sheaf of $\cI_1$ after factoring out the greatest possible monomial 
in local generators of components of the exceptional divisor $E_1$. By Lemma \ref{lem:decrho1}, 
$V(\mathcal{J}_1) \subset \ttau_1^{-1}(a) \cup (E_1 \cap H_1) \cup (K_1 \cap H_1)$. The closure
$\Gamma_1$ of $V(\mathcal{J}_1)\setminus \ttau^{-1}_1(a)$ is a union of two curves $\g_1 = E_1 \cap H_1$
(analogous to $\g_1^{(1)}$ in the proof of Lemma \ref{lem:decrho2}) and $\de_1 = K_1 \cap H_1$, which intersect
$\ttau^{-1}_1(a)$ in distinct points $p_1$ and $q_1$.

\medskip\noindent
\emph{Step 2. Decreasing $\rho$ at the limit points $p_1,\, q_1$ of the curves $\g_1,\, \de_1$.} 
For $p_1$, we can repeat the argument of Lemma \ref{lem:decrho2}, Step 2.

Consider $q_1$. First note that, if $s_{ij} = 0$ for some $i,j$ (i.e., if $w$ does not appear in the monomial
part of the coefficient of $v^j$, for some $i$ and some $j< d$), then we can again repeat the argument
of Lemma \ref{lem:decrho2}, Step 2, thinking of $(u,w)$ here as $(u_2, u_1)$ in Lemma \ref{lem:decrho2}.
In this case, we can finish the proof of Lemma \ref{lem:decrho3} as in Lemma \ref{lem:decrho2}.

It may, however, happen that $s_{ij} > 0$ for all $i,j$. In this case, after finitely many blowings-up of $q_1$ and
its preimages in successive liftings of $\de_1$, we get $\de_2$ and $q_2$ with the property that there
is a coordinate system $(x,y,z)$ at $q_2$, such that $E_2 = (x=0)$, $\de_2 = V(y,z)$ and 
$$
\cJ_2 = (x^{r_d}y^d,\, x^{\tilde{r}_{ij}}z^{s_{ij}}y^j,\, z) = (x^{r_d}y^d,\, z),
$$
for suitable $r_d,\, \tilde{r}_{ij}$, and it follows that $\rho(q_2) = 0$. (We define $D_2$, $\cI_2$ and $\cJ_2$ as
in Lemma \ref{lem:decrho2}.)

\medskip\noindent
\emph{Step 3. Decreasing $\rho$ over $a$, outside the preimages of $p_2,\, q_2$.} As in the proof of
Lemma \ref{lem:decrho2}, we define $\ttau_3 : (U_3,E_3) \to (U_2,E_2)$,
where $\rho(b) < \rho(a)$ for every $b$ not in the preimages of $p_2, q_2$. We define 
$D_3$, $\mathcal{I}_3$, $\mathcal{J}_3$, $p_3,\, q_3$, $\gamma_3$ and $\de_3$ as before.
At $p_3$, we can compute $\cI_3$ or $\cJ_3$ as in the proof of Lemma \ref{lem:decrho2}.

At $q_3$, there is a coordinate system $(u,v,w)$ in which $\ttau_3$ is given by
\[
x=u,\quad y=v,\quad z = u^{r_d}w,
\]
and in which $\de_3 = V(v,w)$ and 
\begin{equation}\label{eq:NG3}
\cJ_3 = (v^d, w).
\end{equation}
It follows that $\rho(q_3) =0$ and that, if $W$ denotes the complement of the curves $\gamma_3$, $\de_3$
in $U_3$, then $\cJ_3\cdot\cO_W$ has support disjoint from a neighbourhood $V$ of $\de_3$.

\medskip\noindent
\emph{Step 4. Decreasing $\rho$ in the preimages of $p_2,\, q_2$.} We define
$\ttau_4 : (U_4,E_4) \to (U_3,E_3)$ as in the proof of Lemma  
\ref{lem:decrho2}. Then, as in the latter, $\rho < \rho(a)$ outside the preimages of $p_3,\ q_3$,
and, moreover, $\rho < \rho(a)$ in the preimage of $p_3$. Since $\ttau_4$ is an isomorphism over
$V$, we have $\rho < \rho(a)$ on $U_4$.
\end{proof}

This completes the proof of Lemma \ref{lem:decrho}.

\subsection{Proof of Proposition \ref{prop:declared}}
Consider $\tau: (\widetilde{U},\widetilde{D}) \to (U,D)$. In each case in Definitions \ref{def:declared},
$\supp \cI$ has codimension at least two; therefore, $\cosupp \tau^*(\cI) \subset \supp \tE$
and $\tau^*(\cI)$ is $\tE$-principal. 
The proof of the proposition will be divided in three cases depending on the nature of $a$.

\medskip\noindent
\emph{Case I. $a$ is a generic 1-point.}  Then $D = (uv = 0)$. We consider two cases depending on whether 
$b$ is a 1- or 2-point \emph{with respect to} $\widetilde{D}$:

\medskip\noindent
\emph{Subcase I.1. $b$ is a 1-point of $\widetilde{D}$.} There are adapted coordinates $(x,y,z)$ at $b$, 
where $\widetilde{D} = (x=0)$ and $\tau$ is given by
$$
u =x^{\lambda_1}, \quad v = x^{\lambda_2}(\zeta + z),\quad  w=y,
$$
where $\zeta \neq 0$. Clearly, $\widetilde{E} = \widetilde{D}$ at $b$, since $\supp \widetilde{E} = V(\tau^*(u))$.
By $\eqref{eq:weier}$, $\eqref{eq:norm1}$, each
\begin{equation}\label{eq:I.1}
\tau^{\ast}T_i = x^{\tilde{r}_d} (\zeta+z)^d \overline{T}_i + \sum_{j=1}^{d-1}  x^{\tilde{r}_{ij}}  (\zeta+z)^j \tb_{ij} + b_{i0},
\end{equation}
for suitable $\tilde{r}_d$, $\tilde{r}_{ij}$, where $\tb_{ij} = \tb_{ij}(x,y) = \tau^*\ta_{ij}$, 
$\overline{T}_i = \tau^*\widetilde{T}_i$ and 
$b_{i0} = \tau^*a_{i0}$; in particular, $b_{i_0,0} = x^{\tilde{\beta}} y$, for some $\tilde{\beta}$. All the exponents can be
computed explicitly from the $\la_i$ and the original $r_{ij}$, $\beta$, but we will not need the explicit formulas.
Let $x^{\tilde{\gamma}}$ denote a generator of $\tau^*\mathcal{I}$. Then \eqref{eq:prelim} becomes
\begin{equation}\label{eq:I.1sigma}
\begin{aligned}
\sigma_1 & = x^{\tilde{\alpha}},\\
\sigma_i & = g_i(x) + x^{\tilde{\delta} + \tilde{\gamma}} \frac{\tau^*T_i}{x^{\tilde{\gamma}}},
\end{aligned}
\end{equation}
for suitable $\tilde{\alpha}$, $\tilde{\delta}$. We now consider three further subcases depending on which 
monomial generator of $\mathcal{I}$ pulls back to a generator of $\tau^*\mathcal{I}$.

\medskip\noindent
\emph{Subcase I.1.1. $\tau^*(v^d)$ generates $\tau^*\mathcal{I}$.} Then
\[
\frac{\tau^*T_2}{x^{\tilde{\gamma}}} =  (\zeta +z)^d \overline{T}_2 
+ \frac{1}{x^{\tga}}\left(\sum_{j=1}^{d-2} x^{ \tilde{r}_{2j}} (\zeta +z)^j \tb_{2j}  + b_{20}\right),
\]
where $\zeta \neq 0$, $\overline{T}_2$ is a unit, and $\tb_{2j},\, b_{20}$ are independent of $z$. It follows
that $d(b) < d(a)$, so that $\rho(b)<\rho(a)$.

\medskip\noindent
\emph{Subcase I.1.2. $\tau^*(v^d)$ does not generate $\tau^*\mathcal{I}$, but $\tau^*(u^{r_{ij}}v^{j})$ 
generates $\tau^*\mathcal{I}$, for some $(i,j)$.} Let $(i_1,j_1)$ denote such $(i,j)$ with maximal $j$. Then
\[
\frac{\tau^{\ast}T_{i_1}}{x^{\tilde{\gamma}}} = (\zeta +z)^{j_1} \tb_{i_1,j_1} 
+ \frac{1}{x^{\tga}}\left(\sum_{j=1}^{j_1-1} x^{\tilde{r}_{i_1,j}} (\zeta +z)^j \tb_{i_1,j}  + b_{i_1,0} \right) + x R(x,y,z).
\]
Since the $\tb_{i_1,j}$ and $b_{i_1,0}$ are independent of $z$, $d(b) \leq j_1$, so that $\rho(b)<\rho(a)$.

\medskip\noindent
\emph{Subcase I.1.3. Neither $\tau^*(v^d)$ nor any $\tau^*(u^{r_{ij}}v^{j})$ generates $\tau^*\mathcal{I}$.}
Then $\tau^*(u^{\beta})$ generates $\tau^{\ast}\mathcal{I}$. In this case, 
\[
\tau^*T_{i_0}  = x^{\tilde{\gamma}} \left( y + x R(x,y,z)\right),
\]
so that $\rho(b)= 0 <\rho(a)$.

\medskip\noindent
\emph{Subcase I.2. $b$ is a 2-point of $\widetilde{D}$:} There are adapted coordinates $(\pmb{x},y)=(x_1,x_2,y)$ at $b$
such that $\widetilde{D} = (x_1x_2=0)$ and $\tau$ is given by
\[
u = \pmb{x}^{\pmb{\lambda}_1}, \quad v = \pmb{x}^{\pmb{\lambda}_2}, \quad w=y,
\]
where $\pmb{\lambda}_1,\pmb{\lambda}_2$ are $\mathbb{Q}$-linearly independent. By $\eqref{eq:weier}$, 
$\eqref{eq:norm1}$,
\begin{equation}\label{eq:I.2}
\tau^{\ast}T_i = \pmb{x}^{\pmb{\tilde{r}}_d} \overline{T}_{i} + \sum_{j=1}^{d-1}  \pmb{x}^{\pmb{\tilde{r}}_{ij}} \tb_{ij} + b_{i0},
\end{equation}
for suitable $\pmb{\tilde{r}}_d$, $\pmb{\tilde{r}}_{i,j}$, where $\tb_{ij} = \tau^*\ta_{ij}$, 
$\overline{T}_i = \tau^*\widetilde{T}_i$ and 
$b_{i0}= \tau^*a_{i0}$; in particular, $b_{i_0,0} = \pmb{x}^{\pmb{\tilde{\beta}}} y$, for some $\pmb{\tilde{\beta}}$.  Again
all exponents can be computed explicitly from the $\pmb{\lambda}_i$ and the original exponents. In particular, since 
$\pmb{\lambda}_1$, $\pmb{\lambda}_2$ are linearly independent, the multi-indices $\pmb{\tilde{r}}_d$, $\pmb{\tilde{r}}_{i,j}$ (for fixed $i$) and $\pmb{\tilde{\beta}}$ are all distinct. Let $\pmb{x}^{\pmb{\tilde{\gamma}}}$ be a generator of 
$\tau^*\mathcal{I}$. Then \eqref{eq:prelim} becomes
\begin{equation}
\begin{aligned}
\sigma_1 & = \pmb{x}^{\pmb{\tilde{\alpha}}}, \\
\sigma_i & = g_i(\pmb{x}) + \pmb{x}^{\pmb{\tilde{\delta}} + \pmb{\tilde{\gamma}}} 
\frac{\tau^*T_i}{\pmb{x}^{\pmb{\tilde{\gamma}}}},
\end{aligned}
\label{eq:LocalI2prelim}
\end{equation}
for some $\pmb{\tilde{\alpha}}$, $\pmb{\tilde{\delta}}$. Note that $\pmb{x}^{\pmb{\tilde{\alpha}}}$ and
$\pmb{x}^{\pmb{\tilde{\delta}} + \pmb{\tilde{\gamma}}}$ are supported in $\widetilde{E}$. We again consider
three further subcases depending on the generator of $\tau^*\mathcal{I}$.

\medskip\noindent
\emph{Subcase I.2.1. $\tau^*(v^d)$ generates $\tau^*\mathcal{I}$.} Then
\[
\tau^*T_2  = \pmb{x}^{\pmb{\tilde{\gamma}}} \left( \overline{T}_2 + R(\pmb{x},y) \right),
\]
where $R(0,y) = 0$. Moreover, $\pmb{\tilde{\delta}} + \pmb{\tilde{\gamma}}=  \delta \pmb{\lambda}_1 + d \pmb{\lambda}_2$
and $\pmb{\tilde{\alpha}} = \alpha \pmb{\lambda}_1$; therefore, $\pmb{\tilde{\delta}} + \pmb{\tilde{\gamma}}$ and 
$\pmb{\tilde{\alpha}}$ are linearly independent. By \eqref{eq:LocalI2prelim}, $\rho(b) = 0 < \rho(a)$.

\medskip\noindent
\emph{Subcase I.2.2.  $\tau^*(v^d)$ does not generate $\tau^*\mathcal{I}$, but $\tau^*(u^{r_{ij}}v^{j})$ 
generates $\tau^*\mathcal{I}$, for some $(i,j)$.} Let $(i_1,j_1)$ denote such $(i,j)$ with maximal $j$. Then
\[
\tau^*T_{i_1}  = \pmb{x}^{\pmb{\tilde{\gamma}}} \left( b_{i_1,j_1} + R(\pmb{x},y) \right),
\]
where $R(0,y) = 0$. Moreover, $\pmb{\tilde{\delta}} + \pmb{\tilde{\gamma}}=  (\delta+r_{i_1,j_1}) \pmb{\lambda}_1 +j_1  \pmb{\lambda}_2$ and $\pmb{\tilde{\alpha}} = \alpha \pmb{\lambda}_1$; therefore, 
$\pmb{\tilde{\delta}} + \pmb{\tilde{\gamma}}$ and $\pmb{\tilde{\alpha}}$ are linearly independent. By 
\eqref{eq:LocalI2prelim}, $\rho(b) = 0 < \rho(a)$.

\medskip\noindent
\emph{Subcase I.2.3. $\tau^*(u^{\beta})$ is the only generator of $\tau^*\mathcal{I}$.} Then
\[
\tau^*T_{i_0}  = \pmb{x}^{\pmb{\tilde{\gamma}}} \left( y + R(\pmb{x},y) \right),
\]
where $R(0,y) = 0$, so that $d(b) = 1< d(a)$ and $\rho(b)<\rho(a)$.

\medskip\noindent
\emph{Case II. $a$ is a non-generic 1-point.} Then $D = (uvw = 0)$. We consider three cases depending on 
whether $b$ is a 1-, 2- or 3-point of $\widetilde{D}$.

\medskip\noindent 
\emph{Subcase II.1. $b$ is a 1-point of $\widetilde{D}$.} We follow the steps of Subcase I.1. There are adapted 
coordinates $(x,y,z)$ at $b$, where $\widetilde{D} = (x=0)$ and $\tau$ is given by
\[
u = x^{\lambda_1},\quad w = x^{\lambda_2}(\eta+y) ,\quad v=x^{\lambda_3}(\zeta+z).
\]
where $\eta \neq 0$, $\zeta \neq 0$. Again $\widetilde{E} = \widetilde{D}$ at $b$, since $\supp \widetilde{E} = 
V(\tau^*(u))$. By $\eqref{eq:weier}$ and $\eqref{eq:norm1}$, we again have formulas \eqref{eq:I.1},
for suitable $\tilde{r}_d$, $\tilde{r}_{ij}$, where now $\tb_{ij} = \tb_{ij}(x,y) = \tau^*\ta_{ij}$ times a unit, 
$\overline{T}_i = \tau^*\widetilde{T}_i$ and 
$b_{i0} = \tau^*a_{i0}$; in particular, $b_{i_0,0} = x^{\tilde{\beta}} (\eta +y)$, for some $\tilde{\beta}$. 
Let $x^{\tilde{\gamma}}$ be a generator of $\tau^*\mathcal{I}$. Again $\s$ has the form \eqref{eq:I.1sigma},
for suitable $\tilde{\alpha}$, $\tilde{\delta}$, and we consider three further subcases depending on the
generator of $\tau^*\mathcal{I}$.

\medskip\noindent
\emph{Subcase II.1.1. $\tau^*(v^d)$ generates $\tau^*\mathcal{I}$.} In this case, we can repeat
Subcase I.1.1 word-for-word. 

\medskip\noindent
\emph{Subcase II.1.2. $\tau^*(v^d)$ does not generate $\tau^*\mathcal{I}$, but some $\tau^*(u^{r_{ij}}w^{s_{ij}}v^{j})$ 
generates $\tau^*\mathcal{I}$.} Then we can repeat Subcase I.1.2 word-for-word.

\medskip\noindent
\emph{Subcase II.1.3. $\tau^*(u^{\beta}w)$ is the only generator of $\tau^*\mathcal{I}$.} Then
\[
\tau^*T_{i_0}  = x^{\tilde{\gamma}} \left((\eta + y) + x R(x,y,z)\right),
\]
so that $\rho(b)= 0 <\rho(a)$.

\medskip\noindent
\emph{Subcase II.2. $b$ is a $2$-point of $\widetilde{D}$.} Then there are two possibilities:
\begin{enumerate}
\item[(a)] There are adapted coordinates $(\pmb{x},z) = (x_1,x_2,z)$ 
at $b$ such that $\widetilde{D} = (x_1x_2=0)$ and $\tau$ is given by 
\begin{equation*}
u = \pmb{x}^{\pmb{\lambda}_1},\quad w= \pmb{x}^{\pmb{\lambda}_2},\quad v = \pmb{x}^{\pmb{\lambda}_3}(\zeta+z),
\end{equation*}
where $\pmb{\lambda}_1, \pmb{\lambda}_2$ are $\mathbb{Q}$-linearly independent and $\zeta\neq 0$.
\smallskip
\item[(b)] There are adapted coordinates $(\pmb{x},y) = (x_1,x_2,y)$
such that $\widetilde{D} = (x_1x_2=0)$ and $\tau$ is given by 
\begin{equation*}
u = \pmb{x}^{\pmb{\lambda}_1}, \quad w= \pmb{x}^{\pmb{\lambda}_2}(\eta+y),\quad v = \pmb{x}^{\pmb{\lambda}_3},
\end{equation*}
where $\pmb{\lambda}_1, \pmb{\lambda}_3$ are linearly independent, 
$\pmb{\lambda}_1, \pmb{\lambda}_2$ are linearly dependent, and $\eta\neq 0$. 
\end{enumerate}
We have to consider both (a) and (b).

\medskip\noindent
\emph{Subcase II.2(a).} By $\eqref{eq:weier}$, $\eqref{eq:norm1}$,
\begin{equation}\label{eq:II.2(a)}
\tau^{\ast}T_i = \pmb{x}^{\pmb{\tilde{r}}_d} (\zeta+z)^d \overline{T}_i 
+ \sum_{j=1}^{d-1}  \pmb{x}^{\pmb{\tilde{r}}_{ij}}  (\zeta+z)^j \tb_{ij} + b_{i0},
\end{equation}
for suitable $\pmb{\tilde{r}}_d$, $\pmb{\tilde{r}}_{ij}$, where $\tb_{ij} = \tb_{ij}(\bx) = \tau^{\ast}\ta_{ij}$,
$\overline{T}_2 = \tau^*\widetilde{T}_i$, and $b_{i0} = \tau^*a_{i0}$; in particular, $b_{i_0,0} = \pmb{x}^{\pmb{\tilde{\beta}}}$, for some $\pmb{\tilde{\beta}}$. Let $\pmb{x}^{\pmb{\tilde{\gamma}}}$ be a generator of $\tau^{\ast}\mathcal{I}$. Then
$\s$ takes the form \eqref{eq:LocalI2prelim}, for suitable $\pmb{\tilde{\alpha}},\, \pmb{\tilde{\delta}}$. 
Note that $\pmb{x}^{\pmb{\tilde{\alpha}}},  \pmb{x}^{\pmb{\tilde{\delta}}}$ and $\pmb{x}^{\pmb{\tilde{\gamma}}}$ are supported
in $\widetilde{E}$. As before, we consider three further subcases depending on which monomial generator of 
$\mathcal{I}$ pulls back to a generator of $\tau^{\ast}\mathcal{I}$.

\medskip\noindent
\emph{Subcase II.2.1(a). $\tau^*(v^d)$ generates $\tau^*\mathcal{I}$.} This subcase is similar to I.1.1 and II.1.1. We have
\[
\frac{\tau^{\ast}T_2}{\pmb{x}^{\pmb{\tilde{\gamma}}}} = (z+\zeta)^d \overline{T}_2 + \frac{1}{\pmb{x}^{\pmb{\tilde{\gamma}}}}\left(\sum_{j=1}^{d-2} \pmb{x}^{ \pmb{\widetilde{r}}_{2j}} (z+\zeta)^j \tb_{2j}  + b_{20}\right),
\]
where $\zeta \neq 0$, $\overline{T}_2$ is a unit and $\tb_{2j},\, b_{20}$ are independent of $z$. It follows that
$d(b) <  d(a)$, so $\rho(b)<\rho(a)$.

\medskip\noindent
\emph{Subcase II.2.2(a). $\tau^*(v^d)$ does not generate $\tau^*\mathcal{I}$, but 
$\tau^*(u^{r_{i_1,j_1}}w^{s_{i_1,j_1}}v^{j_1})$ (with maximal $j_1$) generates $\tau^*\mathcal{I}$.} As in I.1.2 and II.1.2,
\[
\frac{\tau^{\ast}T_{i_1}}{\pmb{x}^{\pmb{\tilde{\gamma}}}} = (z+\zeta)^{j_1} \tb_{i_1,j_1} 
+ \frac{1}{\pmb{x}^{\pmb{\tilde{\gamma}}}}\left(\sum_{j=1}^{j_1-1} \pmb{x}^{\pmb{\widetilde{r}}_{i_1,j}} (z+\zeta)^j \tb_{i_1,j}  + b_{i_1,0} \right) + \pmb{x} R(\pmb{x},z),
\]
where the $\tb_{i_1,j}$ and $b_{i_1,0}$ are independent of $z$. Therefore,
$d(b) \leq j_1$, so that $\rho(b)<\rho(a)$.

\medskip\noindent
\emph{Subcase II.2.3(a). $\tau^*(u^{\beta}w)$ is the only generator of $\tau^*\mathcal{I}$.} This is similar to
I.2.1 or I.2.2. We have
\[
\tau^{\ast}T_{i_0}  = \pmb{x}^{\pmb{\tilde{\gamma}}} \left( 1 +  R(\pmb{x},z)\right),
\]
where $R(0,z) = 0$. Moreover, $\pmb{\tilde{\delta}} + \pmb{\tilde{\gamma}}
=  (\delta+\beta) \pmb{\lambda}_1 +  \pmb{\lambda}_2$ and $\pmb{\tilde{\alpha}} = \alpha \pmb{\lambda}_1$, so 
that $\pmb{\tilde{\delta}} + \pmb{\tilde{\gamma}}$ and $\pmb{\tilde{\alpha}}$ are linearly independent. Since 
$\pmb{x}^{\pmb{\tilde{\gamma}}}$ is supported in $\widetilde{E}$, $\rho(b) = 0 < \rho(a)$.

\medskip\noindent
\emph{Subcase II.2(b).} We follow the steps of Subcase I.2. By $\eqref{eq:weier}$, $\eqref{eq:norm1}$,
we have formulas \eqref{eq:I.2},
for suitable $\pmb{\tilde{r}}_d$, $\pmb{\tilde{r}}_{ij}$, where $\tb_{ij} = \tau^{\ast}\ta_{ij}$ times a unit, 
$\overline{T}_i = \tau^{\ast}\widetilde{T}_i$ and $b_{i0}= \tau^*a_{i0}$; in particular, 
$b_{i_0,0} = \pmb{x}^{\pmb{\tilde{\beta}}}(\eta + y)$ for suitable $\tilde{\beta}$. For fixed $i$, since $\pmb{\lambda}_1,\,\pmb{\lambda}_3$ are linearly independent, the exponents $\pmb{\tilde{r}}_d$, $\pmb{\tilde{r}}_{ij}$ and 
$\pmb{\tilde{\beta}}$ are distinct. Let $\pmb{x}^{\pmb{\tilde{\gamma}}}$ denote a generator of $\tau^*\mathcal{I}$.
Then \eqref{eq:prelim} takes the form \eqref{eq:LocalI2prelim},
for suitable $\pmb{\tilde{\alpha}},\, \pmb{\tilde{\delta}}$. Note that $\pmb{x}^{\pmb{\tilde{\alpha}}}$ and 
$\pmb{x}^{\pmb{\tilde{\delta}} + \pmb{\tilde{\gamma}}}$ are supported in $\widetilde{E}$. We consider three further 
subcases II.2.1(b), II.2.2(b) and II.2.3(b) analogous to I.2.1, I.2.2 and I.2.3, respectively, in each of which we can
argue word-for-word as in the latter.

\medskip\noindent
\emph{Subcase II.3. $b$ is a $3$-point of $\widetilde{D}$.} We follow the steps of Subcase I.2. 
There are adapted coordinates $\pmb{x}=(x_1,x_2,x_3)$ at $b$, such that $\widetilde{D} = (x_1x_2x_3=0)$ and
$\tau$ is given by
\[
u = \pmb{x}^{\pmb{\lambda}_1},\quad w=\pmb{x}^{\pmb{\lambda}_2} ,\quad v = \pmb{x}^{\pmb{\lambda}_3},
\]
where $\pmb{\lambda}_1,\pmb{\lambda}_2,\pmb{\lambda}_3$ are $\mathbb{Q}$-linearly independent. By
$\eqref{eq:weier}$, $\eqref{eq:norm1}$, we again have formulas \eqref{eq:I.2} (here of course
$\pmb{x}=(x_1,x_2,x_3)$),
for suitable $\pmb{\tilde{r}}_d$, $\pmb{\tilde{r}}_{ij}$, where $\tb_{ij} = \tau^{\ast}\ta_{ij}$, 
$\overline{T}_i = \tau^{\ast}\widetilde{T}_i$ and $b_{i0} = \tau^*a_{i0}$; in particular, 
$b_{i_0,0} = \pmb{x}^{\pmb{\tilde{\beta}}}$ for suitable $\pmb{\tilde{\beta}}$. For fixed $i$, since 
$\pmb{\lambda}_1,\, \pmb{\lambda}_2,\, \pmb{\lambda}_3$ are $\mathbb{Q}$-linearly independent, 
the expoments $\pmb{\tilde{r}}_d$, $\pmb{\tilde{r}}_{ij}$ and $\pmb{\tilde{\beta}}$ are distinct. Let $\pmb{x}^{\pmb{\tilde{\gamma}}}$ be a generator of $\tau^*\mathcal{I}$. Then equation \eqref{eq:prelim} 
takes the form \eqref{eq:LocalI2prelim},
for suitable $\pmb{\tilde{\alpha}},\, \pmb{\tilde{\delta}}$, where $\pmb{x}^{\pmb{\tilde{\alpha}}}$,
$\pmb{x}^{\pmb{\tilde{\delta}} + \pmb{\tilde{\gamma}}}$ are supported in $\widetilde{E}$. We consider three subcases
as before.

\medskip\noindent
\emph{Subcase II.3.1. $\tau^{\ast}(v^d) = \pmb{x}^{d \pmb{\lambda}_2}$ generates $\tau^{\ast}\mathcal{I}$.} 
Then
\[
\tau^{\ast}T_2  = \pmb{x}^{\pmb{\tilde{\gamma}}} \left( \overline{T}_2 + R(\pmb{x}) \right),
\]
where $R(0) = 0$. Moreover, $\pmb{\tilde{\delta}} + \pmb{\tilde{\gamma}}=  \delta \pmb{\lambda}_1 + d \pmb{\lambda}_3$ 
and $\pmb{\tilde{\alpha}} = \alpha \pmb{\lambda}_1$, so that $\pmb{\tilde{\delta}} + \pmb{\tilde{\gamma}}$,
$\pmb{\tilde{\alpha}}$ are linearly independent. By \eqref{eq:LocalI2prelim}, $\rho(b) = 0 < \rho(a)$.

\medskip\noindent
\emph{Subcase II.3.2. $\tau^{\ast}(v^d)$ does not generate $\tau^{\ast}\mathcal{I}$, but 
$\tau^{\ast}(u^{r_{i_1,j_1}}w^{s_{i_1,j_1}}v^{j_1})$ (with maximal $j_1$) generates $\tau^{\ast}\mathcal{I}$.} Then
\[
\tau^{\ast}T_{i_1}  = \pmb{x}^{\pmb{\tilde{\gamma}}} \left( b_{i_1,j_1} + R(\pmb{x}) \right),
\]
where $R(0) = 0$. Moreover, $\pmb{\tilde{\delta}} + \pmb{\tilde{\gamma}}=  (\delta+r_{i_1,j_1}) \pmb{\lambda}_1 +s_{i_1,j_1} \pmb{\lambda}_2 + j_1  \pmb{\lambda}_3$ and $\pmb{\tilde{\alpha}} = \alpha \pmb{\lambda}_1$, so that 
$\pmb{\tilde{\delta}} + \pmb{\tilde{\gamma}}$, $\pmb{\tilde{\alpha}}$ are linearly independent. By
\eqref{eq:LocalI2prelim}, $\rho(b) = 0 < \rho(a)$.

\medskip\noindent
\emph{Subcase II.3.3. $\tau^{\ast}(u^{\beta})$ is the only generator of $\tau^{\ast}\mathcal{I}$.} Then
\[
\tau^{\ast}T_{i_0}  = \pmb{x}^{\pmb{\tilde{\gamma}}} \left( 1 + R(\pmb{x}) \right),
\]
where $R(0) = 0$. Moreover, 
$\pmb{\tilde{\delta}} + \pmb{\tilde{\gamma}}=  (\delta+\beta) \pmb{\lambda}_1 + \pmb{\lambda}_2 $ and $\pmb{\tilde{\alpha}} = \alpha \pmb{\lambda}_1$, so that $\pmb{\tilde{\delta}} + \pmb{\tilde{\gamma}}$, $\pmb{\tilde{\alpha}}$ are linearly independent. By \eqref{eq:LocalI2prelim}, $\rho(b) = 0 < \rho(a)$.

\medskip\noindent
\emph{Case III. $a$ is a 2-point.} Then $D:= (u_1u_1v = 0)$. Since $\tau$ is combinatorial in respect to $D$, 
we consider three cases depending on whether $b$ is a 1-, 2- or 3-point of $\widetilde{D}$:

\medskip\noindent
\emph{Subcase III.1. $b$ is a 1-point of $\widetilde{D}$.} We follow the steps of Subcase I.1. There
are adapted coordinates $(x,y,z)$ at $b$, where $\widetilde{D} = (x=0)$ and $\tau$ is given by
\[
u_1 = x^{\lambda_1}(\eta+y)^{\alpha_2},\quad u_2 = x^{\lambda_2}(\eta+y)^{-\alpha_1} ,\quad v=x^{\lambda_3}(\zeta+z),
\]
where $\eta \neq 0$, $\zeta \neq 0$. Clearly, $\widetilde{E} = \widetilde{D}$ at $b$, since 
$\supp \widetilde{E} = V(\tau^{\ast}\pmb{u}^{\pmb{\alpha}})$. By $\eqref{eq:weier}$, $\eqref{eq:norm1}$,
we have formulas \eqref{eq:I.1},
for suitable $\tilde{r}_d$, $\tilde{r}_{ij}$, where $\tb_{ij} = \tb_{ij}(x,y) = \tau^{\ast}\ta_{ij}$ times a unit,
$\overline{T}_i = \tau^{\ast}\widetilde{T}_i$, and $b_{i0} = \tau^*a_{i0}$. Note that
$\tau^{\ast}(\pmb{u}^{\pmb{\delta}}a_{i_0,0}) = x^{\tilde{\beta}+\tilde{\delta}} (\eta +y)^{\tilde{\epsilon}}$ for appropriate 
$\tilde{\delta}$, $\tilde{\beta}$ and $\tilde{\ep}$, where $\tilde{\ep}\neq 0$ since 
$\tilde{\ep} = (\delta_2+\beta_2) \alpha_1 - (\delta_1 + \beta_1) \alpha_2$ and $\pmb{\alpha}$, $\pmb{\delta}+\pmb{\beta}$
are $\mathbb{Q}$-linearly independent. Let $x^{\tilde{\gamma}}$ denote a generator of $\tau^*\mathcal{I}$. Then
$\s$ has the form \eqref{eq:I.1sigma}, and we consider 
three further subcases III.1.1, III.1.2 and III.1.3 as before, depending on which monomial generator of $\cI$ 
pulls back to a generator of $\tau^{\ast}\mathcal{I}$. 

Subcases III.1.1 and III.1.2 parallel I.1.1 and I.1.2 (respectively), word-for-word.

\medskip\noindent
\emph{Subcase III.1.3. $\tau^{\ast}\pmb{u}^{\pmb{\beta}}$ is the only generator of $\tau^{\ast}\mathcal{I}$.} Then
\[
\tau^{\ast}(\pmb{u}^{\pmb{\delta}}T_{i_0})  = x^{\tilde{\gamma}+\tilde{\delta}} \left((\eta + y)^{\tilde{\ep}} + x R(x,y,z)\right);
\]
therefore, $\rho(b)= 0 <\rho(a)$.

\medskip\noindent
\emph{Subcase III.2. $b$ is a 2-point of $\widetilde{D}$.} Then there are two possibilities:
\begin{enumerate}
\item[(a)] There are adapted coordinates $(\pmb{x},z) = (x_1,x_2,z)$ at $b$ such that $\widetilde{D} = (x_1x_2=0)$
and $\tau$ is given by
\begin{equation*}
u_1 = \pmb{x}^{\pmb{\lambda}_1},\quad u_2= \pmb{x}^{\pmb{\lambda}_2},\quad v = \pmb{x}^{\pmb{\lambda}_3}(\zeta+z),
\end{equation*}
where $\pmb{\lambda}_1, \pmb{\lambda}_2$ are linearly independent and $\zeta\neq 0$. In this case, $\widetilde{E} = \widetilde{D}$.
\smallskip
\item[(b)] There are adapted coordinates $(\pmb{x},y) = (x_1,x_2,y)$ at $b$ such that $\widetilde{D} = (x_1x_2=0)$
and $\tau$ is given by
\begin{equation*}
u_1 = \pmb{x}^{\pmb{\lambda}_1}(\eta+y)^{\alpha_2}, \quad u_2= \pmb{x}^{\pmb{\lambda}_2}(\eta+y)^{-\alpha_1},\quad v = \pmb{x}^{\pmb{\lambda}_3},
\end{equation*}
where $\pmb{\lambda}_1,\, \pmb{\lambda}_3$ are linearly independent, $\pmb{\lambda}_1,\, \pmb{\lambda}_2$ are 
linearly dependent, and $\eta\neq 0$.
\end{enumerate}
We again consider both (a) and (b).

\medskip\noindent
\emph{Subcase III.2(a).} By equations $\eqref{eq:weier}$, $\eqref{eq:norm1}$, we have formulas \eqref{eq:II.2(a)},
for suitable $\pmb{\tilde{r}}_d$, $\pmb{\tilde{r}}_{ij}$, where $\tb_{ij} = \tb_{ij}(\bx) = \tau^{\ast}\ta_{i,j}$,
$\overline{T}_i = \tau^{\ast}\widetilde{T}_i$ and $b_{i0} = \tau^*a_{i0}$; in particular, 
$b_{i_0,0} = \pmb{x}^{\pmb{\tilde{\beta}}}$ for suitable $\pmb{\tilde{\beta}}$. Let $\pmb{x}^{\pmb{\tilde{\gamma}}}$ be a generator of $\tau^{\ast}\mathcal{I}$. Then \eqref{eq:prelim} takes the form \eqref{eq:LocalI2prelim},
for suitable $\pmb{\tilde{\alpha}}$, $\pmb{\tilde{\delta}}$. We again consider three further subcases III.2.1(a),
III.2.2(a) and III.2.3(a), depending on which 
monomial generator of $\mathcal{I}$ pulls back to a generator of $\tau^{\ast}\mathcal{I}$. In each subcase,
we can follow the corresponding subcase of II.2(a) essentially word-for-word.

\medskip\noindent
\emph{Subcase III.2(b).} We follow the steps of I.2. By $\eqref{eq:weier}$, $\eqref{eq:norm1}$,
we have formulas \eqref{eq:I.2}, for suitable
$\pmb{\tilde{r}}_d$, $\pmb{\tilde{r}}_{ij}$, where $\tb_{ij} = \tau^{\ast}\ta_{ij}$ times a unit, 
$\overline{T}_i = \tau^{\ast}\widetilde{T}_i$ and $b_{i0} = \tau^*a_{i0}$. Note that 
$\tau^{\ast}(\pmb{u}^{\pmb{\delta}}a_{i_0,0}) = \pmb{x}^{\pmb{\tilde{\delta}}+\pmb{\tilde{\beta}}}(\eta + y)^{\tilde{\ep}}$,
suitable $\pmb{\tilde{\delta}}$, $\pmb{\tilde{\beta}}$ and $\tilde{\ep}$, where $\tilde{\ep} \neq 0$ since $\tilde{\ep} = \alpha_2(\beta_1+\delta_1)-\alpha_1(\delta_2+\delta_1)$ and $\pmb{\alpha}$, $\pmb{\delta}+\pmb{\beta}$ are
linearly independent. Moreover, for each fixed $i$, since $\pmb{\lambda}_1$, $\pmb{\lambda}_3$ are linearly independent, 
the exponents $\pmb{\tilde{r}}_d$, $\pmb{\tilde{r}}_{ij}$ and $\pmb{\tilde{\beta}}$ are distinct. Let $\pmb{x}^{\pmb{\tilde{\gamma}}}$ be a generator of $\tau^*\mathcal{I}$. Then \eqref{eq:prelim} takes the form
\eqref{eq:LocalI2prelim}, for suitable
$\pmb{\tilde{\alpha}}$, $\pmb{\tilde{\delta}}$,  and $\pmb{x}^{\pmb{\tilde{\alpha}}}$,
$\pmb{x}^{\pmb{\tilde{\delta}} + \pmb{\tilde{\gamma}}}$ are supported 
in $\widetilde{E}$. As usual, we consider three subcases. The first two, III.2.1(b) and III.2.2(b), parallel
I.2.1 and I.2.2 (respectively).

\medskip\noindent
\emph{Subcase III.2.3(b). $\tau^{\ast}(\pmb{u}^{\pmb{\beta}})$ is the only generator of $\tau^{\ast}\mathcal{I}$.} Then
\[
\tau^{\ast}\pmb{u}^{\pmb{\delta}}T_{i_0}  = \pmb{x}^{\pmb{\tilde{\delta}} + \pmb{\tilde{\gamma}}} \left( (\eta+ y)^{\tilde{\ep}} + R(\pmb{x},y) \right),
\]
where $R(0,y) = 0$. It follows that $\rho(b) = 0 <\rho(a)$.

\medskip\noindent
\emph{Subcase III.3. $b$ is a 3-point of $\widetilde{D}$.} We can again follow the steps of Subcase I.2.
There are adapted coordinates $\pmb{x}=(x_1,x_2,x_3)$ at $b$, such that $\widetilde{D} = (x_1x_2x_3=0)$ and $\tau$ 
is given by:
\[
u_1 = \pmb{x}^{\pmb{\lambda}_1},\quad u_2=\pmb{x}^{\pmb{\lambda}_2} ,\quad v = \pmb{x}^{\pmb{\lambda}_3},
\]
where $\pmb{\lambda}_1,\, \pmb{\lambda}_2,\, \pmb{\lambda}_3$ are linearly independent. By
$\eqref{eq:weier}$, $\eqref{eq:norm1}$, we have formulas \eqref{eq:I.2}, for suitable
$\pmb{\tilde{r}}_d$, $\pmb{\tilde{r}}_{ij}$, where $\tb_{ij} = \tau^{\ast}\ta_{ij}$, 
$\overline{T}_i = \tau^{\ast}\tilde{T}_i$ and $b_{i0} = \tau^*a_{i0}$; in particular, $b_{i_0,0} = \pmb{x}^{\pmb{\tilde{\beta}}}$ for suitable $\pmb{\tilde{\beta}}$. For each fixed $i$, since $\pmb{\lambda}_1$, $\pmb{\lambda}_2$, $\pmb{\lambda}_3$ are 
linearly independent, the exponents $\pmb{\tilde{r}}_d$, $\pmb{\tilde{r}}_{ij}$ and $\pmb{\tilde{\beta}}$ are distinct. 
Let $\pmb{x}^{\pmb{\tilde{\gamma}}}$ be a generator of $\tau^*\mathcal{I}$. Then \eqref{eq:prelim} takes the
form \eqref{eq:LocalI2prelim}, for suitable
$\pmb{\tilde{\alpha}}$, $\pmb{\tilde{\delta}}$; moreover, $\pmb{x}^{\pmb{\tilde{\alpha}}}$, 
$\pmb{x}^{\pmb{\tilde{\delta}} + \pmb{\tilde{\gamma}}}$ are supported in $\widetilde{E}$. We consider three subcases,
as usual.

\medskip\noindent
\emph{Subcase III.3.1. $\tau^{\ast}(v^d)$ generates $\tau^{\ast}\mathcal{I}$.} Then
\[
\tau^{\ast}T_2  = \pmb{x}^{\pmb{\tilde{\gamma}}} \left( \bar{T}_2 + R(\pmb{x}) \right),
\]
where $R(0) = 0$. Moreover, $\pmb{\tilde{\delta}} + \pmb{\tilde{\gamma}}=  \delta_1 \pmb{\lambda}_1 +\delta_2 \pmb{\lambda}_2 + d \pmb{\lambda}_3$ and $\pmb{\tilde{\alpha}} = \alpha_1 \pmb{\lambda}_1+\alpha_2 \pmb{\lambda}_2$;
therefore, $\pmb{\tilde{\delta}} + \pmb{\tilde{\gamma}}$ and $\pmb{\tilde{\alpha}}$ are linearly independent, and it follows
that $\rho(b) = 0 < \rho(a)$.

\medskip\noindent
\emph{Subcase III.3.2. $\tau^{\ast}(v^d)$ does not generate $\tau^{\ast}\mathcal{I}$, but 
$\tau^{\ast}(\pmb{u}^{\pmb{r}_{i_1,j_1}}v^{j_1})$ (with maximal $j_1$) generates $\tau^{\ast}\mathcal{I}$.} Then
\[
\tau^{\ast}T_{i_1}  = \pmb{x}^{\pmb{\tilde{\gamma}}} \left( b_{i_1,j_1} + R(\pmb{x}) \right),
\]
where $R(0) = 0$. Moreover, $\pmb{\tilde{\delta}} + \pmb{\tilde{\gamma}}=  (\delta_1+r_{i_1,j_1,1}) \pmb{\lambda}_1 
+(\delta_2+r_{i_1,j_1,2}) \pmb{\lambda}_2 + j_1  \pmb{\lambda}_3$ and $\pmb{\tilde{\alpha}} = \alpha_1 \pmb{\lambda}_1+\alpha_2 \pmb{\lambda}_2$; therefore, $\pmb{\tilde{\delta}} + \pmb{\tilde{\gamma}}$ and $\pmb{\tilde{\alpha}}$ are linearly independent, and it follows that $\rho(b) = 0 < \rho(a)$.

\medskip\noindent
\emph{Subcase III.3.3. $\tau^{\ast}(\pmb{u}^{\pmb{\beta}})$ is the only generator of $\tau^{\ast}\mathcal{I}$.} Then
\[
\tau^{\ast}\pmb{u}^{\pmb{\delta}}T_{i_0}  = \pmb{x}^{\pmb{\tilde{\delta}}+\pmb{\tilde{\gamma}}} \left( 1 + R(\pmb{x}) \right),
\]
where $R(0) = 0$. Moreover, $\pmb{\tilde{\delta}} + \pmb{\tilde{\gamma}}=  (\delta_1+\beta_1) \pmb{\lambda}_1 + 
(\delta_2+\beta_2)\pmb{\lambda}_2 $ and $\pmb{\tilde{\alpha}} = \alpha_1 \pmb{\lambda}_1+\alpha_2 \pmb{\lambda}_2$. Since $\pmb{\delta}+\pmb{\beta}$, $\pmb{\alpha}$ are linearly independent, it follows that $\pmb{\tilde{\delta}} + \pmb{\tilde{\gamma}}$, $\pmb{\tilde{\alpha}}$ are linearly independent, and finally again, $\rho(b) = 0 < \rho(a)$.
\end{proof}

\bibliographystyle{amsplain}

\begin{thebibliography}{99}

\bibitem{AKMW}
D. Abramovich, K. Karu, K. Matsuki and J. W{\l}odarczyk,
\textit{Torification and factorization of birational maps},
J. Amer. Math. Soc. \textbf{15} (2002), 531--572.

\bibitem{Be1} A. Belotto da Silva,
\textit{Local resolution of singularities in foliated spaces}, preprint (2014), arXiv:1411.5009 [math.CV].

\bibitem{Be2} A. Belotto da Silva,
\textit{Local monomialization of a system of first integrals}, preprint (2014), arXiv:1411.5333 [math.CV].

\bibitem{BMinv}
E. Bierstone and P.D. Milman,
\textit{Canonical desingularization in characteristic zero by
blowing up the maximum strata of a local invariant},
Invent. Math. \textbf{128} (1997), 207--302.

\bibitem{BMfunct}
E. Bierstone and P.D. Milman,
\textit{Functoriality in resolution of singularities},
Publ. R.I.M.S. Kyoto Univ. \textbf{44} (2008), 609--639.

\bibitem{C1}
J. Cheeger,
\textit{On the spectral geometry of spaces with cone-like singularities},
Proc. Nat. Acad. Sci. U.S.A. \textbf{76} (1979), 2103--2106.

\bibitem{C2}
J. Cheeger,
\textit{On the Hodge theory of Riemannian pseudomanifolds},
Geometry of the Laplace operator (Proc. Sympos. Pure Math., Univ. Hawaii, Honolulu,
Hawaii, 1979), pp. 91--146, Proc. Sympos. Pure Math. \textbf{XXXVI}, Amer. Math. Soc.,
Providence, R.I., 1980.

\bibitem{CGM}
J. Cheeger, M. Goresky and R. MacPherson,
\textit{$L^2$-cohomology and intersection homology of singular algebraic varieties},
Seminar on Differential Geometry, pp. 303--340,
Ann. of Math. Stud. \textbf{102}, Princeton Univ. Press, Princeton, N.J., 1982.

\bibitem{Cut1}
S.D. Cutkosky,
\textit{Monomialization of morphisms from 3-folds to surfaces},
Lecture Notes in Math. \textbf{1786}, Springer-Verlag, Berlin, 2002.

\bibitem{Cut3}
S.D. Cutkosky,
\textit{Toroidalization of dominant morphisms of 3-folds},
Mem. Amer. Math. Soc. \textbf{190} (2007).

\bibitem{Cut2}
S.D. Cutkosky,
\textit{A simpler proof of toroidalization of morphisms from 3-folds to surfaces},
Ann. Inst. Fourier (Grenoble) \textbf{63} (2013), 865--922.

\bibitem{Cut4}
S.D. Cutkosky,
\textit{Errata to A simpler proof of toroidalization of morphisms from 3-folds to surfaces},
http://faculty.missouri.edu/~cutkoskys/Torerrata.pdf, University of Missouri, 2015. Web September 8, 2015.

\bibitem{E}
D. Eisenbud, 
Commutative algebra with a view toward algebraic geometry, 
Graduate Texts in Mathematics \textbf{150}, Springer, New York, 1995.

\bibitem{Ha}
R. Hartshorne,
Algebraic geometry, 
Graduate Texts in Mathematics \textbf{52}, Springer, New York, 1977.

\bibitem{Hiro}
H. Hironaka,
\textit{Stratification and flatness}, 
Real and complex singularities (Proc. Ninth Nordic Summer School/NAVF Sympos. Math., 
Oslo, 1976), pp. 199--265, Sijthoff and Noordhoff, Alphen aan den Rijn, 1977.

\bibitem{HP}
W.C. Hsiang and V. Pati,
\textit{$L^2$-cohomology of normal algebraic surfaces. I},
Invent. Math. \textbf{81} (1985), 395--412.



\bibitem{PS}
W. Pardon and M. Stern,
\textit{Pure Hodge structure on the $L_2$-cohomology of varieties with isolated singularities},
J. Reine Angew. Math. \textbf{533} (2001), 55--80.

\bibitem{Pati}
V. Pati,
\textit{The Laplacian on algebraic three folds with isolated singularities},
Proc. Indian Acad. Sci. Math. Sci. \textbf{104} (1994), 435--481.

\bibitem{Rond}
G. Rond,
\textit{Homomorphisms of local algebras in positive characteristic},
J. Algebra \textbf{322} (2009), 4382--4407.

\bibitem{Laura}
L. Taalman,
\textit{The Nash sheaf of a complete resolution},
Manuscripta Math. \textbf{106} (2001), 249--270.

\bibitem{Tognoli}
A. Tognoli,
\textit{Propriet\`a globali degli spazi analitici reali},
Ann. Mat. Pura Appl. (4) \textbf{75} (1967), 143--218.

\bibitem{Y}
B. Youssin,
\textit{Monomial resolutions of singularities},
Abstract, Conference on Geometric Analysis and Singular Spaces, Oberwolfach, June 21-27, 1998.

\end{thebibliography}

\end{document}